\documentclass{elsarticle}

\usepackage{dsfont,amssymb,bm}
\usepackage{amsthm}
\usepackage{color}
\usepackage{mathtools}
\usepackage{braket}
\usepackage[colorlinks=true]{hyperref}

\numberwithin{equation}{section}

\newtheorem{theorem}{Theorem}[section]
\newtheorem{proposition}[theorem]{Proposition}
\newtheorem{lemma}[theorem]{Lemma}
\newtheorem{corollary}[theorem]{Corollary}

\newtheorem{definition}[theorem]{Definition}

\newcommand\R{{\mathbb{R}}}
\newcommand\N{\mathbb{N}}

\newcommand\E{\mathcal{E}}

\newcommand\td{\mathrm{d}}
\newcommand\dx{\mathop{}\!\mathrm{d}x}
\newcommand\dy{\mathop{}\!\mathrm{d}y}
\newcommand\dr{\mathop{}\!\mathrm{d}r}
\newcommand\dmu{\mathop{}\!\mathrm{d}\mu}

\newcommand\ds{\mathop{}\!\mathrm{d}s}
\newcommand\dt{\mathop{}\!\mathrm{d}t}
\newcommand\dtau{\mathop{}\!\mathrm{d}\tau}

\DeclareMathOperator{\loc}{loc}

\DeclareMathOperator{\Rg}{Rg}

\def\1{\raisebox{2pt}{\rm{$\chi$}}}
\DeclareMathOperator*{\sign}{sign}

\newcommand\abs[1]{\lvert#1\rvert}
\newcommand\norm[1]{\lVert#1\rVert}
\newcommand\lnorm[1]{\left\lVert#1\right\rVert}

\definecolor{darkred}{rgb}{0.7,0.1,0.1}

\begin{document}
	
\begin{frontmatter}
	\title{A doubly nonlinear evolution problem involving the fractional $p$-Laplacian}
	
	\author{Timothy A. Collier\corref{mycorrespondingauthor}}
	\ead{timothyc@maths.usyd.edu.au}
	\cortext[mycorrespondingauthor]{Corresponding author}
	
	\author{Daniel Hauer\fnref{ARC}}
	\ead{daniel.hauer@sydney.edu.au}
	
	\address{The University of Sydney, School of Mathematics and Statistics, NSW 2006, Australia}

	\begin{abstract}
		In this article we focus on a doubly nonlinear nonlocal parabolic initial boundary value problem 
		driven by the fractional $p$-Laplacian
		equip\-ped with homogeneous Dirichlet boundary conditions on
		a domain in $\R^{d}$ and composed with a continuous,
		strictly increasing function.
		We establish well-posedness in $L^1$ in the sense of mild solutions, a comparison principle, and for restricted initial data we obtain that mild solutions of the inhom\-ogeneous evolution problem are strong. 
        We obtain $L^{q}-L^{\infty}$ regularity estimates for mild solutions, implying decay estimates and extending the property of strong solutions for more initial data.
        More\-over, we prove local and global H\"older continuity results as well as a comparison principle that yields extinction in finite time of mild solutions to the homogeneous evolution equation.
	\end{abstract}
	
	
\end{frontmatter}

	%
	%
	%
	%

 
\section{Introduction and main results}	

\subsection{Introduction}
\label{sec:introduction}
	Let $\Omega$ be an open set in $\mathbb{R}^{d}$, $d \ge 1$,
    $0 < T <\infty$, $1 < p < \infty$ and $0 < s< 1$. Then our
    main focus in this article is the following initial boundary value
    problem
	\begin{equation}
	\label{eq:1}
	\begin{cases}
    \begin{alignedat}{2}
	u_{t}(t)+(-\Delta_{p})^{s} \varphi(u(t))+f(\cdot,u(t))&=g(\cdot,t)\quad && 
	\text{in $\Omega\times(0,T),$}\\
	u(t)&=0\quad && \text{in $\R^{d}\setminus\Omega\times
		(0,T),$}\\
	u(0)&=u_{0}\quad && \text{on $\Omega$,}
    \end{alignedat}
	\end{cases}
	\end{equation}
	for given $u_{0}\in L^{1}$ and $g \in L^{1}(0,T;L^{1})$, where we abbreviate the Lebesgue space $L^q(\Omega)$ by $L^{q}$, $1 \le q \le \infty$, and impose the following conditions on
	$\varphi$ and $f$:
	\begin{align}
	\label{hyp:1} & \text{$\varphi \in C(\R)$ is a strictly increasing
		function satisfying $\varphi(0)=0$,}
	\end{align}
	and $f:\Omega\times\mathbb{R}\rightarrow\mathbb{R}$ admits the following properties:
	\begin{align}
	\label{hyp:2} \tag{1.3a}
	&\text{$f$ is a \emph{Lipschitz-continuous Carath\'eodory} function;
		that is,}\\ \notag
	& \left\{
	\begin{array}[c]{l}
	\text{for every $u\in \R$, $x\mapsto f(x,u)$ is measurable on
		$\Omega$, and there is}\\ 
	\text{an $\omega > 0$ such that}\\[5pt]
	\hspace{1cm}|f(x,u_{1})-f(x,u_{2})| 
	\le \omega|u_{1}-u_{2}|\qquad \text{for all $u_{1},u_{2} \in \mathbb{R}$,}\\[5pt]
	\text{uniformly for a.e.~$x\in\Omega$,}
	\end{array}
	\right.\\ 
	\label{hyp:3} \tag{1.3b}
	&\text{and $f(x,0) = 0$ for a.e.~$x \in \Omega$.}
	\end{align}
	\stepcounter{equation} 
	The doubly nonlinear nonlocal operator $(-\Delta_{p})^{s}\varphi$
    in the evolution problem~\eqref{eq:1} models a (singular or
    degenerated) nonlocal diffusion and is the composi\-tion of
    the \emph{(variational) Dirichlet fractional $p$-Laplacian}
    (see Section~\ref{subsec:doubly-nonlinear-operator})
	\begin{equation}
	\label{eq:11}
	\langle
	(-\Delta_{p})^{s}u,v\rangle:=\int_{\R^{2d}}\frac{|u(x)-u(y)|^{p-2}(u(x)-u(y))
		(v(x)-v(y))}{|x-y|^{d+ps}}
	\dx\dy
	\end{equation}
	for every $u$, $v\in W^{s,p}_{0}$ where we write $W^{s,p}_{0}$ for the fractional Sobolev space $W^{s,p}_{0}(\Omega)$ (see
        Section~\ref{sec:frac-sobolev-spaces}) and a function
        $\varphi$ satisfying~\eqref{hyp:1}.\medskip

        In this paper, we present well-posedness of~\eqref{eq:1} in the
        sense of \emph{mild} solutions $u : [0,T]\to L^{1}$, that is, $u(t)$ is the limit in $L^1$ of step-functions $u_{n}$ with coefficients solving the corresponding time discretized problem (for a precise definition, we refer to Definition~\ref{def:mild-solution}), and a comparison
        result for such solutions (see Theorem~\ref{thm:1} below). Moreover, we establish global regularity estimates implying an immediate smoothing effect, and finite time of extinction under minimal natural hypotheses
        on the given initial data $u_{0}$ and the forcing term $g$. 
        Further, we give sufficient conditions implying that mild solutions $u$ of~\eqref{eq:1} are \emph{strong distributional solutions} in $L^1$, that is, for a.e. $t\in (0,T)$, $u$ is differentiable at $t$, $\varphi(u(t))\in W^{s,p}_{0}$ and satisfies $(-\Delta_{p})^{s}\varphi(u(t))=g(\cdot,t)-f(\cdot,u(t))-u_{t}(t)$ in $L^{1}$ (see
        Definition~\ref{def:distributional-solution}).
        
        We focus, in particular, on the case $\varphi(u) = |u|^{m-1}u$ (which we denote by $u^m$) for $m > 0$. Then the equation in~\eqref{eq:1} reduces to
        \begin{equation}
        \label{eq:doublynonlinearporouseqn}
        u_{t}(t)+(-\Delta_{p})^{s}u^m(t)+f(\cdot,u(t))=g(\cdot,t) \qquad
        \text{in $\Omega\times(0,T)$.}
        \end{equation}
        For mild solutions of the initial boundary value problem associated to~\eqref{eq:doublynonlinearporouseqn} with $m \ge 1$, we prove global $L^{\ell}-L^{\infty}$ regularity estimates, $1 \le \ell \le m+1$, implying an immed\-iate \emph{smooth\-ing effect} (see Theorem~\ref{thm:LellLinfty}). This also allows us to improve the integrability of the time derivative of solutions as presented in Theorem~\ref{thm:derivativeEstimates}.
    	For this case, we establish in Theorem~\ref{thm:localHolder} local H\"older continuity of the solutions to~\eqref{eq:doublynonlinearporouseqn} provided $p \ge 2$ and $sp \ge d$ and global continuity in Theorem~\ref{thm:globalHolder} for the case $m = 1$ and $1 < p < \infty$. We prove these in Section~\ref{sec:holderRegularity}. In
        Section~\ref{sec:comparisonPrinciple}, we provide a
        comparison principle for the solutions of evolution
        problem~\eqref{eq:1} when the homogeneous boundary data is
        replaced by time-dependent, inhomogeneous data. With
        this tool, we can establish finite time of extinction of the
        solutions to~\eqref{eq:1}. \medskip
	
	Nonlinear integro-differential operators such as the doubly
        nonlinear non\-local operator $(-\Delta_{p})^{s}\varphi$ have
        received recent interest for their role in the mathematical
        analysis of anomalous diffusion, as well as their applications
        in such fields as statistical mechanics, physics, finance,
        fluid dynamics and image processing. We refer the interested
        reader, for example,
        to~\cite{MR1406564,MR2042661,MR2480109,MR3821542,MR2763032}.
        
        In recent years, the first-order evolution problem for the
        fractional $p$-Lapla\-cian has been studied by many authors,
        including
        \cite{MR3394266,MR3491533,MR3456825,MR3781159,MR3950701,MR4114983,vazquez2021fractional,vazquez2021growing}. This problem corresponds to that of~\eqref{eq:doublynonlinearporouseqn} with $m = 1$ so that
         the equation in~\eqref{eq:1} reduces to
        \begin{equation}
        \label{eq:3}
        u_{t}(t)+(-\Delta_{p})^{s}u(t)+f(\cdot,u(t))=g(\cdot,t) \qquad
        \text{in $\Omega\times(0,T)$.}
        \end{equation}
         In particular, Maz\'on et al.~\cite{MR3491533} (see also,
        \cite{MR3456825}) obtained existence and unique\-ness of strong
        solutions to equation~\eqref{eq:3} for $f\equiv g\equiv 0$
        equipped with either homogeneous Dirichlet or Neumann boundary
        conditions. Giacomoni and Tiwari~\cite{MR3781159} obtained
        well-posedness in $L^{\infty}$ of strong solutions to
        equa\-tion~\eqref{eq:3} equipped with homogeneous Dirichlet
        boundary data. Global $L^{q}-L^{\infty}$ regularity
        estimates, $1\le q<\infty$, for solutions to the parabolic
        equation \eqref{eq:3} for $g\equiv 0$ and equipped with
        homogeneous Dirichlet boundary conditions have been obtained in the
        monograph~\cite{CoulHau2016} of the second author.
	
	It is worth noting that the operator
		$(-\Delta_{p})^{s}\varphi$ is the \emph{nonlocal} counterpart of
		the \emph{local} doubly nonlinear operator
		$(-\Delta_{p})\varphi$ (see, for
		example~\cite{MR1218742,MR2135744,CoulHau2016}). In addition, the fractional $p$-Laplacian
		$(-\Delta_{p})^{s}$ is a natural generalizat\-ion of the
		well-known linear \emph{fractional Laplacian}
		$(-\Delta)^{s}$. Hence, an important special case of this doubly nonlinear nonlocal
        evolution problem~\eqref{eq:1} is given by the celebrated
        \emph{fraction\-al porous medium equation}
        \begin{equation}
          \label{eq:7}
           u_{t}(t)+(-\Delta)^{s}u^{m}(t)=g(\cdot,t) \qquad
	\text{in $\Omega\times(0,T)$,}
        \end{equation}
        for $m > 0$. For initial boundary-value problems associated with~\eqref{eq:7},
        V\'{a}zquez and his
        collaborators~\cite{MR2737788,MR2954615,MR3393253} have developed a
        theory to obtain existence, uniqueness, regularity, and decay
        estimates. Their methods were partially based on classical
        nonlinear semigroup theory
        (cf.~\cite{MR0348562,BenilanNonlinearEvolutionEqns,MR2582280}),
        the extension technique~\cite{MR2354493} (when
        $\Omega=\R^{d}$) and other classical PDE methods. Finite time of extinction of classical solutions to~\eqref{eq:7} satisfying homogeneous Dirichlet boundary conditions was proved in~\cite{KLkiahmholderregularitypaper}. Here we refine the method in~\cite{KLkiahmholderregularitypaper} and make it available for mild solutions of the more general boundary value problem~\eqref{eq:1}.
	
	The parabolic boundary-value problem~\eqref{eq:1} studied in
        this paper was recently studied by Giacomoni
        et al.~\cite{MR4230963} in the case $\varphi(r) = r^m$,
        $r \in \R$, given by~\eqref{eq:doublynonlinearporouseqn}, for $\frac{1}{2p-1} \le m < 1$, $p < \frac{d}{s}$, and
        $\Omega$ a bounded domain with boundary $\partial\Omega$
        of class $C^{1,1}$. By using similar techniques as B\'enilan
        and Gariepy \cite{BenilanGariepy}, they obtained existence and
        uniqueness of positive strong solutions under the assumption
        that the forcing term $g\equiv 0$ and the lower-order perturbation
        $f(\cdot,u)$ satisfies a specific growth condition related to
        $u^{q-1}$, $1<q\le p<\infty$. Our results presented here
        complement and improve those in~\cite{MR4230963} in various
        directions as specified below.

	\subsection{Main results}
	\label{sec:results}
	
	We begin by stating our well-posedness result in the
	sense of \emph{mild solutions} in $L^{1}$ and a comparison principle. We denote the intersection space $L^1\cap L^\infty$ by $L^{1\cap\infty}$ and restrict the fractional p-Laplacian to $L^{1\cap\infty}$ with $(-\Delta_p)^s_{\vert L^{1\cap \infty}}$ (see Section~\ref{sec:subgradientSection} and, in particular, Definition~\ref{def:sub-diff-in-X}). The composition operator $(-\Delta_p)^s_{\vert L^{1\cap \infty}}\circ\varphi$ is defined in $L^1$ (see Definition~\ref{def:compositionOperator}). Here we also use the notation $[u]^{+}$ to denote $\max\{u,0\}$, the positive part of $u$, and $[u]^1= u$. We also use the notation of $q$-brackets defined by Definition~\ref{def:q-bracket}.

	\begin{theorem}[{Well-posedness \& comparison principle in $L^1$}]
		\label{thm:1}
		Let $\Omega$ be an open domain in $\R^d$, $d \ge 1$. Let $1< p< \infty$, $0<s<1$, and suppose that $\varphi$, $f$
		satisfy~\eqref{hyp:1} and~\eqref{hyp:2}-\eqref{hyp:3}, respectively. Then the following statements hold.
		\begin{enumerate}[(1)]
			\item If either $\varphi \in W^{1,\infty}_{loc}(\R)$ or $\varphi\in W^{1,q}_{\loc}(\R)$ for $q > \frac{1}{1-s}$, 
            then one has that
            \begin{displaymath}
                \overline{D((-\Delta_p)^s_{\vert L^{1\cap \infty}}\varphi)}^{\mbox{}_{L^1}} = L^1.
            \end{displaymath}
			\item For every initial value $u_{0}\in \overline{D((-\Delta_p)^s_{\vert L^{1\cap \infty}}\varphi)}^{\mbox{}_{L^1}}$ and
			$g\in L^{1}(0,T;L^{1})$, there is a unique mild solution
			$u\in C([0,T];L^{1})$ of the initial boundary value
			problem~\eqref{eq:1}.  Moreover, for all $1\le q\le \infty$, one
			has that
			\begin{equation}
			\label{eq:theorem1.1estimate1}
			\norm{u(t)}_{q}\le e^{\omega t}\norm{u_{0}}_{q} 
			+ \int_{0}^{t}e^{\omega (t-r)}\norm{g(r)}_{q}\dr
			\end{equation}
			for every $u_{0}\in \overline{D((-\Delta_p)^s_{\vert L^{1\cap \infty}}\circ\varphi)}^{\mbox{}_{L^1}}\cap
			L^{q}$ and $g\in L^{1}(0,T;L^1\cap L^{q})$.
			
			\item For every $y_1$, $y_2\in
			L^{1}$, $g_{1}$, $g_{2}\in
			L^{1}(0,T;L^{1})$, and corresponding mild solutions
			$u_{1}$, $u_{2}$ of~\eqref{eq:1} with initial data $y_1$, $y_2$ respectively, one has
			\begin{equation}
			\label{eq:theorem1.1estimate2}
			\begin{split}
			&\norm{[u_{1}(t)-u_{2}(t)]^{\nu}}_{1}\le e^{\omega
				t}\norm{[y_1-y_2]^{\nu}}_{1}\\
			&\hspace{3.5cm} + \int_{0}^{t}e^{\omega (t-r)}[u_{1}(r)-u_{2}(r),g_{1}(r)-g_{2}(r)]_{\nu}\dr
			\end{split}
			\end{equation}
			for all $0\le t\le T$, and $\nu\in \{+,1\}$.
		\end{enumerate}
	\end{theorem}
	
	The statements of Theorem~\ref{thm:1} follow as an application of the
	existence theory developed by \cite{BenilanGariepy}, in the
	monograph~\cite{CoulHau2016} by Coulhon and the second author, and by classical nonlinear semigroup theory (cf.~\cite{BenilanNonlinearEvolutionEqns}). We give the details of the proof in Section~\ref{sec:well-posedness}. Our well-posedness result presented here complements and generalizes the recent result by Giacomoni et al.~\cite{MR4230963} by allowing general, strictly increasing functions $\varphi$ and by considering solutions which may change sign. In contrast to~\cite{MR4230963}, Theor\-em~\ref{thm:1} provides well-posedness for all initial data $u_0$ in $L^1$.
	\medskip
	
	Our next result is concerned with global $L^{\ell}-L^{\infty}$ regularity estimates, $1\le \ell <\infty$, for mild solutions $u$ of the initial boundary value problem~\eqref{eq:1}, implying an immediate smoothing effect. For these regularity estimates, the Sobolev embedding (see, for instance,~\cite{MR2527916,MR2944369})
	\begin{equation}
	\label{eq:6}
	W^{s,p}\hookrightarrow
	L^{p_{s}}\quad\text{with}\quad
	p_{s} = \begin{cases}
	\left(\frac{1}{p}-\frac{s}{d}\right)^{-1}& \text{ if } p < \frac{d}{s},\\
	\tilde{p}& \text{ if } p = \frac{d}{s},\\
	\infty & \text{ if } p > \frac{d}{s},
	\end{cases}
	\end{equation}
	and $\tilde{p} \in [p,\infty)$, is crucial, where we write $W^{s,p}$ for $W^{s,p}(\Omega)$. Theorem~\ref{thm:LellLinfty} is a special case of Theorem~\ref{thm:LqLinfty-deGiorgi} in Section~\ref{sec:extrapolation-to-infty} as illustrated in Section~\ref{sec:applicationToFractionalPLaplacian}. Theorem~\ref{thm:LqLinfty-deGiorgi} applies to abstract operators $A$ acting on $L^q$ and satisfying an abstract Sobolev inequality (as introduced in~\cite{CoulHau2016}). In particular we apply a De Georgi iteration (cf.~\cite{MR2680400}) to obtain an $L^{m+1}-L^\infty$ estimate which is then extrapolated to $L^\ell-L^\infty$. Here we refine the methods in~\cite{MR2680400} by Caffarelli and Vasseur,~\cite{MR2529737} by Porzio and~\cite{CoulHau2016} by Coulhon and the second author.
	
	\begin{theorem}[{Global $L^{\ell}-L^{\infty}$ estimates}]\label{thm:LellLinfty}
		Let $\Omega$ be an open domain in $\R^d$, $d \ge 1$. Let $p > 1$, $0 < s< 1$, and $m \ge 1$ such that 
		\begin{equation}
		\label{hyp:16}
        m(p-1)+(m+1)\frac{sp}{d} > 1.
		\end{equation}
		Further, let $\varphi(r)=r^m$ for $r\in \R$, and $f(\cdot,u)$ satisfy~\eqref{hyp:2}-\eqref{hyp:3}. Take $q_{s} = p_s$ if $p \neq \frac{d}{s}$ and $q_{s} > \max(p,1+\frac{1}{m})$ if $p = \frac{d}{s}$. For $\rho \ge m+1$ and $\psi > 1$ satisfying
		\begin{equation}
		\label{eq:psirhohyp1}
		\begin{cases}
        \frac{1}{\rho} < \left(1-\frac{1}{\psi}\right)p\left(\frac{1}{p}-\frac{1}{q_{s}}\right)& \text{if } m(p-1) \ge 1,\\
	 	\frac{1}{\rho}\le \left(1-\frac{1}{\psi}\right)p\left(\frac{m}{m+1}-\frac{1}{q_{s}}\right)& \text{if } m(p-1) < 1,
	 	\end{cases}
		\end{equation}
		and
		\begin{equation}
		\label{eq:psirhohyp2}
		\rho \ge \frac{1-m(p-1)}{1-\frac{p}{q_{s}}},
		\end{equation}
		let $g \in L^{\psi}(0,T;L^{\rho})\cap L^{1}(0,T;L^{1}\cap L^{1+m+\varepsilon})$ for some $\varepsilon > 0$. Let $1 \le \ell< m+1$ satisfy
		\begin{equation}
		\label{eq:45}
		\ell > \frac{1-m(p-1)}{1-\frac{p}{q_{s}}}.
		\end{equation}
		Then for every $u_{0} \in L^{\ell}\cap L^1$ the mild solution $u$ of~\eqref{eq:1} in $L^1$ satisfies
        \begin{equation}
		\label{eq:43}
		\begin{split}
		\norm{u(t)}_{\infty}& \le C\max \left(e^{\omega\beta_1 t}\left(\tfrac{1}{t}+\omega\right)^{\alpha},\, e^{\omega \beta_{2}t}\norm{g}_{L^{\psi}(0,t;L^{\rho})}^{\eta}\right)^{\frac{1}{\theta}}\left(1+N(t)^\gamma\right)\times\\
		&\hspace{3cm}\times\left(e^{\omega t}\norm{u_0}_\ell+\int_0^t e^{\omega(t-\tau)}\norm{g(\tau)}_\ell \dtau\right)^{\frac{\ell \gamma}{(m+1)\theta}}
		\end{split}
		\end{equation}
		for all $t \in (0,T]$, where we set
		\begin{equation}
		\label{eq:thm12MFs}
		\begin{split}
		N(t)& = \sup\limits_{s \in (0,t]}\tfrac{M(\frac{s}{2})\norm{g}_{L^1(\frac{s}{2},s;L^{m+1})}+e^{\frac{\omega \beta_{2}s}{2\gamma}}\, \norm{g}_{L^{\psi}(0,\frac{s}{2};L^{\rho})}^{\frac{\eta}{\gamma}}}
		{M(s)^{\frac{1}{\theta}}\left(e^{\omega s}\norm{u_0}_{\ell}+\int_{0}^{s}e^{\omega(s-\tau)}\norm{g(\tau)}_{\ell}\dtau\right)^{\frac{\ell}{(m+1)\theta}}},\\
        M(t)& = \max\left(e^{\omega\beta_1 t}\left(\tfrac{1}{t}+\omega\right)^{\alpha},e^{\omega \beta_{2}t}\, \norm{g}_{L^{\psi}(0,t;L^{\rho})}^{\eta}\right)^{\frac{1}{\gamma}},
		\end{split}
		\end{equation}
		for the constants given by
        \begin{equation}
		\label{eq:2exponent}
		\begin{split}
		\alpha& = \frac{1}{(m+1)p\left(\frac{m}{m+1}-\frac{1}{q_{s}}\right)},\quad \gamma = \frac{\frac{1}{p}-\frac{1}{q_{s}}}{\frac{m}{m+1}-\frac{1}{q_{s}}},\\
		\theta& = 1-\gamma\left(1-\frac{\ell}{m+1}\right),\\
		\eta& = \frac{1}{1-\frac{m+1}{\rho}+mp \left(1-\frac{1}{\psi}\right)\left(1-\frac{m+1}{mq_{s}}\right)},\\
		\beta_{1}& = \begin{cases}
		\frac{\frac{1}{mp}-\frac{1}{m+1}}{\frac{1}{m+1}-\frac{1}{mq_{s}}}& \text{if } m(p-1) < 1,\\
		0& \text{if } m(p-1) \ge 1,
		\end{cases}\\
        \beta_2& = \begin{cases}
		\eta(1-m(p-1))\left(1-\frac{1}{\psi}\right)& \text{if } m(p-1) < 1,\\
		0& \text{if } m(p-1) \ge 1,
        \end{cases}
		\end{split}
		\end{equation}
    and where $C > 0$ depends on $m$, $p$, $s$, $d$, $q_s$, $\ell$, $\rho$ and $\psi$.
	\end{theorem}
	
	In the case  of no forcing term, $g \equiv 0$, this simplifies to the following $L^\ell-L^\infty$ estimate.
	\begin{corollary}
		\label{cor:simpleLlLinf}
		Let $\Omega$ be an open domain in $\R^d$, $d \ge 1$. Let $p > 1$, $0 < s< 1$, and $m \ge 1$ such that~\eqref{hyp:16} holds. Let $q_{s} = p_s$ if $p \neq \frac{d}{s}$ and $q_{s} > \max(p,1+\frac{1}{m})$ if $p = \frac{d}{s}$. Let $\varphi(r)=r^m$ for $r\in \R$. Further, suppose that $f(\cdot,u)$ satisfies~\eqref{hyp:2}-\eqref{hyp:3} and $g \equiv 0$. Let $1 \le \ell< m+1$ and suppose $\ell$ satisfies~\eqref{eq:45}. Then for every $u_{0} \in L^{\ell} \cap L^1$ the mild solution $u$ of~\eqref{eq:1} in $L^1$ satisfies
		\begin{equation}
		\label{eq:43corg0}
		\begin{split}
		\norm{u(t)}_{\infty}& \le Ce^{\omega\beta t}t^{-\alpha}\norm{u_0}_\ell^{\gamma}
		\end{split}
		\end{equation}
		for all $t \in (0,T]$, where
		\begin{displaymath}
		\begin{split}
        \alpha& = \frac{1}{m(p-1)-1+\ell(1-\frac{p}{q_{s}})},\quad
		\gamma = \frac{\ell(1-\frac{p}{q_{s}})}{m(p-1)-1+\ell(1-\frac{p}{q_{s}})},\\
		\beta& = \begin{cases}
		\frac{\frac{1}{p}-\frac{m}{m+1}}{\frac{m}{m+1}-\frac{1}{p}+\frac{\ell}{m+1}\left(\frac{1}{p}-\frac{1}{q_{s}}\right)}& \text{if } m(p-1) < 1,\\
		0& \text{if } m(p-1) \ge 1,
		\end{cases}
		\end{split}
  		\end{displaymath}
		and $C > 0$ depends on $m$, $p$, $s$, $d$, $q_s$ and $\ell$.
	\end{corollary}
	
	Such regularity and decay estimates are common for these diffusion problems, see for example~\cite{MR2529737} and the references therein. In particular, with $f \equiv 0$, $g \equiv 0$ we have Corollary~\ref{cor:simpleLlLinf} with $\omega = 0$ and hence a decay in time. Similar estimates have been found for related problems, including for a class of local doubly nonlinear problems~\cite{MR2529737} related to a doubly nonlinear $p$-Laplacian evolution equation. In~\cite{MR2680400}, an $L^1-L^\infty$ estimate is found corresponding to the fractional Laplacian with $s = \frac{1}{2}$ and in~\cite{MR2737788} for the fractional porous medium equation with $s = \frac{1}{2}$. In the case of the fractional porous medium equation on $\R^d$ ($d \ge 1$), the same authors in~\cite{MR2954615} find such an $L^\ell-L^\infty$ regularizing effect for all $\ell \ge 1$. We can also compare this to the Barenblatt solutions found for related evolution problems in~\cite{MR4114983,vazquez2021fractional,MR4219198}, noting the convergence in $L^1(\R^d)$ as $t \rightarrow \infty$. We emphasise that in all cases the exponents given by~\eqref{eq:2exponent} in Corollary~\ref{cor:simpleLlLinf} agree with those found in these papers; in the case of~\cite{MR2529737}, by taking $s = 1$. We note that compared to~\cite{MR4230963}, we restrict to $m \ge 1$ rather than $0 < m < 1$ due to the limitation of Lemma~\ref{le:GlambdaPointwiseEstimateCases}. Such regularizing effects require separate consideration in this case.
	
	\medskip
	
	Now, for $1< p <\infty$ and $0 < s < 1$, set 
    \begin{displaymath}
		\tilde{D}((-\Delta_p)^s_{\vert_{L^{1\cap\infty}}}) = \left\lbrace u \in L^{1\cap\infty}\left|
		\begin{aligned}
		& \exists\, (u_{n},h_n)_{n\ge 1} \subseteq (-\Delta_p)^s_{\vert_{L^{1\cap\infty}}} \text{s.t.}\\
		& u_{n} \rightarrow u \text{ in } L^{1} \text{ and }\\
		&(u_n,h_n)_{n \ge 1} \text{ is bounded in } L^{\infty}\times L^1.
		\end{aligned}\right.
		\right\rbrace
  \end{displaymath}
  and for every continuous $\varphi$, let 
  \begin{displaymath}
      \hat{D}((-\Delta_{p})^{s}_{\vert L^{1\cap\infty}}\varphi)=\Big\{u\in L^{1}\,\Big\vert\,\varphi(u)\in \tilde{D}((-\Delta_p)^s_{\vert_{L^{1\cap\infty}}})\Big\}.
  \end{displaymath}
  Then, by taking advantage of \cite[Theorem 4.1]{BenilanGariepy}, we show in our next theorem that for every $u_0\in \hat{D}((-\Delta_{p})^{s}_{\vert L^{1\cap\infty}}\varphi)$, the corresponding mild solutions $u$ of the initial boundary value problem~\eqref{eq:1} is strong and distributional. In this theorem, the term $[\,\cdot\,]_{s,p}$ denotes the Gagliardo semi-norm~\eqref{eq:8}. We prove this result in Section~\ref{sec:strongSols}.
	
	\begin{theorem}[{Mild solutions are strong and distributional}]
		\label{th:MildToStrongLipschitzFracPLap}
		Let $\Omega$ be an open domain in $\R^d$, $d \ge 1$, of finite Lebesgue measure. 
        Let $1< p <\infty$, $0 < s < 1$, $\varphi\in C(\mathbb{R})$ be a strictly increasing function such that $\varphi^{-1} \in AC_{\loc}(\mathbb{R})$. 
        Suppose $f(\cdot,u)$ satisfies~\eqref{hyp:2}-\eqref{hyp:3} and let $F$ be the Nemystkii operator of $f$. 
        Further suppose that $g \in BV(0,T;L^{1})\cap L^{1}(0,T;L^{\infty})$. 
        Then for every $u_0\in \hat{D}((-\Delta_{p})^{s}_{\vert L^{1\cap\infty}}\varphi)$, the mild solution $u$ of~\eqref{eq:1} is a strong distributional solution in $L^{1}$ having the regularity
        \begin{displaymath}
        u \in W^{1,\infty}((0,T);L^{1})\cap L^{\infty}([0,T];L^{\infty})\cap C([0,T];L^{q})
        \end{displaymath}
        for every $1 \le q < \infty$. Moreover, $\varphi(u)$ has a weak derivative given by
        \begin{equation}
        \label{eq:ChainRuleForPhiFracPLap}
        \frac{\td}{\td t}\varphi(u(t)) = \varphi'(u(t))\frac{\td u}{\td t}\qquad \text{ in } L^{2} \text{ for a.e.~} t > 0
        \end{equation}
        and the function $t \mapsto \left[\varphi(u(t))\right]^{p}_{s,p}$ has derivative given by
        \begin{equation}
        \label{eq:DerivativeOfPsiPhiUFracPLap}
        \begin{split}
        \frac{\td }{\td t}\left[\varphi(u(t))\right]^{p}_{s,p}& = -\norm{\sqrt{\varphi'(u(t))}\frac{\td u}{\td t}(t)}_2^2\\
        &\qquad-\braket{F(u(t))-g(t),\frac{\td u}{\td t}(t)\varphi'(u(t))} 
        \end{split}
        \end{equation}
        for a.e.~$t > 0$.
	\end{theorem}
	
	Strong solutions are naturally a standard aim of results in such a theory and so have been investigated for many related problems. For the fractional $p$-Laplacian evolution equation, these have been found in~\cite{MR3491533} and for the fractional porous medium equation on $\R^d$, in~\cite{MR2954615}. Comparing this with \cite{MR4230963}, we note that a different regime is studied with $\frac{1}{2p-1} \le m < 1$ compared to $m > 0$ for non-negative solutions.
	
	In the case $\varphi(u) = u^m$, the operator is homogeneous and we have the following Lipschitz estimate via the regularising effect of homogeneous operators in~\cite{BenilanRegularizingMR648452}. Here we define, for $g \in L^1_{\loc}(0,T;L^1)$ and $0 \le t \le T \le \infty$,
	\begin{equation}
	\label{eq:VLipdef}
	V(t,g) = \limsup_{\xi\rightarrow 0^+}\int_0^{t/(1+\xi)}\frac{\norm{g(\tau(1+\xi))-g(\tau)}_1}{\xi}\dtau.
	\end{equation}
	Note that $V(T,g) < \infty$ is equivalent to $t \rightarrow tg(t)$ having (essentially) finite variation on $[0,T]$.
	
	\begin{theorem}[Derivative estimates for $\varphi(u) = u^m$]
	\label{thm:derivativeEstimates}
		Let $\Omega$ be an open domain in $\R^d$, $d \ge 1$. Let $p > 1$, $0 < s < 1$ and $f(\cdot,u)$ satisfy~\eqref{hyp:2}-\eqref{hyp:3}. 
        Suppose $\varphi(r) = r^m$, $r \in \R$ for $m > 0$. Then we have the following regularity estimates.
		\begin{enumerate}[(1)]
			\item Suppose $m(p-1) \neq 1$, $g \in L^1(0,T;L^1)$ and $V(T,g) < \infty$. Then for every $u_{0} \in L^{1}$, the unique mild solution $u$ to~\eqref{eq:1} is Lipschitz continuous on each compact subset of $(0,T]$ satisfying
			\begin{equation}
			\label{eq:lipschitzContinuity}
			\begin{split}
			\limsup_{h\rightarrow 0^+}\frac{\norm{u(t+h)-u(t)}_{1}}{h}&\le \frac{Ce^{2\omega t}}{t}\left(\norm{u_0}_{1}+\int_0^t
			\norm{g(\tau)}_1 \dtau\right)\\
			&\hspace{3cm}+\frac{e^{\omega t}}{t}V(t,g)
			\end{split}
			\end{equation}
            where $C = \frac{m(p-1)+2}{|m(p-1)-1|}$.
			\item Let $g \in BV(0,T;L^{1})\cap L^{1}(0,T;L^{\infty})$. Further, suppose that $m \ge 1$ satisfies
            \begin{equation}
    		\tag{\ref{hyp:16}}
                m(p-1)+(m+1)\frac{sp}{d} > 1.
    		\end{equation}
			Take $q_s = p_s$ if $p \neq \frac{d}{s}$ and $q_s > \max(p,1+\frac{1}{m})$ if $p = \frac{d}{s}$, and suppose that
			\begin{equation}
			\label{eq:651}
			m\left(p-1\right) > \frac{p}{q_s}.
			\end{equation}
			If, in addition, for $\rho \ge m+1$ and $\psi > 1$ satisfying
			\begin{displaymath}
			\rho \ge \frac{m+1-mp}{1-\frac{p}{q_s}} \quad \text{ and }\quad \frac{1}{\rho} < \left(1-\frac{1}{\psi}\right)\left(1-\frac{p}{q_s}\right),
			\end{displaymath}
			$g$ belongs to $L^{\psi}(0,T;L^{\rho})\cap L^{m+1}(0,T;L^{\infty})$, then for every $u_{0} \in L^{1}$, the mild solution $u$ to~\eqref{eq:1} is a strong solution in $L^1$ and satisfies
			\begin{equation}
			\label{eq:W12estimate}
			\begin{split}
		      \int_{0}^{t}s^{\tilde{\alpha}}
            &\int_{\Omega}u^{m-1}(s)\left|\frac{\td u}{\td s}\right|^{2}\dmu \ds+t^{\tilde{\alpha}}\E(\varphi(u(t)))\\
			& \le Ct\left(1+t^2\right)\max \left(e^{\omega\beta_1 t}, t^{\alpha}e^{\omega \beta_{2}t}\norm{g}_{L^{\psi}(0,t;L^{\rho})}^{\eta}\right)^{\frac{m}{\theta}}\times\\
			&\hspace{1cm}\left(e^{\omega t}\norm{u_0}_1+\int_0^t e^{\omega(t-\tau)}\norm{g(\tau)}_1 \dtau\right)^{\frac{\delta m}{\theta(m+1)}+1}\left(1+F(t)^{\gamma}\right)\\
			&\hspace{3cm}+\int_{0}^{t}\left(\tilde{\alpha}+m s\right)s^{\tilde{\alpha}-1}\norm{g(s)}_{m+1}^{m+1} \ds
			\end{split}
			\end{equation}
			for all $t \in (0,T]$ with $\tilde{\alpha} = \frac{\alpha m}{\theta}+2$, $F(s)$ given by~\eqref{eq:thm12MFs}, constants given by~\eqref{eq:2exponent} with $\ell = 1$ and where $C > 0$ depends on $m$, $p$, $s$, $d$, $q_s$, $\psi$, $\rho$ and $\omega$.
		\end{enumerate}
	\end{theorem}

    In the case  of no forcing term, $g \equiv 0$ and taking $\ell = 1$, this simplifies to the following $L^1-L^\infty$ estimate.
    \begin{corollary}
        Let $\Omega$ be an open domain in $\R^d$, $d \ge 1$. Let $p > 1$, $0 < s < 1$, $f(\cdot,u)$ satisfy~\eqref{hyp:2}-\eqref{hyp:3} and $g \equiv 0$. 
        Suppose $\varphi(r) = r^m$, $r \in \R$ for $m \ge 1$ satisfying
        \begin{equation}
        \tag{\ref{hyp:16}}
            m(p-1)+(m+1)\frac{sp}{d} > 1.
        \end{equation}
        Take $q_s = p_s$ if $p \neq \frac{d}{s}$ and $q_s > \max(p,1+\frac{1}{m})$ if $p = \frac{d}{s}$, and suppose that
        \begin{equation}
        \tag{\ref{eq:651}}
        m\left(p-1\right) > \frac{p}{q_s}.
        \end{equation}
        Then for every $u_{0} \in L^{1}$, the mild solution $u$ to~\eqref{eq:1} satisfies
        \begin{equation}
        \label{eq:W12estimateNoForcing}
        \begin{split}
        \int_{0}^{t}s^{\tilde{\alpha}}\int_{\Omega}u^{m-1}(s)&\left|\frac{\td u}{\td s}\right|^{2}\dmu \ds+t^{\tilde{\alpha}}\E(\varphi(u(t)))\\
        & \le Ct\left(1+t^{2}\right)e^{\omega\beta t}\norm{u_0}_1^{\frac{\gamma m}{\theta(m+1)}+1}
        \end{split}
        \end{equation}
        for all $t \in (0,T]$ with
        \begin{displaymath}
        \begin{split}
            \tilde{\alpha}& = \frac{\alpha m}{\theta}+2,\\
            \beta& = \frac{m\beta_1}{\theta}+\frac{\gamma m}{\theta(m+1)}+1,
        \end{split}
        \end{displaymath}
        constants given by~\eqref{eq:2exponent} with $\ell = 1$ and where $C > 0$ depends on $m$, $p$, $s$, $d$, $q_s$ and $\omega$.
    \end{corollary}
    
	In this case we also have local H\"older regularity. Here $V(T,g)$ is defined as in~\eqref{eq:VLipdef}.
	
	\begin{theorem}[Local H\"older continuity]
		\label{thm:localHolder}
		Let $\Omega$ be a bounded domain in $\R^d$, $d \ge 1$. 
        Assume $2 \le p < \infty$, $0 < s < 1$ such that $sp \ge d$ and $\varphi(u) = u^m$, $m > 0$. Suppose $f$ satisfies~\eqref{hyp:2}-\eqref{hyp:3} with $F$ the Nemytskii operator of $f$, $g \in L^{1}(0,T;L^1)$ and $V(T,g) < \infty$. 
        Let $u(t)$ be the mild solution to~\eqref{eq:1} for $u_0 \in L^1$. 
        
        Then $u^m(t) \in C_{\loc}^{\delta}(\Omega)$ for every $0 < \delta < \min(\frac{sp-d}{p-1},1)$ and a.e.~$t \in (0,T)$. 
        In particular, for $m \ge 1$, $u(t) \in C_{\loc}^{\delta}(\Omega)$ for every $0 < \delta < \min(\frac{sp-d}{(p-1)m},1)$ and a.e.~$t \in (0,T)$.
	\end{theorem}

	Local H\"older regularity has been established for the fractional porous medium equation in~\cite{KLkiahmholderregularitypaper} for $\frac{n-2s}{n+2s} < m < 1$ via an oscillation lemma and in~\cite{MR2954615} for the fractional porous medium equation on $\R^d$ for $m \ge 1$. This result applies the elliptic local regularity proved in~\cite{MR3861716} and extends the work in~\cite{brasco2019continuity}, which considers H\"older regularity in space and time for a weak formulation of the fractional $p$-Laplacian evolution problem.
	
	\medskip
	
	Furthermore, for $\varphi$ given by the identity, we have continuity in time and global H\"older continuity in space for a bounded domain. Both H\"older regularity results are proved in Section \ref{sec:holderRegularity}.
	\begin{theorem}[{Global H\"older regularity for $\varphi(r) = r$}]
		\label{thm:globalHolder}
		Let $\Omega$ be a bounded\\ domain in $\R^{d}$, $d\ge 2$, with a
		boundary $\partial\Omega$ of the class $C^{1,1}$, $1<p<\infty$, $0<s<1-\frac{1}{p}$, and suppose $f$ satisfies~\eqref{hyp:2}-\eqref{hyp:3} with $F$ the Nemytskii operator of $f$ on $C_0(\Omega)$. 
        Then $-((-\Delta_p)^s_{\vert_{C_0}}+F)$ generates a strongly continuous semigroup of quasi contractions on $C_{0}(\Omega)$. 
        In particular, for $\varphi(r) = r$, $r \in \R$, $u_{0}\in C_{0}(\Omega)$ and $g\in L^{\infty}(0,T;L^{\infty})\cap BV(0,T;L^{\infty})$, let $u$ be the unique mild solution of the initial boundary value problem~\eqref{eq:1}. 
        Then $u \in C^{lip}([\delta,T];C_0(\Omega))$ for every $0 < \delta < T$ and for some $\alpha \in (0,s]$, $u(t) \in C^{\alpha}(\overline{\Omega})$ for all $t \in (0,T)$.
	\end{theorem}
	
	This result extends~\cite{MR3781159}, wherein Giacomoni and Tiwari obtain continuity up to the boundary in space, uniform in time. Here we refine the estimate, proving that the resolvent is $m$-accretive in $C_0$. We use the global H\"older regularity of the elliptic problem given by~\cite{IannizzottoHolderRegularity} with semigroup theory for the operator restricted to the space of continuous functions. H\"older regularity has been considered for the elliptic problem, for example, in~\cite{MR3861716}.
	
	\medskip
	
	By obtaining a comparison theorem for solutions with inhomogeneous boundary data and considering the fundamental solution we can prove finite time of extinction of solutions to~\eqref{eq:1}. This is proved in Section \ref{sec:comparisonPrinciple}.
 
	\begin{theorem}[\bf{Finite time of extinction}]
		\label{thm:extinction}
		Let $\Omega$ be a bounded domain in $\R^d$, $d \ge 1$. Let $u$ be a strong distributional solution to~\eqref{eq:1} with $u_0 \in \tilde{D}((-\Delta_p)^s)$ where $\varphi$ is strictly increasing, $\varphi(\mathbb{R}) = \mathbb{R}$, $\varphi(0) = 0$ and $\frac{1}{\varphi} \in L^{1}(0,\norm{u_{0}}_{\infty})$, $f\equiv0$ and $g \equiv 0$. Then $u(t,\cdot) = 0$  for all $t \ge T^{*}$ where $T^*$ is given by
		\begin{equation}
		\label{eq:Tstar}
		T^* = \frac{1}{\tilde{C}R^{d-ps}}\int_{0}^{\norm{u_0}_\infty}\frac{1}{\varphi(s)}\ds
		\end{equation}
		where $\tilde{C}$ is given by~\eqref{eq:tildeCforTstar} and depends on $R$, $p$, $s$ and $d$.
	\end{theorem}

    We have the following Corollary in the case $\varphi(r) = r^m$, $r \in R$.
 
	\begin{corollary}
		\label{thm:extinctionrm}
		Let $\Omega$ be a bounded domain in $\R^d$, $d \ge 1$. Let $u$ be a strong distributional solution to~\eqref{eq:1} with $u_0 \in L^\infty$  where $\varphi(r) = r^m$, $0 < m < 1$, $f\equiv 0$ and $g \equiv 0$. 
        Then $u(t,\cdot) = 0$  for all $t \ge T^{*}$ where $T^*$ is given by~\eqref{eq:Tstar}.
	\end{corollary}

	Finite time of extinction of solutions to~\eqref{eq:1} was also proved for the fractional porous medium equation ($p=2$) in~\cite{KLkiahmholderregularitypaper} in the Dirichlet case and~\cite{MR2954615} for the Cauchy problem. See also~\cite{vazquez2021fractional,vazquez2021growing} for discussion of extinction for the fractional $p$-Laplacian evolution problem on $\R^d$.

	%
	%
	%
	%
	
	\section{Preliminaries}
	\label{sec:pre}
	In this section, we introduce some basic definitions and recall known
	results used throughout this paper.

	%
	%
	%
	%

	\subsection{A brief primer on Gagliardo-Sobolev-Slobodecki\u{\i} spaces}
	\label{sec:frac-sobolev-spaces}
	
	In this section, we provide a short summary of
	Gagliardo-Sobolev-Slobodecki\u{\i} spaces, which are necessary to study the
	intial boundary value problem~\eqref{eq:1} with functional analytical
	tools. For a deeper understanding of this theory, we refer the
	interested reader to~\cite{MR1411441,MR2527916} or~\cite{MR3306694}.\medskip
	
	For $1< p<\infty$, $0 < s< 1$, and an open subset $\Omega$ of $\R^{d}$,
	we write $W^{s,p}$ for the \emph{Gagliardo-Sobolev-Slobodecki\u{\i} space} of fractional order $s$, also known as the \emph{fractional Sobolev space} $W^{s,p}(\Omega)$ given by
	\begin{displaymath}
	W^{s,p}(\Omega) = \left\lbrace u \in L^{p} \Big|  [u]_{s,p}<\infty \right\rbrace
	\end{displaymath}
	 where
	\begin{equation}
	\label{eq:8}
	[u]_{s,p}:=\left(\int_{\Omega}\int_{\Omega}\frac{\abs{u(x)-u(y)}^{p}}{\abs{x-y}^{d+sp}}\dy\dx\right)^{1/p}
	\end{equation}
	denotes the \emph{$s$-Gagliardo semi-norm}. The space $W^{s,p}$ defines a Banach space if it is equipped with the norm 
	\begin{displaymath}
	\norm{u}_{W^{s,p}}:=\left(\norm{u}^{p}_{L^{p}}+[u]^{p}_{s,p}\right)^{1/p}.
	\end{displaymath}
	Further, let $W^{s,p}_{0}(\Omega)$, denoted by $W^{s,p}_{0}$, be the closure in $W^{s,p}$ of the set $C^{\infty}_{c}(\Omega)$ of test functions. By~\cite[Theorem~10.1.1]{MR1411441}, the space $W^{s,p}_{0}$
	admits, for $1<p<\infty$, the characterization
	\begin{displaymath}
	W^{s,p}_{0}=\left\lbrace u\in
	W^{s,p}(\R^{d})\,\bigg|
	\begin{array}[c]{l}
	\exists \text{ } \overline{u} : \R^{d}\to \R\text{
		s.t. }\overline{u}=u\text{ a.e.~on $\R^{d}$}\\
	\text{ and } \overline{u}=0\text{ quasi-everywhere on }\R^{d}\setminus\Omega
	\end{array}
	\right\rbrace,
	\end{displaymath}
	where $\overline{u}$ denotes a (quasi-continuous) representative of $u$. Therefore, the space $W^{s,p}_{0}$ incorporates \emph{homogeneous Dirichlet boundary conditions} in a weak sense.

	\subsection{Basics of nonlinear semigroup theory}
	\label{sec:pre-ns}
	
	We begin by reviewing some basic definitions and important results in nonlinear semigroup theory from the standard literature~\cite{MR0348562,MR2582280} and the
	monograph~\cite{CoulHau2016}.
	
	%
	%
	%
	%
	
	\subsubsection{The general framework}
	\label{subsec:framework}
	
	Let $(\Sigma,\mu)$ be a measure space with a positive $\sigma$-finite
	measure $\mu$, and $M(\Sigma,\mu)$ be the set of $\mu$-a.e.~equivalence classes
	of measurable functions $u : \Sigma\to \R$. For $1\le q\le \infty$, we
	denote by $L^{q}_{\mu}$ the classical \emph{Lebesgue space}
	equipped with the standard $L^{q}$-norm
	\begin{displaymath}
	\norm{u}_{q}:=
	\begin{cases}	
	\bigg(\displaystyle\int_{\Sigma}\abs{u}^{q}\dmu\bigg)^{1/q} 
	&\quad\text{if $q<\infty$,}\\[7pt]
	\inf\Big\{k\in
	[0,\infty]\,\Big\vert\,\mu(\{\abs{u}>k\})=0\Big\}
	&\quad\text{if $q=\infty$.}
	\end{cases}
	\end{displaymath}			
	If $\Omega$ is an open subset of $\R^{d}$, $d\ge 1$, and $\mu$ is the classical $d$-dimensional Lebesgue measure restricted
	to the trace-$\sigma$-algebra $\mathcal{B}(\R^{d})\cap \Omega$
	 then we write $L^{q}$ instead of $L^{q}_{\mu}$.
	 
	Let $X\subseteq M(\Sigma,\mu)$ be a Banach space with norm
	$\norm{\cdot}_{X}$. The main object in this section is the 
	abstract Cauchy problem (in $X$)
	\begin{equation}
	\label{eq:9}
	\begin{cases}
	\tfrac{\td u}{\dt}(t)+Au(t)\ni g(t) & \quad\text{for a.e.~$t\in (0,T)$,}\\
	u(0)=u_{0}, &
	\end{cases}
	\end{equation}
	for given initial value $u_{0}\in \overline{D(A)}^{\mbox{}_{X}}$ and
	forcing term $g\in L^{1}(0,T;X)$. In~\eqref{eq:9}, $A$ denotes a (possibly) multi-valued operator
	$A : D(A)\to 2^{X}$ on $X$ with \emph{effective domain}
	$D(A):=\{u\in X\,\vert\,Au\neq \emptyset\}$,
	the closure of $D(A)$ in $X$ denoted by $\overline{D(A)}^{\mbox{}_{X}}$, and
	\emph{range} $\Rg(A):=\bigcup_{u\in D(A)}Au$. In this setting, it is
	standard to view an operator $A$ as a subset of $X\times X$, or
	\emph{relation} on $X$, and to identify $A$ with its graph
	\begin{displaymath}
	A:=\Big\{(u,v)\in X\times X\,\Big\vert\; v\in Au\Big\}.
	\end{displaymath}
	
	\begin{definition}
		An operator $A$ on $X$ is called 
		\emph{accretive} (in $X$) if
		\begin{equation}
		\label{eq:63}
		\norm{u-\hat{u}}_{X}\le \norm{u-\hat{u}+\lambda (v-\hat{v})}_{X}
		\end{equation}
		for every $(u,v)$, $(\hat{u},\hat{v})\in A$ and every
		$\lambda > 0$. Further, an operator $A$ on $X$ is called \emph{quasi accretive} if there is
		an $\omega\in \R$ such that $A+\omega I$ is accretive in
		$X$. 
	\end{definition}
	
	Clearly, if $A+\omega I$ is accretive in $X$ for some $\omega\in \R$
	then $A+\tilde{\omega} I$ is accretive for every
	$\tilde{\omega}\ge \omega$. Thus, there is no loss of generality in
	assuming that if $A$ is quasi accretive then there is an $\omega\ge 0$
	such that $A+\omega I$ is accretive in $X$.
	
	Equivalently, $A$ is accretive in $X$ if and only if, for every
	$\lambda>0$, the \emph{resolvent operator} of $A$, defined by
	$J_{\lambda}:=(I+\lambda A)^{-1}$, is a single-valued mapping
	from $\Rg(I+\lambda A)$ to $D(A)$ which is \emph{contractive} (also
	called \emph{non-expansive}) with respect to the norm of
	$X$. That is,
	\begin{displaymath}
	\norm{J_{\lambda}u-J_{\lambda}\hat{u}}_{X}\le \norm{u-\hat{u}}_{X}
	\end{displaymath}
	for all $u$, $\hat{u}\in \Rg(I+\lambda A)$ and $\lambda>0$. 
	
	The next definition is taken from~\cite{CoulHau2016}.
	\begin{definition}
	\label{def:q-bracket}
	For $1\le q<\infty$, we define the \emph{$q$-bracket} on
	$L^{q}_{\mu}$ to be the mapping
	$[\cdot,\cdot]_{q} : L^{q}_{\mu}\times L^{q}_{\mu}\to \R$
	defined by
	\begin{displaymath}
	[u,v]_{q}:=
	\displaystyle{\lim_{\lambda\to 0+}
		\frac{\frac{1}{q}\norm{u+\lambda v}_{q}^{q}-
			\frac{1}{q}\norm{u}_{q}^{q}}{\lambda}}
	\end{displaymath}
	for $u$, $v\in L^{q}_{\mu}$. We further define the bracket
	\begin{displaymath}
	[u,v]_{+}:=
	\displaystyle{\lim_{\lambda\to 0+}
		\frac{\norm{[u+\lambda v]^{+}}_{1}-
			\norm{[u]^{+}}_{1}}{\lambda}}
	\end{displaymath}
	for $u$, $v\in L^{1}_{\mu}$.
	\end{definition}

	Then an operator $A$ on
	$L^{q}_{\mu}$ is accretive in $L^{q}_{\mu}$ if and only if
	\begin{displaymath}
	[u-\hat{u},v-\hat{v}]_{q}\ge 0\qquad\text{for all $(u,v)$,
		$(\hat{u},\hat{v})\in A$.}
	\end{displaymath}
	The $q$-bracket $[\cdot,\cdot]_{q}$
	is upper semicontinuous (respectively, continuous if $1<q<\infty$) and  if $1<q<\infty$,
	\begin{equation}
	\label{eq:35}
	[u,v]_{q}=\int_{\Sigma} \abs{u}^{q-2}u\,v\dmu
	\end{equation}
	for every 
	$u$, $v\in L^{q}_{\mu}$. While for $q=1$, $[\cdot,\cdot]_{1}$ reduces to the classical
	\emph{brackets} $[\cdot,\cdot]$ on $L^{1}_{\mu}$ given by
	\begin{equation}
	\label{eq:67}
	[u,v]_{1}=\int_{\{u\neq 0\}}\text{sign}_{0}(u)\,v\dmu + \int_{\{u=0\}}\abs{v}\dmu
	\end{equation}
	for $u$, $v\in L^{1}_{\mu}$, where the \emph{restricted
		signum} $\sign_{0}$ is defined by
	\begin{displaymath}
	\text{sign}_0(s)=
	\begin{cases}
	1 & \text{if $s>0$,}\\
	0 & \text{if $s=0$,}\\
	-1 & \text{if $s<0$,}
	\end{cases}
	\end{displaymath}
	for $s\in \R$ (cf.~\cite[Section 2.2 \&
	Example~(2.8)]{BenilanNonlinearEvolutionEqns} or \cite[pp 102]{MR2582280}).
	\medskip
	
	Next, we introduce the following class of operators. 
	
	\begin{definition}
		An operator $A$ on $X$ is called \emph{$m$-accretive in $X$} if $A$ is accretive in $X$ and satisfies the so-called \emph{range condition}
		\begin{equation}
		\label{eq:range-condition}
		\Rg(I+\lambda A)=X\qquad
		\text{for some (or equivalently all) $\lambda>0$,}
		\end{equation}
		and an operator $A$ on $X$ is called \emph{quasi $m$-accretive in $X$} if there is an $\omega\ge 0$ such that $A+\omega I$ is $m$-accretive in $X$.
	\end{definition}

	By the classical theory of nonlinear evolution problems
	(cf.~\cite{BenilanNonlinearEvolutionEqns}, or
	alternatively,~\cite[Corollary~4.1]{MR2582280}), the condition
	\emph{`$A$ is quasi $m$-accretive in $X$'} ensures that for given
	$u_0\in \overline{D(A)}^{\mbox{}_{X}}$ and $g\in L^{1}(0,T;X)$, the
	Cauchy problem~\eqref{eq:9} admits a unique \emph{mild solution}, which is
	continuously dependent on $u_{0}$ and $g$.
	
	\begin{definition}\label{def:mild-solution}
		Suppose $g\in L^{1}(0,T;X)$ for some $T>0$. A \emph{mild solution} $u$ in $X$ of Cauchy problem~\eqref{eq:9} is a function $u\in C([0,T];X)$ such that for every $\varepsilon>0$ there is a partition $\sigma:0=t_{0}<\cdots<t_{N}\le T$ of the interval $[0,t_{N}]$ and a finite sequence $(g_{i})_{i=q}^{N}$ with the following properties: $T-t_{N} < \varepsilon$, $t_{i}-t_{i-1}<\varepsilon$ for every $i=1,\dots,N$,
		\begin{displaymath}
		\sum_{i=1}^{N}\int_{t_{i-1}}^{t_{i}}\norm{g(s)-g_{i}}_{X}<\varepsilon,
		\end{displaymath}
		there exists a step function
		$u_{\varepsilon,\sigma} : [0,T]\to X$ of the form
		\begin{equation}
		\label{eq:10}
		u_{\varepsilon,\sigma}(t)=u_{0}\,\mathds{1}_{\{t_{0}=0\}}(t)+\sum_{i=1}^{N}u_{\varepsilon,\sigma}(t_{i})\;
		\mathds{1}_{(t_{i-1},t_{i}]}(t),
		\end{equation} 
		where the values of $u_{\varepsilon,\sigma}$ on $(t_{i-1},t_{i}]$, denoted by $u_{i}$,
		recursively solve the finite difference equation
		\begin{displaymath}
		u_{i}+(t_{i}-t_{i-1})Au_{i}\ni (t_{i}-t_{i-1})g_{i} + u_{i-1}\qquad\text{for every $i=1,\dots,N$}
		\end{displaymath}
		and
		\begin{displaymath}
		\sup_{t\in [0,T]}\norm{u(t)-u_{\varepsilon,\sigma}(t)}_{X}\le\varepsilon.
		\end{displaymath}
	\end{definition}
	
	In the homogeneous case $g\equiv 0$, if $A$ is quasi $m$-accretive
	in $X$, then the Crandall-Liggett theorem
	(cf.~\cite[Theorem~I]{MR0287357}) states that for every element
	$u_{0}$ of $\overline{D(A)}^{\mbox{}_{X}}$, there is a unique mild
	solution $u$ of~\eqref{eq:9} in $X$ for every $T > 0$ and this solution $u$ can be given by the \emph{exponential formula}
	\begin{equation}
	\label{eq:36}
	u(t)=\lim_{n\to\infty}\left(I+\tfrac{t}{n}A\right)^{-n}u_{0}
	\end{equation}
	uniformly in $t$ on compact intervals. For every
	$u_0\in \overline{D(A)}^{\mbox{}_{X}}$, setting
	\begin{equation}
	\label{eq:semigroup}
	T_{t}u_{0}:=u(t),\qquad\text{for every $t\ge0$,}
	\end{equation}
	defines a (nonlinear) \emph{strongly continuous semigroup}
	$\{T_{t}\}_{t\geq 0}$ of \emph{($\omega$-)quasi con\-tractions}
	$T_{t} : \overline{D(A)}^{\mbox{}_{X}}\to
	\overline{D(A)}^{\mbox{}_{X}}$ with
	$\omega\in \R$. More precisely, the family $\{T_{t}\}_{t\geq 0}$
	satisfies the following three properties:
	\begin{itemize}
		\item \emph{semigroup property}
		\begin{displaymath}
		T_{s+t}=T_{t}\circ T_{s}\qquad\text{for every $s$, $t\ge 0$,}
		\end{displaymath}
		\item \emph{strong continuity}
		\begin{displaymath}
		\lim_{t\to0+}\norm{T_{t}u-u}_{X}=0\qquad
		\text{for every $u\in \overline{D(A)}^{\mbox{}_{X}}$,}
		\end{displaymath}
		\item \emph{exponential growth property in $X$} or
		\emph{($\omega$-)quasi contractivity in $X$}
		\begin{displaymath}
		\norm{T_{t}u-T_{t}v}_{X}\le e^{\omega\,t}\norm{u-v}_{X}
		\qquad\text{for all $u$,
			$v\in \overline{D(A)}^{\mbox{}_{X}}$, $t\ge0$.}
		\end{displaymath}
	\end{itemize}
	
	For the family $\{T_{t}\}_{t\ge 0}$ on
	$\overline{D(A)}^{\mbox{}_{X}}$, the operator
	\begin{displaymath}
	A_{\circ}:=\Bigg\{(u,v)\in X\times X\Bigg\vert\;\lim_{h\rightarrow
		0^{+}}\frac{T_{h}u-u}{h}=v\text{ in $X$}\Bigg\}
	\end{displaymath}
	is a well-defined mapping $A_{\circ} : D(A_{\circ})\to X$ with domain
	\begin{displaymath}
	D(A_{\circ}):=\Big\{u\in X\, \Big\vert\, \lim_{h\rightarrow
		0^{+}}\frac{T_{h}u-u}{h}\text{ exists
		in }X\Big\}
	\end{displaymath}
	called the \emph{infinitesimal generator}
	of $\{T_{t}\}_{t\ge 0}$. If $\{T_{t}\}_{t\ge 0}$ is $\omega$-quasi contractive in
	$X$, then $-A_{\circ}$ is $\omega$-quasi accretive in $X$. 
	
	Since mild solutions of Cauchy problem~\eqref{eq:9} are only the locally uniform (in time) limit of step functions~\eqref{eq:10} with values in $X$, it is important to know whether they are actually \emph{strong solutions} of~\eqref{eq:9} in $X$.
	
	\begin{definition}\label{def:strong-solution}
		Given $u_{0}\in X$ and $g\in L^{1}(0,T;X)$ for some $T>0$, a
		function $u \in C([0,T];X)$ is called a \emph{strong solution} in
		$X$ of Cauchy problem~\eqref{eq:9} if $u(0)=u_{0}$, $u$ belongs to
		$W^{1,1}_{loc}((0,T);X)$ and for a.e.~$0<t<T$, one has that
		$u(t)\in D(A)$ and $g(t)-\tfrac{\td u}{\dt}(t)\in Au(t)$.
	\end{definition}

    Since we typically take $A$ to be the closure of $(-\Delta_p)^s$ in $L^1\times L^1$, we also want to consider solutions with further regularity on $\varphi(u)$. Hence we consider \emph{distributional solutions} with the following definition.

    \begin{definition}
    \label{def:distributional-solution}
       For given $f$ satisfying \eqref{hyp:2}-\eqref{hyp:3} and $g\in L^{1}_{loc}((0,\infty);L^1_{loc})$, a function $u \in C([0,\infty);L^{1})$ is called a \emph{distributional} solution of initial value problem~\eqref{eq:1} if $u(0)=u_{0}$ in $L^{1}$, $\varphi(u)\in L^{p}_{loc}((0,\infty);W^{s,p})$ and for every test function $\xi \in C_c^{\infty}([0,\infty)\times\R^d)$, one has that
        \begin{align*}
	    & -\int_{\R^{d}}u\,\xi\,\dx\Bigg\vert_{t_{1}}^{t_{2}}
        -\int_{t_{1}}^{t_{2}}\int_{\R^d}u\,\xi_t\,\dx\dt+\int_{t_1}^{t_2}\int_{\R^d}(f(x,u)-g)\xi\dx\dt\\ 
        &\hspace{2cm} +\int_{t_{1}}^{t_{2}}\int_{\R^{2d}}\frac{(\varphi(u)(t,x)-\varphi(u)(t,y))^{p-1}
        (\xi(t,x)-\xi(t,y))}{|x-y|^{d+sp}}\td(x,y)\dt=0
	   \end{align*}
	 for all $0< t_{1}<t_{2}<\infty$.
	 
	 Furthermore, if $u$ is also differentiable with $u_t(t) \in L^1$ for a.e.~$t > 0$, then we call this a \emph{strong distributional} solution in $L^1$.
    \end{definition}
	
	If, for example, $X=L^{q}_{\mu}$ for $1<q<\infty$ and $g \equiv 0$, then $X$ is a
	uniformly convex Banach space and so the classical regularity
	theory of nonlinear semigroups (cf.~\cite[Theorem~4.6]{MR2582280})
	applies: let $A$ be a quasi $m$-accretive operator on $L^{q}_{\mu}$, then for every $u_{0}\in
	D(A)$ the mild solution $u$ of Cauchy problem~\eqref{eq:9} is a
	strong solution of~\eqref{eq:9} and $t\mapsto T_{t}u_{0}$ given
	by~\eqref{eq:semigroup} satisfies
	\begin{equation}
	\label{eq:14}
	\tfrac{\td }{\td t}_{+}T_{t}u_{0}=-A^{\circ}T_{t}u_{0}\qquad\text{for
		every $t>0$,}
	\end{equation}
	where $A^{\circ}$ denotes the \emph{minimal selection} of $A$ given by the operator
	\begin{displaymath}
	A^{\circ}:=\Big\{(x,y)\in A\, \Big\vert\,\norm{y}=\min\limits_{\hat{y}\in
		Ax}\norm{\hat{y}}\Big\}.
	\end{displaymath}
	Under additional
	geometric conditions on the Banach space $X$, one has that
	$-A_{\circ}\subseteq A^{\circ}$. Ignoring these details on $X$, we
	nevertheless say that a strongly continuous semigroup
	$\{T_{t}\}_{t\ge 0}$ of quasi contractions on
	$\overline{D(A)}^{\mbox{}_{X}}$ is \emph{generated by $-A$} if $A$ is
	quasi $m$-accretive in $X$ and $\{T_{t}\}_{t\ge 0}$ is the family
	induced by~\eqref{eq:semigroup}.
	
	If $X$ is a Hilbert space $H$ with inner product $(\cdot,\cdot)_{H}$,
	then an important class of $m$-accretive operators in $H$ is
	given by the \emph{subdifferential operator}
	\begin{equation}
	\label{eq:subdifferentialOperator}
	\partial \E:=\Big\{(u,v)\in H\times
	H\,\Big\vert\;  (v, \xi-u)_{H}\le
	\E(\xi)-\E(u)\text{ for all }\xi\in H\Big\}
	\end{equation}
	of a proper, convex, lower semicontinuous 
	functional $\E : H\to (-\infty,+\infty]$. In
	Hilbert spaces, accretivity is equivalent to \emph{monotonicity}; that is, an operator $A$ is monotone if
	\begin{displaymath}
	(u-\hat{u},v-\hat{v})_{H}\ge 0\qquad\text{for all $(u,v)$,
		$(\hat{u},\hat{v})\in A$}
	\end{displaymath}
	(cf.~\cite{MR0348562}, and see
	also~\cite{MR3465809,MR4041276}). For this class of operators $A=\partial \E$, the Cauchy problem~\eqref{eq:9} has the smoothing effect that every mild solution $u$ of~\eqref{eq:9} is strong. This result is due to Brezis~\cite{MR0283635} (see
	also~\cite{MR4041276}).
	
	It is well known that the \emph{fractional $p$-Laplacian} $(-\Delta_{p})^{s}$ defined in~\eqref{eq:11} equipped with homogeneous
	Dirichlet boundary conditions on an open set $\Omega$ can be realized as such a
	subdifferential operator $\partial \E$ in $L^{2}$. To be more precise, let $\Omega \subseteq \R^d$, $d \ge 1$, be an open domain. Define the energy functional $\E : L^{2}\to [0,\infty]$ for $1 < p < \infty$ and $s \in (0,1)$ by
	\begin{equation}
	\label{eq:energyFunctionalpLap}
	\E(u) = 
	\begin{cases}
	\frac{1}{2p}[u]_{s,p}^{p} & \text{ if } u \in W^{s,p}_{0}\cap L^2,\\
	\infty & \text{ if } u \in L^{2}\setminus W^{s,p}_{0},
	\end{cases}
	\end{equation}
	for every $u\in L^{2}$, where we have defined the $s$-Gagliardo semi-norm by~\eqref{eq:8} on $\R^N$ as in the definition of $W^{s,p}_{0}$. It is immediate that this energy functional is convex,  lower semicontinuous and proper with effecive domain $D(\E) = W^{s,p}_{0}\cap L^2$ (cf.~\cite{MR3491533}, ~\cite{CoulHau2016}). Then we have the following characterization of the subdifferential $\partial \E$ (see \cite{MR3491533}).
	\begin{proposition}[Characterization of $(-\Delta_p)^s$]
		\label{prop:characterizeFracPLap}
		For $1 < p < \infty$ and $0 < s < 1$, let $\E$ be given by~\eqref{eq:energyFunctionalpLap}. Then for every $u \in W^{s,p}_{0}\cap L^2$,
		\begin{displaymath}
		\partial\E(u) = \left\lbrace h \in L^{2}\Biggr|
		\scalebox{0.9}{$\begin{aligned}
			&\int_{\Omega}h(x)v(x)\dx = \int_{\mathbb{R}^{d}}\int_{\mathbb{R}^{d}}\frac{(u(x)-u(y))^{p-1}(v(x)-v(y))}{|x-y|^{d+sp}}\\
			&\hspace{3.5cm}\text{ for all } v \in W^{s,p}_{0}\cap L^2
			\end{aligned}$}\right\rbrace.
		\end{displaymath}
	\end{proposition}
	We note that the case $p = 1$ can be characterized similarly with a formulation presented in \cite{MR3491533}, although we do not consider $p = 1$ in this article.

	%
	%
	%
	%

	\subsubsection{Completely accretive and $T$-accretive operators}
	\label{sec:completelyaccretive}
		
	The notion of \emph{completely accretive oper\-ators} was introduced in~\cite{MR1164641} by Crandall and B\'enilan and further developed in~\cite{CoulHau2016}. Following the same notation as in these two references, $\mathcal{J}_{0}$ denotes the set of all convex, lower semicontinuous functions $j : \R\to [0,+\infty]$ satisfying $j(0)=0$.
	
	\begin{definition}\label{def:complete-contraction}
		A mapping $S : D(S)\to M(\Sigma,\mu)$ with domain $D(S)\subseteq
		M(\Sigma,\mu)$ is called a \emph{complete contraction} if 
		\begin{displaymath}
		\int_{\Sigma}j(Su-S\hat{u})\dmu\le 
		\int_{\Sigma}j(u-\hat{u})\dmu
		\end{displaymath}
		for all $j\in \mathcal{J}_{0}$ and every $u$, $\hat{u}\in D(S)$. An operator $A$ on $M(\Sigma,\mu)$ is called \emph{completely
			accretive} if for every $\lambda>0$, the resolvent operator
		$J_{\lambda}$ of $A$ is a \emph{complete contraction}.
	\end{definition}
	
	It is well-known (see, e.g., \cite{MR3491533,CoulHau2016}) that the
	fractional $p$-Laplacian $(-\Delta_{p})^{s}$ equipped with
	Dirichlet boundary conditions is $m$-completely accretive in
	$L^{2}$. A more general class of operators is that of \emph{$T$-accretive} operators. In particular, choosing $j(\cdot)=\abs{[\,\cdot\,]^{+}}^{q}\in \mathcal{J}_{0}$ if
	$1\le q<\infty$ and $j(\cdot)=[[\,\cdot\,]^{+}-k]^{+}\in \mathcal{J}_{0}$
	for $k\ge 0$ large enough if $q=\infty$ shows that a complete
	contraction $S$ satisfies the following \emph{$T$-contractivity} property in $L^{q}_{\mu}$ for every
	$1\le q\le \infty$.
	\begin{definition}
		A mapping $S : D(S) \to L^{q}_{\mu}$ with domain
		$D(S)\subseteq L^{q}_{\mu}$, $1\le q\le \infty$, is called
		a \emph{$T$-contraction} if
		\begin{displaymath}
		\norm{[Su-S\hat{u}]^{+}}_{q}\le \norm{[u-\hat{u}]^{+}}_{q}
		\end{displaymath}
		for every $u$, $\hat{u}\in D(S)$.  We say that an operator $A$
		on $L^{q}_{\mu}$ is \emph{$T$-accretive} if, for every $\lambda>0$, the resolvent
		$J_{\lambda}$ of $A$ defines a $T$-contraction with domain
		$D(J_{\lambda})=Rg(I+\lambda A)$.
	\end{definition}
	By the previous remark, a completely accretive operator is $T$-accretive in $L^q_{\mu}$ for all $1 \le q \le \infty$. Furthermore, a $T$-accretive operator in $L^q_{\mu}$, $1 \le q \le \infty$, is accretive in $L^q_{\mu}$ and the resolvent is order-preserving in $L^q_{\mu}$. That is, denoting the usual order
	relation on $L^{q}_{\mu}$ by $\le$, if $S$ is $T$-contractive in $L^q_{\mu}$, $1 \le q \le \infty$, then $u \le \hat{u}$ implies that $S u \le S \hat{u}$ for $u$, $\hat{u}\in L^q$ (see
	\cite[Lemma 19.11]{BenilanNonlinearEvolutionEqns} for further properties).
	
	We can naturally extend these definitions to quasi $m$-completely accretive operators (and similarly for quasi $m$-$T$-accretive operators).
	\begin{definition}\label{def:completely-accretive}
		An operator $A$ on $M_{\mu}$ is called \emph{quasi completely accretive} if there is an $\omega > 0$
		such that for every $\lambda>0$, the resolvent operator
		$J_{\lambda}$ of $A+\omega I$ is a complete contraction. Moreover,
		for $1\le q<\infty$, an operator $A$ on $L^q_{\mu}$ is said to
		be \emph{quasi $m$-completely accretive} on $L^q_{\mu}$ if
		there is an $\omega > 0$ such that $A+\omega I$ is completely
		accretive and the range condition~\eqref{eq:range-condition} holds
		with $X = L^{q}_{\mu}$.
	\end{definition}
	
	%
	%
	%
	%

	\subsubsection{$T$-accretive operators in $L^{1}$ with complete resolvent}
	\label{sec:accretiveWithCompleteResolvent}
	
	In this last part of Section~\ref{sec:pre-ns}, we introduce the class of
	operators $A$ which are merely $T$-accretive in $L^1_{\mu}$ but
	have a so-called \emph{complete resolvent}. This class of operators was
	introduced in~\cite{BenilanHDR} and further elaborated
	in~\cite{CoulHau2016}. It's worth mentioning that for a given completely accretive operator $A$ in $L^1$, the composed operator $A\varphi$ becomes $T$-accretive in $L^1$ with complete resolvent so long as $\varphi$ is a strictly increasing, continuous function on $\R$ (see~\cite{CoulHau2016}). Typical examples of this class of operators include the doubly-nonlinear operators $-\Delta_{p}\varphi$ and $(-\Delta_{p})^{s}\varphi$. Hence we need to introduce the notion \emph{complete} mappings.
	\begin{definition}
		\label{def:complete-mapping}
		Let $D(S)$ be a subset of $M_{\mu}$. A mapping\\
		$S : D(S)\to M(\Sigma,\mu)$ is called \emph{complete} if
		\begin{equation}
		\label{eq:4}
		\int_{\Sigma}j(Su)\dmu\le 
		\int_{\Sigma}j(u)\dmu
		\end{equation}
		for every $j\in \mathcal{J}_{0}$ and $u\in D(S)$.
	\end{definition}
	
	We now introduce the class of accretive operators in $L^{1}_{\mu}$ with complete resolvent (similarly for $T$-accretive operators with complete resolvent).
	
	\begin{definition}
		An operator $A$ on $L^{1}_{\mu}$ is called \emph{($m$-)accretive in $L^{1}_{\mu}$ with complete resolvent} if $A$ is ($m$-)accretive in $L^{1}_{\mu}$ and for every $\lambda>0$, the resolvent operator $J_{\lambda} : Rg(I+\lambda A)\to D(A)$ of $A$ is a complete mapping. For $\omega\in \R$, we call an operator $A$ on $L^1_{\mu}$ \emph{$\omega$-quasi ($m$-)accretive in 	$L^{1}_{\mu}$ with complete resolvent} (or simply \emph{quasi	($m$-)accretive in $L^{1}_{\mu}$ with complete resolvent}) if	$A+\omega I$ is ($m$-)accretive in $L^{1}_{\mu}$ with complete resolvent.
	\end{definition}
	
	The condition~\eqref{eq:4} on the resolvent provides a growth estimate on $u(t)$, allowing us to estimate $u(t)$ in $L^p_{\mu}$ by the norms of $u_0$ and $g$ in $L^p_{\mu}$. In particular, we prove such an estimate in Lemma~\ref{le:GrowthConditionCompleteResolvent}, extending~\cite[Proposition 2.4]{BenilanHDR}.
	
	\subsubsection{The sub-differential operator in $X$}
	\label{sec:subgradientSection}
	We rely on the $m$-accretivity of $(-\Delta_p)^s$ in $L^1$ to obtain mild and strong solutions to~\eqref{eq:1}. We also find that it is sufficient to work with the part in $L^{1\cap\infty}$ to establish existence of such solutions. Hence we consider operators restricted to $L^{1\cap\infty}$ and introduce a definition of sub\-differential operators from \cite{MR1164641} which we will apply in the case $X = L^1$. We will also use this to apply the results of~\cite{BenilanGariepy} and in particular to obtain strong solutions. We include functionals on a (possibly distinct) subspace $Y$ for completeness.
	
	\begin{definition}\label{def:sub-diff-in-X}
		Let $X$, $Y$ and $Z$ be linear subspaces of $M(\Sigma,\mu)$. Then for an energy functional $\E : Y\to (-\infty,\infty]$ with effective domain $D(\E):=\{u\in Y\,\vert\,\E(u)<\infty\}$, we define the \emph{part of $\E$ in $X$} by
		\begin{displaymath}
		\E_{\vert_{X}}(u) = \begin{cases}
		\E(u)& \text{for } u \in D(\E)\cap X,\\
		\infty& \text{otherwise.}\\
		\end{cases}
		\end{displaymath} 
		Further, we define the operator	$\partial_{X}\E$ in $X$ by
		\begin{displaymath}
		\partial_{X}\E=\Bigg\{(u,v)\in X\times X\,\Bigg\vert
		\begin{array}[c]{l}
		u\in D(\E)\; \text{ and }\displaystyle \int_{\Sigma}v(w-u)\dmu\le \E(w)-\E(u)\\
		\text{ for all } w\in X\text{ with }v(w-u)\in
		L^{1}_{\mu}
		\end{array}
		\Bigg\}
		\end{displaymath}
		and the \emph{part of $\partial_{X}\E$ in $Z$} by
		\begin{displaymath}
		(\partial_{X}\E)_{\vert_{Z}} = \Big\lbrace (u,h) \in Z\times Z\, \Big| (u,h) \in \partial_{X}\E
		\Big\rbrace.
		\end{displaymath}
	\end{definition}
	
	In the case $X = L^2_{\mu}$ and $Y$ a linear subspace of $M(\Sigma,\mu)$, this coincides with the previous definition of the subdifferential $\partial\E$ in $L^2_{\mu}$ given by~\eqref{eq:subdifferentialOperator}. One sees that for a functional $\E:Y\rightarrow (-\infty,\infty]$ and $\tilde{\E}$ given by $\E_{\vert_{L^2_{\mu}}}$, that $\partial_{L^2_{\mu}} \E = \partial \tilde{\E}$.
	
	In the case $X = L^1_{\mu}$, $Y = L^2_{\mu}$, we have the inclusion $\left(\partial\E\right)_{\vert_{L^{1\cap\infty}_{\mu}}} \subseteq \left(\partial_{L^1_{\mu}}\E\right)_{\vert_{L^{1\cap\infty}_{\mu}}}$. However, since we are primarily interested in $m$-accretivity, they will be largely interchangeable due to the $m$-accretivity properties of $\left(\partial\E\right)_{\vert_{L^{1\cap\infty}_{\mu}}}$ proved in Theorem~\ref{thm:nonlinearpLapmTAccretive}. For the class of operators $\partial_{X}\E$, B\'enilan and
	Crandall~\cite{MR1164641} found sufficient conditions implying that the closure $\overline{\partial_{X}\E}^{\mbox{}_{X}}$ of
	$\partial_{X}\E$ in $X$ is $m$-completely accretive. We note that if $\E$ is lower semicontinuous then $\overline{\partial_{L^2_{\mu}}\E}^{\mbox{}_{L^2_{\mu}}} = \partial_{L^2_{\mu}}\E$.
	
	\begin{theorem}[{\cite[Lemma~7.1 \& Theorem~7.4]{MR1164641}}]
		\label{thm:m-compl-accretivity-of-E-in-Lq}
		Let $X$	be either $L^{r}_{\mu}$, $1\le r\le \infty$, or
		$L^{1\cap \infty}_{\mu}$. Then the following statements hold.
		\begin{enumerate}[(1)]
			\item If a functional $\E : X\to (-\infty,\infty]$ satisfies
			\begin{equation}
			\label{eq:12}
			\left\{
			\begin{array}[c]{l}
			\hspace{1cm}\E(u+q(\hat{u}-u))+\E(\hat{u}+q(\hat{u}-u))\le
			\E(u)+\E(\hat{u})\\[5pt]
			\text{for all $u$, $\hat{u}\in X$ and $q\in C^{\infty}(\R)$ such that $q(0) = 0$ and}\\[5pt]
			\text{$q'$ has compact support, satisfying $0 \le q' \le 1$ on $\R$}
			\end{array}
			\right.
			\end{equation}
			then the operator
			$\partial_{X}\E$ is completely accretive.
			\item If $\E : X\to [0,\infty]$ satisfies~\eqref{eq:12},
			$(0,0)\in \partial_{X}\E$, and if $\E$ is lower semicontinuous
			for the topology of $X+L^{2}_{\mu}$, then the closure
			$\overline{\partial_{X}\E}^{\mbox{}_{X}}$ of $\partial_{X}\E$ in
			$X$ is $m$-completely accretive in $X$.
		\end{enumerate}
	\end{theorem}
	
	Note that if $\E(0) = 0$ and $\E(u) \ge 0$ for all $u \in X$ with $\partial_{X}\E$ completely accretive, then the condition $(0,0) \in \partial \E$ of Theorem~\ref{thm:m-compl-accretivity-of-E-in-Lq} is satisfied.
	
	%
	%
	%
	%
	\subsection{The doubly nonlinear operator $(-\Delta_{p})^{s}\varphi$}
	\label{subsec:doubly-nonlinear-operator}
	Throughout this paper we focus on the composed operator $(-\Delta_p)^s\varphi = (-\Delta_p)^s\circ\varphi$. Hence, we introduce the composition operator on $L^1$ of the form $A\circ\varphi$ for $A$ an operator on $L^{1}_{\mu}$ and $\varphi$ a function on $\R$.
	\begin{definition}
		\label{def:compositionOperator}
		For an operator $A$ on $L^{1}_{\mu}$ and a function $\varphi$ on $\R$, we define the composed operator $A\circ\varphi$ in $L^{1}_{\mu}$ as a graph by
		\begin{displaymath}
		A\varphi = \{(u,v) \in L^1_{\mu} \times L^1_{\mu} : (\varphi(u),v) \in A\}.
		\end{displaymath}
		and interchangeably as a (possibly multi-valued) operator on $L^1_{\mu}$.
	\end{definition}
	
	 We may also extend the domain of an operator defined on $L^{1\cap\infty}_{\mu}$ in the following manner introduced by~\cite{BenilanGariepy}.
	\begin{definition}
		\label{def:extendedDomain}
		For an operator $A$ in $L^{1\cap\infty}_{\mu}\times L^{1\cap\infty}_{\mu}$, we extend the domain $D(A)$ by the set
		\begin{displaymath}
		\tilde{D}(A) := \left\lbrace u \in L^{1\cap\infty}_{\mu}\left|
		\begin{aligned}
		& \exists\, (u_{n},h_n)_{n\ge 1} \subseteq A \text{ such that}\\
		& u_{n} \rightarrow u \text{ in } L^{1}_{\mu} \text{ and }\\
		&(u_n,h_n)_{n \ge 1} \text{ is bounded in } L^{\infty}_{\mu}\times L^1_{\mu}.
		\end{aligned}\right.
		\right\rbrace 
		\end{displaymath}
        and for a strictly increasing $\varphi\in C(\R)$, we define
     \begin{displaymath}
        \hat{D}(A\varphi):=\Big\{
        u\in L^1_{\mu}\,\Big\vert\, 
        \varphi(u)\in \tilde{D}(A)\Big\}.
    \end{displaymath}
	\end{definition}
   
	Then, for an operator $A$ in $L^{1\cap\infty}_{\mu}\times L^{1\cap\infty}_{\mu}$, one has that 
    \begin{displaymath}
        D(A) \subseteq \tilde{D}(A) \subseteq L^{1}_{\mu}\qquad \text{and}\qquad
        D(A\varphi)\subseteq \hat{D}(A\varphi).
    \end{displaymath}
    
    Now, for $\Omega \subset \R^d$ and $\varphi(u_0) \in \tilde{D}\left(\left(\partial_{L^1}\E \right)\vert_{L^{1\cap\infty}}\right)$ we have the following theorem for existence of strong solutions from~\cite{BenilanGariepy}. 
    Note that here $\beta$ corresponds to $\varphi^{-1}$ in our setting and $v$ to $\varphi(u)$. 
    The statement in~\cite{BenilanGariepy} also uses a slightly different implementation of the sub-differential, closer to $\partial_{L^{\infty}}\E$ (giving potentially a larger operator), however this does not affect the proof. 
    A similar result for the homogeneous evolution problem can be found in \cite{CoulHau2016}.
	
	\begin{theorem}[{\cite[Theorem 4.1]{BenilanGariepy}}, Existence of strong distributional solutions]
		\label{th:BGStrongExistence}
		Let $\Omega$ be an open domain in $\R^d$, $d \ge 1$, of finite Lebesgue measure, and $T > 0$. Suppose $\E:L^{2} \rightarrow [0,\infty]$ is a lower semicontinuous function satisfying $\E(0) = 0$ and~\eqref{eq:12} for $X = L^{2}$. Further, let $\beta \in AC_{\loc}(\mathbb{R})$ be nondecreasing satisfying $\beta(\mathbb{R}) = \mathbb{R}$. Then, for every $v_{0} \in \tilde{D}(\left(\partial_{L^1}\E \right)_{\vert_{L^{1\cap\infty}}})$ and $f \in BV((0,T);L^{1})\cap L^{1}(0,T;L^{1})\cap L^{1}(0,T;L^{\infty})$, there exists $v \in L^{\infty}((0,T)\times \Omega)$ such that $u := \beta(v) \in W^{1,\infty}(0,T;L^{1})$ is the unique strong distributional solution to
		\begin{equation}
		\begin{cases}
		u'(t)+\left(\partial_{L^1}\E \right)_{\vert_{L^{1\cap\infty}}}v(t) \in f(t)& \text{ for a.e.~} t\in (0,T),\\
		u(0) = \beta(v_{0}).&
		\end{cases}
		\end{equation}
	\end{theorem}
	
	We also require the following chain rule.
	
	\begin{theorem}[{\cite[Theorem 1.1]{BenilanGariepy}}]
		Let $(\Sigma,\mu)$ be a measure space.\\ If $w \in W^{1,1}((0,T);L^{1}_{\mu})$, $p \in L_{\loc}^{1}(\R)$ and
		\begin{displaymath}
		u = \int_{0}^{w}p(r)\dr \in BV(0,T;L^{1}_{\mu})\cap L^{1}((0,T);L^{1}_{\mu})
		\end{displaymath}
		then $u \in W^{1,1}((0,T);L^{1}_{\mu})$ and for a.e.~$t \in (0,T)$
		\begin{displaymath}
		\frac{\td u}{\td t}(t) = p(w(t))\frac{\td w}{\td t}(t)
		\end{displaymath}
		for $\mu$-a.e.~$x \in \Sigma$.
	\end{theorem}

	With the above preliminaries we can now focus on proving our main results.
%
%
%
%

\section[Well-posedness \& comparison principle]{Well-posedness and a comparison principle}
\label{sec:well-posedness}

We now apply the nonlinear semigroup theory summarized in the previous section to the doubly nonlinear operator $(-\Delta_p)^{s}\varphi$ to prove Theorem~\ref{thm:1}. To do this we prove two key results for operators of the form $\overline{(\partial\E)_{\vert_{L^{1\cap\infty}_{\mu}}}\varphi}^{\mbox{}_{L^1}}+F$, which might be of independent interest.

\medskip

The first result gives us $T$-accretivity with complete resolvent for operators $A\varphi$ with a Lipschitz pertur\-bation. This follows from~\cite[Proposition 2.17 and Proposition 2.19]{CoulHau2016}. A comparable result can also be found for operators defined on finite measure spaces $(\Sigma,\mu)$ in~\cite[pg. 24]{BenilanGariepy} (see also~\cite[Lemma 7.1]{MR1164641}).

\begin{theorem}
\label{thm:nonlinearpLapmTAccretive}
	Suppose $(\Sigma,\mu)$ is a $\sigma$-finite measure space, $\E:L^2_{\mu} \rightarrow [0,\infty]$ is convex, lower semicontinuous with $\E(0) = 0$ and satisfying~\eqref{eq:12}. Let $\varphi\in C(\R)$ be strictly increasing with $\varphi(0) = 0$ and satisfying
	\begin{equation}
	\label{eq:Yosidaqbracket2}
	[\beta_\lambda(u), (\partial\E)_{\vert_{L^{1\cap\infty}_{\mu}}}u]_1 \ge 0
	\end{equation}
	and
	\begin{equation}
	\label{eq:Yosidaqbracket3}
	[\beta_\lambda(u), (\partial\E) _{\vert_{L^{1\cap\infty}_{\mu}}}u]_2 \ge 0
	\end{equation}
	for every $\lambda > 0$ and $u \in D((\partial\E)_{\vert_{L^{1\cap\infty}_{\mu}}})$, where $\beta = \varphi^{-1}$ and $\beta_\lambda$ is the Yosida approximation of $\beta$. Suppose $F:L^1_{\mu} \rightarrow L^1_{\mu}$ satisfies the Lipschitz property
	\begin{equation}
	\label{eq:LipschitzProperty}
	|F(u)-F(\hat{u})| \le \omega|u-\hat{u}| \quad \text{ on } \Sigma
	\end{equation}
	for all $u$, $\hat{u} \in L^1_{\mu}$ with constant $\omega \ge 0$ and satisfies $F(0) = 0$. Then $\overline{(\partial\E)_{\vert_{L^{1\cap\infty}_{\mu}}}\varphi}^{\mbox{}_{L^1_{\mu}}}+F$ is $\omega$-quasi $m$-$T$-accretive in $L^1_{\mu}$ with complete resolvent.
\end{theorem}

\begin{proof}
	Since $\E$ is convex and attains its global minimum at $0$, we have that $(0,0) \in \partial\E$. Then by the complete accretivity property, $\partial\E$ has complete resolvent. Hence $\partial \E$ is $m$-completely accretive in $L^2$ with complete resolvent so that $(\partial\E)_{\vert_{L^{1\cap\infty}_{\mu}}}$ is also completely accretive. Since $\varphi$ is injective, we have by~\cite[Proposition 2.17]{CoulHau2016} that $(\partial\E) _{\vert_{L^{1\cap\infty}_{\mu}}}\varphi$ is $T$-accretive in $L^1_{\mu}$ with complete resolvent. So by~\cite[p.31]{CoulHau2016} and~\cite[Proposition 2.12]{CoulHau2016}, the operator $(\partial\E) _{\vert_{L^{1\cap\infty}_{\mu}}}\varphi+F$ is $\omega$-quasi $T$-accretive in $L^1$ with complete resolvent. Then we can apply~\cite[Proposition 2.19]{CoulHau2016} to $\partial\E$ to obtain the range condition for the closure and hence $\omega$-quasi  $m$-$T$-accretivity.
\end{proof}

We now apply the above theorem to the doubly nonlinear operator $(-\Delta_p)^s \varphi$.
\begin{corollary}
	\label{cor:doublynonlinearAccretivity}
	Let $\Omega$ be an open domain in $\R^d$, $d \ge 1$. Suppose $F:L^1 \rightarrow L^1$ satisfies the Lipschitz property~\eqref{eq:LipschitzProperty} for all $u$, $\hat{u} \in L^1$ with constant $\omega \ge 0$ and satisfies $F(0) = 0$. Let $\E$ be the energy functional given by~\eqref{eq:energyFunctionalpLap} associated to the $s$-Gagliardo semi-norm. Then the following statements hold.
    \begin{enumerate}
        \item The sub-differential operator $\partial\E$ in $L^2$ of $\E$ is $m$-completely accretive on $L^2$.
        \item The part $(\partial\E)_{\vert L^{1\cap\infty}}$ of $\partial\E$ in $L^{1\cap\infty}\times L^{1\cap\infty}$ satisfies \eqref{eq:Yosidaqbracket2} and~\eqref{eq:Yosidaqbracket3}.
        \item The operator $\overline{(\partial\E)_{\vert_{L^{1\cap\infty}}}\varphi}^{\mbox{}_{L^1}}+F$ is $\omega$-quasi $m$-$T$-accretive in $L^1$ with complete resolvent.
    \end{enumerate}
\end{corollary}

\begin{proof}
	We first show that $\E$ satisfies~\eqref{eq:12} and so, the sub-differential operator $\partial\E$ of $\E$ on $L^2$ is $m$-completely accretive on $L^2$ owing to Theorem~\ref{thm:m-compl-accretivity-of-E-in-Lq}. Obviously, it is sufficient to show that $\E$ satisfies~\eqref{eq:12} for every $v$ and $\hat{v}\in D(\E)$. For given $x$, $y\in \R^{d}$, set $a = v(x)-v(y)$ and $b = \hat{v}(x)-\hat{v}(y)$, and let $q \in C^1(\R)$ satisfying $q(0) = 0$ and $0 \le q' \le 1$. Since $0 \le q' \le 1$, there is a $k \in [0,1]$ such that $q(b-a) = k(b-a)$. Then, by the convexity of $|\cdot|^p$, one has that 
	\begin{displaymath}
	\left|ka+(1-k)b\right|^p+\left|(1-k)a+kb\right|^p \le |a|^p+|b|^p
	\end{displaymath}
	holds and so, we have
  \begin{displaymath}
      \left|b-q(b-a))\right|^p+\left|a+q(b-a)\right|^p \le |a|^p+|b|^p
  \end{displaymath}
or, equivalently,
\begin{align*}
  &\left|\hat{v}(x)-\hat{v}(y)-q(\hat{v}(x)-\hat{v}(y)-
  v(x)-v(y)))\right|^p\\
  &\qquad\qquad +\left|v(x)-v(y)+q(\hat{v}(x)-\hat{v}(y)-v(x)-v(y))\right|^p\\ 
  &\qquad\le |v(x)-v(y)|^p+|\hat{v}(x)-\hat{v}(y)|^p.  
\end{align*}
Therefore~\eqref{eq:12} follows from integrating over $\R^{2d}$ with respect to $|x-y|^{-d-sp}\dx\dy$. 

Further, since for every $\lambda>0$ and $q\ge 1$, $u\mapsto\left(\beta_\lambda(\cdot)\right)^{q-1}$ is monotone increasing, differentiable with a bounded derivative, it follows from the characterization of the sub-differential operator $\partial\E$ given in Proposition~\ref{prop:characterizeFracPLap} that he part $(\partial\E)_{\vert L^{1\cap\infty}}$ of $\partial\E$ in $L^{1\cap\infty}\times L^{1\cap\infty}$ satisfies \eqref{eq:Yosidaqbracket2} and~\eqref{eq:Yosidaqbracket3}.
Hence we may apply Theorem~\ref{thm:nonlinearpLapmTAccretive}, to conclude that the operator $\overline{(\partial\E)_{\vert_{L^{1\cap\infty}}}\varphi}^{\mbox{}_{L^1}}+F$ is $\omega$-quasi $m$-$T$-accretive in $L^1$ with complete resolvent. 
\end{proof}

In order to apply our regularity results to all initial data in $L^1$, we use the following density result for the composition operator $(\partial\E)_{\vert_{L^{1\cap\infty}_{\mu}}}\varphi$. This generalizes the classic density	result for sub-differential operators (cf.~\cite{MR0348562} or \cite[Propos\-ition~1.6]{MR2582280}) and in particular generalizes an idea from~\cite[p.48]{MR2582280}. We note that~\eqref{eq:Yosidaqbracket2} and~\eqref{eq:Yosidaqbracket3} are satisfied for the fractional $p$-Laplacian.

\begin{theorem}[{Density of $D((\partial\E) _{\vert_{L^{1\cap\infty}_{\mu}}}\varphi)$ in $L^1_{\mu}$}]
	\label{prop:L1densityofE}
	Let $(\Sigma,\mu)$ be a $\sigma$-finite meas\-ure space, $\E : M(\Sigma,\mu) \rightarrow [0,\infty]$ a proper, convex functional satisfying $\E(0)=0$, and suppose the restriction $\E_{\vert L^2_{\mu}}$ of $\E$ on $L^2_{\mu}$ is lower semi\-continuous on $L^2_{\mu}$ and the restriction $\E_{\vert L^{1\cap\infty}_{\mu}}$ of $\E$ on $L^{1\cap \infty}_{\mu}$ satisfies~\eqref{eq:12}.  Further, let $\varphi\in C(\R)$ be a strictly increasing function such that $\varphi(\R) = \R$, $\varphi(0) = 0$ and the Yosida approximation $\beta_\lambda$ of $\varphi^{-1}$ satisfies~\eqref{eq:Yosidaqbracket2} and~\eqref{eq:Yosidaqbracket3} for every $\lambda > 0$. Then the following statements hold. 
    \begin{enumerate}
        \item The domain $D(\partial \E_{\vert_{L^{1\cap\infty}_{\mu}}}\varphi)$ of the composed operator $\partial \E_{\vert_{L^{1\cap\infty}_{\mu}}}\varphi$ is dense in the closure $\overline{D(\E_{\vert L^{1\cap\infty}_{\mu}} \varphi)}^{\mbox{}_{L^1_{\mu}}}$ of $D(\E_{\vert L^{1\cap \infty}_{\mu}}\varphi)$ with respect to the $L^1_{\mu}$-norm topology.
        \item If the set $D(\E_{\vert L^{1\cap \infty}_{\mu}} \varphi)$ is dense in $L^{1}_{\mu}$, then $D((\partial\E) _{\vert_{L^{1\cap\infty}_{\mu}}}\varphi)$ is dense in $L^{1}_{\mu}$.
    \end{enumerate}
\end{theorem}

\begin{proof}[Proof of Theorem~\ref{prop:L1densityofE}]
     Under the hypotheses of this theorem, one can apply \cite[Lemma A.3.1]{CoulHau2016} to
     $A=\partial\E_{\vert_{L^{1\cap\infty}_{\mu}}}$ and obtains that for every $\lambda>0$,
     every $\varepsilon > 0$ sufficiently small, and every $u \in D(\E_{\vert L^{1\cap
     \infty}_{\mu}}\circ \varphi)$, there is a unique $u_\lambda \in D(\partial
     \E_{\vert_{L^{1\cap\infty}_{\mu}}}\circ \varphi)$ satisfying
    \begin{displaymath}
     	u_{\lambda}+\lambda\left(\varepsilon \varphi(u_\lambda)+ 
    \partial\E_{\vert_{L^{1\cap\infty}_{\mu}}}\varphi (u_\lambda)\right) \ni u,        
    \end{displaymath}
    or equivalently, there exists $v_{\lambda}\in \partial\E_{\vert_{L^{1\cap\infty}_{\mu}}}\varphi (u_\lambda)$ such that
    \begin{equation}
        \label{eq:rangeProblemAphi}
	   u_{\lambda}+\lambda\left(\varepsilon \varphi(u_\lambda)+ 
        v_\lambda \right)= u
	\end{equation}
    Multiplying \eqref{eq:rangeProblemAphi} by $\varphi(u_\lambda)-\varphi(u)$ gives
	\begin{displaymath}
	\begin{split}
	\big (u_\lambda-u,\varphi(u_\lambda)-\varphi(u) \big)_{L^{2}_{\mu}} =& -\lambda\varepsilon \big(\varphi(u_\lambda),\varphi(u_\lambda)-\varphi(u) \big)_{L^{2}_{\mu}}\\
	&\hspace{2cm} -\lambda\big( v_\lambda,\varphi(u_\lambda)-\varphi(u)\big)_{L^{2}_{\mu}}.
	\end{split}
	\end{displaymath}
	The first term on the right-hand side can be estimates by
	\begin{displaymath}
	\begin{split}
	-\big(\varphi(u_\lambda),\varphi(u_\lambda)-\varphi(u) \big)_{L^{2}_{\mu}}& \le \norm{\varphi(u_\lambda)}_\infty\norm{\varphi(u)}_{1}\\
	& \le \sup_{[-\norm{u}_\infty,\norm{u}_\infty]}\norm{\varphi}_\infty\norm{\varphi(u)}_{1},
	\end{split}
	\end{displaymath}
	and since $v_{\lambda}\in \partial\E_{\vert_{L^{1\cap\infty}_{\mu}}}\varphi (u_\lambda)$ and noting that $\E(\varphi(u_\lambda)) \ge 0$, it follows that
	\begin{displaymath}
	\begin{split}
	-\big( v_{\lambda},\varphi(u_\lambda)-\varphi(u)\big)_{L^{2}_{\mu}}& =  -\big( (\partial\E)\varphi(u_\lambda),\varphi(u_\lambda)-\varphi(u)\big)_{L^{2}_{\mu}}\\
	& \le -\left(\E(\varphi(u_\lambda))-\E(\varphi(u))\right)\\
	& \le \E(\varphi(u))
	\end{split}
	\end{displaymath}
	for all $\lambda > 0$. Thus and since $\varphi$ is increasing, we have shown that
	\begin{displaymath}
	0\le (u_{\lambda}-u,\varphi(u_{\lambda})-\varphi(u))_{L^{2}_{\mu}}\le 
	\lambda\, \E(\varphi(u))+\lambda\varepsilon\sup_{[-\norm{u}_\infty,\norm{u}_\infty]}\norm{\varphi}_{\infty}\norm{\varphi(u)}_{1}
	\end{displaymath}
	for all $\lambda>0$, from where we can conclude that
	\begin{displaymath}
	\lim_{\lambda\rightarrow 0^{+}}\int_{\Sigma}(u_{\lambda}-u)(\varphi(u_{\lambda})-\varphi(u))\dmu = 0.
	\end{displaymath}
	Since
	\begin{displaymath}
	f_{\lambda}(x):=(u_{\lambda}(x)-u(x))(\varphi(u_{\lambda}(x)-\varphi(u(x))) \ge 0 \quad \text{$\mu$-a.e.~on $\Sigma$,}
	\end{displaymath}
	the letter limit means that $f_{\lambda} \rightarrow 0$ in
	$L^{1}_{\mu}$. After possibly passing to a subsequence, we
	know that
	\begin{displaymath}
	\lim_{\lambda\rightarrow 0^{+}}f_{\lambda}(x) = 0 \qquad \text{for $\mu$-a.e.~$x \in
    \Sigma$,}
	\end{displaymath}
	which due to the strict monotonicity of $\varphi$ implies that
	\begin{equation}
        \label{eq:pointwise-mu-conv}
	\lim_{\lambda\rightarrow 0^{+}}u_{\lambda}(x) = u(x) \qquad\text{for $\mu$-a.e.~$x \in
    \Sigma$.}
	\end{equation}

    We recall from \cite[Lemma 2.2.1]{CoulHau2016} that if $u_{\lambda}\ge 0$ for all $\lambda>0$, then the $\mu$-pointwise limit \eqref{eq:pointwise-mu-conv} together with the fact that $u_{\lambda}$ satisfies
	\begin{equation}
       \label{eq:L1-increasingness}
	\norm{u_\lambda}_1 \le \norm{u}_1\qquad\text{for all $\lambda>0$,}
	\end{equation}
    implies that $u_{\lambda} \rightarrow u$ in $L^{1}_{\mu}$.

    Now, let $u_{+,\lambda}:=J_\lambda^{\varepsilon\varphi_1+A\varphi}(u^{+})$, where
    $u^{+}=u\vee 0$ is the positive part of $u$. Since $\varepsilon\varphi_1+A\varphi$ is
    $T$-accretive in $L^1_{\mu}$ (cf. \cite[Proposition 2.3.6]{CoulHau2016}), one has that    
    $u_{+,\lambda}\ge 0$ for all $\lambda>0$. Moreover, by the above argument, $u_{+,\lambda}$
    satisfies \eqref{eq:pointwise-mu-conv} and \eqref{eq:L1-increasingness}. Therefore one has
    that $u_{+,\lambda} \rightarrow u^{+}$ in $L^{1}_{\mu}$. Next, let $u^{-}=(-u)\vee 0$ be
    the negative part of $u$ and set $u_{-,\lambda}:=
    J_\lambda^{\varepsilon\varphi_1+A\varphi}(-u^{-})$. Then, one also has that $u_{-,\lambda}$
    satisfies \eqref{eq:pointwise-mu-conv} and so, in particular, $-u_{-,\lambda}$ satisfies
    \eqref{eq:pointwise-mu-conv}. Since $-u_{-,\lambda}$ is positive and satisfies 
    \eqref{eq:L1-increasingness}, it follows that $-u_{-,\lambda} \rightarrow -u^{-}$ in
    $L^{1}_{\mu}$. Moreover, for $u_{\lambda}:=J_\lambda^{\varepsilon\varphi_1+A\varphi}u$, one
    has that
    \begin{displaymath}
     -u_{-,\lambda}\le u_{\lambda}\le u_{+,\lambda}\qquad
     \text{for every $\lambda>0$.}      
    \end{displaymath}    
    From this sandwich inequality and since $u_{+,\lambda} \rightarrow u^{+}$ in $L^{1}_{\mu}$
    and $-u_{-,\lambda} \rightarrow -u^{-}$ in $L^{1}_{\mu}$, one can extract from every zero sequence $(\lambda_{n})_{n\ge 0}$ a subsequence $(\lambda_{k_{n}})_{n\ge 1}$ and finds a positive function $g\in L^{1}_{\mu}$ such that $\abs{u_{\lambda_{k_{n}}}}\le g$ $\mu$-a.e. on $\Sigma$ for all
    $n\ge 1$. Thus, by Lebesgue's dominated convergence theorem, it follows that $u_{\lambda_{k_{n}}}\to u$ in $L^{1}_{\mu}$ as $n\to\infty$ and, thereby we have shown that
    the domain $D(\partial \E_{\vert_{L^{1\cap\infty}_{\mu}}}\circ \varphi)$ lies dense in the closure $\overline{D(\E_{\vert L^{1\cap\infty}_{\mu}}\circ \varphi)}^{\mbox{}_{L^1_{\mu}}}$
    of $D(\E_{\vert L^{1\cap \infty}_{\mu}}\circ \varphi)$ with respect to the $L^1_{\mu}$-norm topology. Thus, if the domain $D(\E_{\vert L^{1\cap \infty}_{\mu}}\circ \varphi)$ lies dense in $L^1_{\mu}$ then $D(\partial \E_{\vert_{L^{1\cap\infty}_{\mu}}}\circ \varphi)$ lies dense in $L^1_{\mu}$.
\end{proof}

We can now apply the previous results to the operator $\partial
     \E_{\vert_{L^{1\cap\infty}_{\mu}}}\circ \varphi$ in the case of the fractional $p$-Laplacian $\partial \E=(-\Delta_{p})^s$ equipped with homogeneous Dirichlet boundary conditions, obtaining $m$-$T$-accretivity in $L^1$ and well-posedness of mild solut\-ions.

\begin{proof}[Proof of Theorem~\ref{thm:1}]
	 Let $\E$ be given by~\eqref{eq:energyFunctionalpLap} and $\varphi$, $f$
  satisfy~\eqref{hyp:1} and~\eqref{hyp:2}-\eqref{hyp:3}, respectively. Then, it follows from Corollary~\ref{cor:doublynonlinearAccretivity} that $\partial_{L^2}\E$ is $m$-completely accretive and $(\partial\E)_{\vert L^{1\cap\infty}}$ satisfies \eqref{eq:Yosidaqbracket2} and~\eqref{eq:Yosidaqbracket3} with respect to the Yosida approximation $\beta_\lambda$ of $\varphi^{-1}$. Further, under the hypothesis $\varphi\in W^{1,\infty}_{loc}(\R)$, one has that
  \begin{displaymath}
      [\varphi(\xi)]_{s,p}\le \norm{\varphi'}_{L^{\infty}(-\norm{\xi}_{\infty},\norm{\xi}_{\infty})}\,[\xi]_{s,p}
  \end{displaymath}
 for every $\xi\in C^{\infty}_{c}(\Omega)$, and if $\phi \in W^{1,q}_{\loc}(\R)$ for $q > 1/(1-s)$, then H\"older's inequality yields that
\begin{align*}
    &\int_{\R^d}\int_{\R^d}\frac{|\phi(\xi(x))-\phi(\xi(y))|^p}{|x-y|^{d+sp}}\dy\dx\\
    & \qquad = \int_{\R^d}\int_{\R^d}
    \frac{|\int_{\xi(y)}^{\xi(x)}\phi'(r)\dr|^p}{|x-y|^{d+sp}}\dy\dx\\
    & \qquad\le \int_{\R^d}\int_{\R^d}\frac{\left|\left(\int_{\xi(y)}^{\xi(x)}|\phi'(r)|^q\dr\right)^{\frac{p}{q}}\right||\xi(x)-\xi(y)|^{\frac{p}{q'}}}{|x-y|^{d+sp}}\dy\dx\\
    & \qquad \le \norm{\phi'}_{L^q(-\norm{\xi}_{\infty},\norm{\xi}_{\infty})}^p
    \int_{\R^d}\int_{\R^d}\frac{|\xi(x)-\xi(y)|^{\frac{p}{q'}}}{|x-y|^{d+sp}}\dy\dx\\
    & \qquad= \norm{\phi'}_{L^q(-\norm{\xi}_{\infty},\norm{\xi}_{\infty})}^p\,[\xi]_{sq',\frac{p}{q'}}
\end{align*}
for every $\xi\in C^{\infty}_{c}(\Omega)$. In the last estimate, we note that $q>1/(1-s)$ is equivalent to $0<sq'<1$ and hence, $[\xi]_{sq',\frac{p}{q'}}$ is finite. Thus, under both each condition $\varphi\in W^{1,\infty}_{loc}(\R)$ or $\phi \in W^{1,q}_{\loc}(\R)$ for $q > 1/(1-s)$, one has that the set $C^{\infty}_{c}(\Omega)$ is contained in $D(\E_{\vert L^{1\cap\infty}}\varphi)$ and dense in $L^1$.  Thus Theorem~\ref{prop:L1densityofE} implies that under those conditions on $\varphi$, the domain $D((\partial\E)_{\vert_{L^{1\cap\infty}}}\varphi)$ is dense in $L^1$.
  
 Since by Corollary~\ref{cor:doublynonlinearAccretivity}, the operator $\overline{(\partial\E)_{\vert_{L^{1\cap\infty}}}\varphi}^{\mbox{}_{L^1}}+F$ is $m$-$T$
  accretive in $L^1$ with complete resolvent, it follows from standard semigroup theory (e.g.~\cite[Corollary 4.2]{MR2582280}) that for every $u_0 \in \overline{D(\E_{\vert L^{1\cap \infty}}\circ\varphi)}^{\mbox{}_{L^1}}$, there exists a unique mild solution $u$ to problem~\eqref{eq:1}. Moreover this mild solutions satisfies growth estimate~\eqref{eq:theorem1.1estimate1} (see also Lemma~\ref{le:GrowthConditionCompleteResolvent}) and~\eqref{eq:theorem1.1estimate2} for $\nu = 1$. The case $\nu = +$ follows in the same way, applying the $T$-contractivity condition of the resolvent~\cite{BenilanNonlinearEvolutionEqns}. This completes the proof of this theorem.
\end{proof}

%
%
%
%

\section{Extrapolation toward $L^{\infty}$}
\label{sec:extrapolation-to-infty}

As in the previous section, we again state these results for Lebesgue spaces $L^{q}_{\mu}$ with $1 \le q \le \infty$. We note that the following
theorem generalizes~\cite[Theorem~2.1]{MR2529737} and~\cite[Theorem~1]{MR2680400}. While these two theorems
in~\cite{MR2529737,MR2680400} are restricted to derive
$L^{q}_{\mu}-L^{\infty}_{\mu}$ regularity estimates of solutions of parabolic
diffusion problems with homogeneous forcing terms $g\equiv 0$ and
without Lipschitz perturbations (see also~\cite{MR2819280}), the following results can also treat evolution problems involving
Lipschitz continuous non\-linearities and $L^{\infty}_{\mu}$-bounded forcing terms. In particular we consider mild solutions $u$ to
\begin{equation}
\label{eq:section3equation}
\begin{cases}
\begin{alignedat}{2}
u'(t)+Au(t)& = g(t)\quad&& \text{for a.e.~} t\in (0,T),\\
u(0)& = u_{0},\quad&&
\end{alignedat}
\end{cases}
\end{equation}
where $A$ is quasi $m$-accretive in $L^{q_0}_{\mu}$ with complete resolvent for some $1 \le q_0 < \infty$.

For $\lambda\ge 0$, we define the signed truncator $G_{\lambda}(s):=[\abs{s}-\lambda]^{+}\sign(s)$ for every $s\in
\R$ and we set $0<T\le \infty$. Note that $q-1 \in [0,\infty)$ so we are using the notation $\norm{\cdot}_{q-1}$ to denote the usual Lebesgue integral even when this is not a norm. When $q-1 = 0$ this is given by
\begin{displaymath}
\norm{u}_{0} = \int_{\Sigma}\text{sign}_{0}\left(\left|u\right|\right)\dmu.
\end{displaymath}
The condition~\eqref{hyp:31} for the forcing term can always be satisfied by choosing $\rho = \infty$ and $\psi = \infty$.
\begin{theorem}
  \label{thm:LqLinfty-deGiorgi}
  For $1\le q< r\le \infty$, $1\le \sigma<r$, $q \le \rho \le \infty$
  and $1 < \psi \le \infty$ satisfying
  \begin{equation}
  \label{hyp:31}
  \begin{cases}
    \frac{1}{\rho} < \left(1-\frac{1}{\psi}\right)\left(1-\frac{\sigma}{r}\right)& \text{if } \sigma \ge q,\\
    \frac{1}{\rho} \le \left(1-\frac{1}{\psi}\right)\left(\frac{\sigma}{q}-\frac{\sigma}{r}\right)& \text{if } \sigma < q,
  \end{cases}
  \end{equation}
  and
  \begin{equation}
  \label{hyp:32}
    \frac{1}{\rho} \le \frac{1}{q}-\frac{\sigma}{r}\left(1-\frac{1}{\psi}\right)
  \end{equation}
  let $g \in L^{\psi}(0,T;L^{\rho}_{\mu})\cap L^{1}(0,T;L^{q}_{\mu})$ and
  $u_{0} \in L^{q}_{\mu}$. Suppose $u\in C([0,T];L^{q}_{\mu})$ satisfies
  $u(0) = u_{0}$ and for some $L>0$, $\omega\ge 0$ and every
  $\lambda\ge 0$, the ``level set energy inequality''
  \begin{equation}
    \label{eq:15}
    \begin{split}
     &\norm{G_{\lambda}(e^{-\omega  t_{2}}u(t_{2}))}_{q}^{q}+L\int_{t_{1}}^{t_{2}}e^{\omega (\sigma-q)s}\norm{G_{\lambda}(e^{-\omega s}u(s))}_{r}^{\sigma}\ds\\
     &\hspace{2cm}\le \norm{G_{\lambda}(e^{-\omega t_{1}}u(t_{1}))}_{q}^{q}+ q\lambda\,\omega\int_{t_{1}}^{t_{2}}\norm{G_{\lambda}(e^{-\omega s}u(s))}_{q-1}^{q-1}\ds\\
  	&\hspace{2cm}\quad +q\int_{t_{1}}^{t_{2}}e^{-\omega s}\left|[G_{\lambda}(e^{-\omega s}u),g(s)]_{q}\right|\ds
    \end{split}
  \end{equation}
  holds for all $0\le t_{1}<t_{2}\le T$. Further, assume $t\mapsto G_{\lambda}(e^{-\omega
  	t}u(t)) $ satisfies the following growth estimate in the $L^{q}$-norm,
  \begin{equation}
    \label{eq:27}
    \begin{split}
    \norm{G_{\lambda}(e^{-\omega t}u(t))}_{q} \le & \lnorm{G_{\lambda}(e^{-\omega s}u(s))}_{q}\\
    &+\int_{s}^{t}e^{-\omega \tau}\norm{g(\tau)\mathds{1}_{\set{e^{-\omega \tau}|u(\tau)| > \lambda}}}_{q}\dtau
    \end{split}
  \end{equation}
 for all $0\le s<t\le T$. Then there exists $C > 0$ depending on $\sigma$, $q$, $r$, $\rho$, $\psi$ and $L$ such that
 \begin{equation}
 \label{eqn:lqlinfty}
 \begin{split}
 \norm{u(t)}_{\infty}&\le C\max\Biggr(e^{\omega\beta_1 t}\left(\tfrac{1}{t}+
 \omega\right)^{\frac{1}{\sigma(1-\frac{q}{r})}}\left(\norm{u_0}_q+\norm{g}_{L^1(0,t;L^q)}\right)^{\gamma},\\
 &\hspace{1.6cm} e^{\omega\beta_2 t}\,\norm{g}_{L^{\psi}(0,T;L^{\rho})}^{\eta}\left(\norm{u_0}_q+\norm{g}_{L^1(0,t;L^q)}\right)^{\gamma_{\psi}}\Biggr)
 \end{split}
 \end{equation}
 for all $t \in (0,T]$ with the exponents
 \begin{equation}
 \label{eq:thm31exponents}
 \begin{split}
 \gamma& = \frac{\frac{1}{\sigma}-\frac{1}{r}}{\frac{1}{q}-\frac{1}{r}},\quad
 \gamma_{\psi} = \tfrac{\left(1-\frac{1}{\psi}\right)\left(\frac{1}{\sigma}-\frac{1}{r}\right)-\frac{1}{\rho\sigma}}{\left(1-\frac{1}{\psi}\right)\left(\frac{1}{q}-\frac{1}{r}\right)-\frac{1}{\rho\sigma}+\frac{1}{q\sigma}},\quad
 \eta = \tfrac{\frac{1}{q\sigma}}{ \left(1-\frac{1}{\psi}\right)\left(\frac{1}{q}-\frac{1}{r}\right)-\frac{1}{\rho\sigma}+\frac{1}{q\sigma}},\quad \\
 \beta_1& = \begin{cases}
 \frac{\frac{1}{\sigma}-\frac{1}{q}}{\frac{1}{q}-\frac{1}{r}}& \text{if } \sigma < q,\\
 0& \text{if } \sigma \ge q,
 \end{cases}\quad
 \beta_2 = \begin{cases}
 \eta(q-\sigma)(1-\frac{1}{\psi})& \text{if } \sigma < q,\\
 0& \text{if } \sigma \ge q.
 \end{cases}
 \end{split}
 \end{equation}
\end{theorem}

Our proof of Theorem~\ref{thm:LqLinfty-deGiorgi} is based on a
DeGiorgi iteration inspired by \cite{MR2529737} and \cite{MR2680400}. For this, we modify \cite[Chapter~2.5, Lemma~5.6]{MR0241822} to prove convergence of the following recurrence relation.
\begin{lemma}
	\label{lem:ladyMod}
	Let $b \ge 1$, $0 < f < 1$ and $M \in \N\setminus\set{0}$. Suppose a sequence $(y_{k})_{k\ge 0}$ in $[0,\infty)$ satisfies the recursion relation
	\begin{displaymath}
	y_{k+1}\le b^{k}\,\sum_{i=1}^{M}c_{i}\,y_{k}^{1+\delta_{i}}\qquad\text{for all $k \in \N$}
	\end{displaymath}
	where $c_i$, $\delta_i$ are positive constants for all $i \in \set{1,...,M}$. 
    Choose $\delta_m = \min_{i \in \set{1,...,M}}\delta_i$ and $C = \min_{i\in\set{1,...,M}}\left(c_{i}^{-\frac{1}{\delta_{i}}}\right)$.
    If 
    \begin{displaymath}
        y_{0}\le \frac{C}{M} b^{-\frac{1}{\delta_m^{2}}},
    \end{displaymath}
    then
	\begin{equation}
	\label{eq:recurrenceEstimate3.2}
	y_{k} \le \frac{C}{M} b^{-\frac{1}{\delta_m^2}}b^{-\frac{k}{\delta_m}}\qquad\text{for all $k \in \N$.}
	\end{equation}
	In particular, if $b > 1$ then $y_{k}\to 0$ as $k\to \infty$.
\end{lemma}

\begin{proof}
	Estimate~\eqref{eq:recurrenceEstimate3.2} follows via induction with
	\begin{displaymath}
	\begin{split}
	y_{k+1}& \le b^{k}\,\sum_{i=1}^{M}c_{i}\,y_{k}^{1+\delta_{i}}\\
	& \le \frac{C b^{k}}{M}\left(b^{-\frac{1}{\delta_m^2}}b^{-\frac{k}{\delta_m}}\right)^{1+\delta_m}\,\sum_{i=1}^{M}c_{i}\,\left(\frac{C}{M}\right)^{\delta_{i}}\\
	& \le \frac{C}{M} b^{-\frac{1}{\delta_{m}^{2}}}b^{-\frac{k+1}{\delta_{m}}}.
	\end{split}
	\end{displaymath}
\end{proof}

\begin{proof}[Proof of Theorem~\ref{thm:LqLinfty-deGiorgi}]
  By~\eqref{eq:15},
  one sees that $G_{\lambda}(u)\in L^{\sigma}(0,T;L^{r}_{\mu})$ for every
  $\lambda\ge 0$.
  Let $\lambda\ge 0$ and $t\in (0,T]$ and for
  every integer $k\ge 0$, set
  \begin{displaymath}
    t_{k}=t(1-2^{-k}),\quad \lambda_{k}=\lambda(1-2^{-k}), \quad G_{k}(\cdot)=G_{\lambda_{k}}(\cdot),
  \end{displaymath}
  and
  \begin{displaymath}
    U_{k}=\sup_{\hat{s}\in [t_{k},t]}\norm{G_{k}(e^{-\omega
    \hat{s}}u(\hat{s}))}_{q}^{q}+ L\int_{t_{k}}^{t}e^{\omega (\sigma-q)s}\norm{G_{k}(e^{-\omega s}u(s))}_{r}^{\sigma}
      \ds.
  \end{displaymath}
  Then the aim is to choose $\lambda\ge 0$ such that
  $U_{k}\to 0$ as $k\to \infty$.
  By the continuity of
  $t\mapsto \norm{G_{k}(e^{-\omega t}\,u(t))}_{q}^{q}$, there is an
  $s_{k}\in (t_{k-1},t_{k})$ satisfying
  \begin{equation}
    \label{eq:29}
    \norm{G_{k}(e^{-\omega
    		s_{k}}\,u(s_{k}))}_{q}^{q}=\tfrac{2^{k}}{t} \int_{t_{k-1}}^{t_{k}}\norm{G_{k}( e^{-\omega s}\,u(s))}_{q}^{q}\ds.
  \end{equation}
  Further, note that
\begin{displaymath}
    \mathds{1}_{\{\abs{e^{-\omega s}u}>\lambda_{k}\}}
    \le \mathds{1}_{\{\abs{e^{-\omega s}u}>\lambda_{k-1}\}}
    \left(\frac{2^{k}\,[\abs{e^{-\omega s}u}-\lambda_{k-1}]^{+}}{\lambda}\right)^{\ell}
  \end{displaymath}
  for every $\ell\ge 0$. We can then estimate
  \begin{equation}
  \label{eq:31}
  |G_{k}(e^{-\omega s}u(s))|^{q} \le \left(\frac{2^{k}}{\lambda}\right)^{\ell}|G_{k-1}(e^{-\omega s}u(s))|^{q+\ell}
  \end{equation}
  on $[t_{k-1},t]$ for $q\ge 0$ and $\ell \ge 0$. We now aim to obtain a recurrence relation for $U_{k}$ of the form in Lemma~\ref{lem:ladyMod}. Taking a supremum over $[t_{k},t]$ in \eqref{eq:15} we can bound $U_{k}$ by
  \begin{equation}
  \label{eq:UkEst}
  \begin{split}
  U_{k}& \le 2\norm{G_{k}(e^{-\omega
  		t_{k}} u(t_{k}))}_{q}^{q} +2q\lambda_{k}\,\omega\int_{t_{k}}^{t}\norm{G_{k}(e^{-\omega s}u(s))}_{q-1}^{q-1}\ds\\
  & \quad + 2q\int_{t_{k}}^{t}e^{-\omega s}|[G_{k}(e^{-\omega s}u)g(s)]_q|\ds.
  \end{split}
  \end{equation}
  Estimating the first term by Lemma~\ref{le:GrowthConditionCompleteResolvent} and choosing $s_k$ according to~\eqref{eq:29},
  \begin{align*}
  &\norm{G_{k}(e^{-\omega
  		t_{k}} u(t_{k}))}_{q}^{q}\\
  		&\quad \le \left(\norm{G_{k}(e^{-\omega
  		s_{k}} u(s_{k}))}_{q}+\int_{s_k}^{t_k}e^{-\omega \tau}\norm{g(\tau)\mathds{1}_{\set{e^{-\omega \tau}|u(\tau)| > \lambda_k}}}_{q}\dtau\right)^{q}\\
  		&\quad \le \tfrac{2^{k+q}}{t} \int_{t_{k-1}}^{t_{k}}\norm{G_{k}( e^{-\omega s}\,u(s))}_{q}^{q}\ds\\
  		&\qquad+2^q\left(\int_{s_k}^{t_k}e^{-\omega \tau}\norm{g(\tau)\mathds{1}_{\set{e^{-\omega \tau}|u(\tau)| > \lambda_k}}}_{q}\dtau\right)^{q}.
  \end{align*}
  Separating the $g$ in the second term here by H\"older's inequality, we have
  \begin{displaymath}
  \begin{split}
  &\int_{s_k}^{t_k}e^{-\omega \tau}\norm{g(\tau)\mathds{1}_{\set{e^{-\omega \tau}|u(\tau)| > \lambda_k}}}_{q}\dtau\\
  &\quad \le \left(\int_{s_k}^{t_k}\norm{e^{-\omega \tau}g(\tau)}_{\rho}^{\psi}\dtau\right)^{\frac{1}{\psi}}\left(\int_{s_k}^{t_k}\norm{\mathds{1}_{\set{e^{-\omega \tau}|u(\tau)| > \lambda_k}}}_{\rho_q'}^{\psi'}\dtau\right)^{\frac{1}{\psi'}}
  \end{split}
  \end{displaymath}
  where we choose $\frac{1}{\rho}+\frac{1}{\rho_q'} = \frac{1}{q}$ and $\frac{1}{\psi}+\frac{1}{\psi'} = 1$. We can then estimate $U_k$, extending the time integrals $(t_{k-1},t)$, with
  \begin{displaymath}
  \begin{split}
  U_{k} &\le \tfrac{2^{k+q+1}}{t} \int_{t_{k-1}}^{t}\norm{G_{k}( e^{-\omega s}\,u(s))}_{q}^{q}\ds\\
  &\quad+2^{q+1}\norm{g}_{L^{\psi}(t_{k-1},t;L^{\rho}_{\mu})}^{q}\left(\int_{t_{k-1}}^{t}\norm{\mathds{1}_{\set{e^{-\omega \tau}|u(\tau)| > \lambda_k}}}_{\rho_q'}^{\psi'}\dtau\right)^{\frac{q}{\psi'}}\\
  &\quad +2q\lambda_{k}\,\omega\int_{t_{k-1}}^{t}\norm{G_{k}(e^{-\omega s}u(s))}_{q-1}^{q-1}\ds\\
  &\quad+2q\norm{g}_{L^{\psi}(t_{k-1},t;L^{\rho}_{\mu})}\left(\int_{t_{k-1}}^{t}\norm{G_{k}(e^{-\omega s}u(s))}_{(q-1)\rho'}^{(q-1)\psi'}\ds\right)^{\frac{1}{\psi'}}.
  \end{split}
  \end{displaymath}
  where we choose $\frac{1}{\rho}+\frac{1}{\rho'} = 1$.
  
  We apply~\eqref{eq:31} to each $G_k$ term, as well as $\mathds{1}_{\set{e^{-\omega \tau}|u(\tau)| > \lambda_k}}$, with $\ell=\varepsilon_1$, $1+\varepsilon_1$, $q+\varepsilon_2-\rho_q'$ and $q+\varepsilon_3-(q-1)\rho'$. 
  The positive constants $\varepsilon_1$, $\varepsilon_2$ and $\varepsilon_3$ will later be chosen such that $\ell \ge 0$ in each case and an appropriate recurrence relation may be obtained. 
  Note that the requirement $\ell \ge 0$ will be satisfied as a result of assumption \eqref{hyp:31}. Then noting that $\lambda_k < \lambda$, there exists $C > 0$ depending on $q$ such that
  \begin{equation}
    \label{eq:30}
  \begin{split}
    \frac{U_{k}}{C} & \le
    \frac{2^{k(1+\varepsilon_1)}}{\lambda^{\varepsilon_1}}\left(\tfrac{1}{t}+
      \omega\right) \int_{t_{k-1}}^{t}
    \norm{G_{k-1}(e^{-\omega s}u(s))}^{q+\varepsilon_1}_{q+\varepsilon_1}\; \ds\\
    &+\left(\tfrac{2^k}{\lambda}\right)^{\frac{q(q+\varepsilon_2)}{\rho_q'}}\norm{g}_{L^{\psi}(0,t;L^{\rho_q}_{\mu})}^{q}\left(\int_{t_{k-1}}^{t}\norm{G_{k-1}(e^{-\omega s}u(s))}_{q+\varepsilon_2}^{(q+\varepsilon_2)\frac{\psi'}{\rho_q'}}\dtau\right)^{\frac{q}{\psi'}}\\
    & +\norm{g}_{L^{\psi}(0,t;L^{\rho}_{\mu})}\left(\tfrac{2^k}{\lambda}\right)^{\frac{q+\varepsilon_3}{\rho'}-(q-1)}\left(\int_{t_{k-1}}^{t}\norm{G_{k-1}(e^{-\omega s}u(s))}_{q+\varepsilon_3}
    ^{(q+\varepsilon_3)\frac{\psi'}{\rho'}}\ds\right)^{\frac{1}{\psi'}}.
  \end{split}
\end{equation}
Now it remains to recover $U_{k-1}$ from integrals of the form
\begin{equation}
\label{eq:generalGkNorm}
\int_{t_{k-1}}^{t}\norm{G_{k-1}(e^{-\omega s}u(s))}_{q+\varepsilon}
^{(q+\varepsilon)M}\ds
\end{equation}
where $\varepsilon > 0$ and $M > 0$. 
In particular, we set $q_{\varepsilon} := q+\varepsilon$ and choose $\varepsilon$ as follows. 
To obtain $U_{k-1}$ from~\eqref{eq:generalGkNorm} we will apply H\"older's inequality, so choose $\varepsilon > 0$ and $\theta \in [0,1]$ such that
  \begin{displaymath}
  \frac{1}{q_{\varepsilon}} = \frac{\theta}{q}+\frac{1-\theta}{r}\quad \text{and}\quad (1-\theta)q_{\varepsilon}M = \sigma.
  \end{displaymath}
  In particular, we choose
  \begin{equation}
  \varepsilon  = \begin{cases}
  \frac{\sigma}{M}\left(1-\frac{q}{r}\right)& \text{if } r < \infty,\\
  \frac{\sigma}{M}& \text{if } r = \infty,
  \end{cases}
  \end{equation}
  and
  \begin{equation}
  \theta = \begin{cases}
  1-\frac{1}{1+q\left(\frac{M}{\sigma}-\frac{1}{r}\right)}& \text{if } r < \infty,\\
  \frac{Mq}{\sigma+Mq}& \text{if } r = \infty,
  \end{cases}
  \end{equation}
  satisfying $\theta < 1$ and $\varepsilon > 0$  given that $M > 0$. 
  The condition $\theta \ge 0$ requires that $M \ge \frac{\sigma}{r}$. 
  Since we take $M = 1$, $\frac{\psi'}{\rho_q'}$ and $\frac{\psi'}{\rho'}$ in the case of~\eqref{eq:30}, this is satisfied by assumptions~\eqref{hyp:31} and~\eqref{hyp:32}. 
  Then applying standard $L^{p}$ interpolation with $\theta$,
  \begin{align*}
  &\int_{t_{k-1}}^{t}
  \norm{G_{k-1}(e^{-\omega s}u(s))}_{q_{\varepsilon}}^{q_{\varepsilon}M}\; \ds\\
  &\hspace{2cm} \le
  \int_{t_{k-1}}^{t}e^{-\omega (\sigma-q)s}\left(\norm{G_{k-1}(e^{-\omega
  		s}u(s))}^{q}_{q}\right)^{\frac{\theta
  		q_{\varepsilon}M}{q}}\times\\
  &\hspace{3cm} e^{\omega (\sigma-q)s}\norm{G_{k-1}(e^{-\omega s}u(s))}^{(1-\theta) q_{\varepsilon}M}_{r}\; \ds\\
  &\hspace{2cm} \le \frac{1}{L}\sup_{\hat{s}\in
  	[t_{k-1},t]}e^{-\omega (\sigma-q)\hat{s}}\left(\norm{G_{k-1}(e^{-\omega
  	\hat{s}}u(\hat{s}))}_{q}^{q}\,\right)^{\frac{\theta
  	q_{\varepsilon}M}{q}}\times\\
  &\hspace{3cm} L\int_{t_{k-1}}^{t}e^{\omega (\sigma-q)s} \norm{G_{k-1}(e^{-\omega s}u(s))}^{\sigma}_{r}\;
  \ds.
  \end{align*}
  We estimate $e^{-\omega(\sigma-q)s}$ on $[t_{k-1},t]$ according to the sign of $\sigma-q$ so that
  \begin{displaymath}
  \sup_{s\in
  	[t_{k-1},t]}e^{-\omega(\sigma-q)s} = \begin{cases}
  e^{-\omega(\sigma-q)t}& \text{if } \sigma \le q,\\
  1& \text{if } \sigma > q.
  \end{cases}
  \end{displaymath}
  Hence, applying Young's inequality such that both terms have the same exponent and evaluating, we have
  \begin{displaymath}
  \int_{t_{k-1}}^{t}\norm{G_{k-1}(e^{-\omega s}u(s))}_{q_{\varepsilon}}^{q_{\varepsilon}M}\; \ds 
  \le \frac{1}{L}e^{\omega t(q-\sigma)^{+} }U_{k-1}^{M-\frac{\sigma}{r}+1}
  \end{displaymath}
  where $(q-\sigma)^{+} = \max(0,q-\sigma)$.
  
  To apply Lemma \ref{lem:ladyMod}, the exponents of $U_{k-1}$ corresponding to~\eqref{eq:30} must be of the form $1+\delta$ with $\delta > 0$. Hence we require
  \begin{displaymath}
  \begin{split}
  &\frac{q}{\rho_q'}+\frac{q}{\psi'}\left(1-\frac{\sigma}{r}\right) > 1\quad \text{ and }\quad \frac{1}{\rho'}+\frac{1}{\psi'}\left(1-\frac{\sigma}{r}\right) > 1
  \end{split}
  \end{displaymath}
  which follow from~\eqref{hyp:31}. Rewriting~\eqref{eq:30} as a recurrence relation for $U_{k+1}$, we introduce the following constants
  \begin{displaymath}
  \begin{split}
  c_{1}& = \left(\tfrac{1}{\lambda}\right)^{\varepsilon_1}\left(\tfrac{1}{t}+
  \omega\right)e^{\omega t(q-\sigma)^{+} },\quad
  c_{2} = \left(\tfrac{1}{\lambda}\right)^{\frac{q(q+\varepsilon_2)}{\rho_q'}}\norm{g}_{L^{\psi}(0,t;L^{\rho_q}_{\mu})}^{q}e^{\frac{q\omega t}{\psi'}(q-\sigma)^{+} },\\
  c_3& = \left(\tfrac{1}{\lambda}\right)^{\frac{q+\varepsilon_3}{\rho'}-q+1}\norm{g}_{L^{\psi}(0,t;L^{\rho}_{\mu})}e^{\frac{\omega t}{\psi'}(q-\sigma)^{+} },\quad b = \max\left(2^{1+\varepsilon_1},2^{\frac{q(q+\varepsilon_2)}{\rho_q'}}, 2^{\frac{q+\varepsilon_3}{\rho'}-q+1}\right),
  \end{split}
  \end{displaymath}
  and exponents
  \begin{displaymath}
  \delta_1  = 1-\frac{\sigma}{r},\quad 
  \delta_2  = \frac{q}{\psi'}\left(1-\frac{\sigma}{r}\right)+\frac{q}{\rho_q'}-1,\quad
  \delta_3  = \frac{1}{\psi'}\left(1-\frac{\sigma}{r}\right)+\frac{1}{\rho'}-1.
  \end{displaymath}
  Then we obtain
  \begin{displaymath}
	U_{k+1}\le b^{k+1}\sum_{i=1}^{3}Cc_{i}\, U_{k}^{1+\delta_{i}}
  \end{displaymath}
    for some $C > 0$ depending on $q$, $L$ and $\psi$. Then setting
 \begin{displaymath}
	\delta_m := \min\left(\delta_1, \delta_2, \delta_3\right) = \delta_3,
	\end{displaymath}
    in order to apply Lemma~\ref{lem:ladyMod}, we require that
  \begin{equation}
    \label{eq:37}
    U_{0}\le \frac{1}{3\, b^{\frac{1}{\delta_m^2}}}\min_{i \in \set{1,2,3}}\frac{1}{(Cc_{i})^{\frac{1}{\delta_{i}}}}.
  \end{equation}
  We estimate $U_0$ by~\eqref{eq:UkEst} and~\eqref{eq:27}, so that
  \begin{align*}
    U_{0}&\le 2\left(\norm{u_{0}}_{q}^{q}+q\int_{0}^{t}\norm{e^{-\omega s}u(s)}_{q}^{q-1}e^{-\omega s}\norm{g(s)}_{q}\ds\right)\\
    & \le 2\left(\norm{u_{0}}_{q}^{q}+q\left(\norm{u_{0}}_{q}+\int_{0}^{t}e^{-\omega r}\norm{g(r)}_{q}\dr\right)^{q-1}\int_{0}^{t}e^{-\omega s}\norm{g(s)}_{q}\ds)\right)\\
    & \le 2(1+q)\left(\norm{u_0}_q+\int_0^t e^{-\omega s}\norm{g(s)}_{q}\ds\right)^{q}.
  \end{align*}
  As the previous estimates were for arbitrary $\lambda \ge 0$, relabelling $C > 0$ to include $b$, we want to find $\lambda$ such that 
  \begin{equation}
   \label{eq:39}
   c_i \le \frac{C}{\left(\norm{u_0}_q+\norm{g}_{L^1(0,t;L^q)}\right)^{q\delta_i}}
  \end{equation}
  for $i \in \set{1,2,3}$.
  Set
  \begin{displaymath}
  \begin{split}
  \beta_1& = \begin{cases}
  \frac{q-\sigma}{\varepsilon_1}& \text{if } \sigma \le q,\\
  0& \text{if } \sigma > q,
  \end{cases}\\
  \kappa_1& = \begin{cases}
  \frac{(q-\sigma)\rho_q'}{\psi' q(q+\varepsilon_2)}& \text{if } \sigma \le q,\\
  0& \text{if } \sigma > q,
  \end{cases}\\
  \kappa_2& = \begin{cases}
  \frac{q-\sigma}{\psi'\left(\frac{q+\varepsilon_3}{\rho'}-q+1\right)}& \text{if } \sigma \le q,\\
  0& \text{if } \sigma > q.
  \end{cases}
  \end{split}
  \end{displaymath}
  Then~\eqref{eq:39} holds if
  \begin{displaymath}
  \begin{split}
    \lambda&\ge C e^{\omega \beta_1 t}\left(\left(\tfrac{2^q}{t}+
    	q\omega\right)\left(\norm{u_0}_q+\int_0^t e^{-\omega s}\norm{g(s)}_{q}\ds\right)^{q\delta_1}\right)^{\frac{1}{\varepsilon_1}},\\
  \lambda& \ge C e^{\omega \kappa_1 t}\left(\norm{g}_{L^{\psi}(0,t;L^{\rho_q}_{\mu})}^q\left(\norm{u_0}_q+\int_0^t e^{-\omega s}\norm{g(s)}_{q}\ds\right)^{q\delta_2}\right)^{\frac{\rho_q'}{q(q+\varepsilon_2)}},\text{ and}\\
  \lambda& \ge C e^{\omega \kappa_2 t}\left(\norm{g}_{L^{\psi}(0,t;L^{\rho}_{\mu})}\left(\norm{u_0}_q+\int_0^t e^{-\omega s}\norm{g(s)}_{q}\ds\right)^{q\delta_3}\right)^{\frac{1}{\frac{q+\varepsilon_3}{\rho'}-(q-1)}}.
  \end{split}
  \end{displaymath}
  for some $C > 0$ depending on $q,\sigma,r,\rho,\psi$ and $L$.
  So taking $\lambda$ as the maximum of these estimates, we have by Fatou's Lemma,
	\begin{displaymath}
	 0= \liminf_{k\to \infty}U_{k}\ge \sup_{\hat{s}\in [t,t]}\norm{G_{\lambda}(e^{-\omega
 		\hat{s}}u(\hat{s}))}_{q}^{q}\,+ \int_{t}^{t}e^{\omega(\sigma-q)t}\norm{G_{\lambda}(e^{-\omega s}u(s))}_{r}^{\sigma} \ds.
\end{displaymath}
Noting that $t$ was chosen arbitrarily, this implies that 
\begin{displaymath}
	\norm{u(t)}_{\infty}\le \lambda\quad \text{for all } t \in (0,T].
\end{displaymath}
Evaluating constants and simplifying, we obtain~\eqref{eqn:lqlinfty}.
\end{proof}

The following lemma shows that the growth condition on  $G_{\lambda}(e^{-\omega t}u(t))$ given by~\eqref{eq:27} holds for operators with complete resolvent. In the case $\lambda = 0$ this reduces to the standard growth estimate for accretive operators with complete resolvent.

\begin{lemma}
	\label{le:GrowthConditionCompleteResolvent}
	Let $1 \le q_{0} < \infty$ and suppose $A$ is $\omega$-quasi $m$-accretive in $L^{q_{0}}_{\mu}$ with complete resolvent for some $\omega \ge 0$. Let $1 \le q \le \infty$ such that $g \in L^{1}(0,T;L^{q}_{\mu}\cap L^{q+\varepsilon}_{\mu})$ for some $\varepsilon > 0$ and $u_0 \in \overline{D(A)}^{\mbox{}_{L^{q_{0}}_{\mu}}}\cap L^{q}_{\mu}$. Denote by $u(t)$ the mild solution to~\eqref{eq:section3equation}. Then we have the growth estimate
	\begin{equation}
	\label{eq:growthconditionGlambda}
	\begin{split}
	\norm{G_{\lambda}(e^{-\omega t}u(t))}_{q} \le & \lnorm{G_{\lambda}(e^{-\omega s}u(s))}_{q}\\
	&+\int_{s}^{t}e^{-\omega \tau}\norm{g(\tau)\mathds{1}_{\set{e^{-\omega \tau}|u(\tau)| > \lambda}}}_{q}\dtau
	\end{split}
	\end{equation}
	for all $0 \le s \le t \le T$ and $\lambda \ge 0$.
\end{lemma}

\begin{proof}
	For $u \in D(J_h^A)$, we can rewrite the resolvent operator in the following way,
	\begin{displaymath}
	J_{\frac{h}{1-h\omega}}^{A+\omega I}u = (1-h\omega)J_{h}^{A}u.
	\end{displaymath}
	Then for $A+\omega I$ having complete resolvent, consider $\alpha \in \R$ and take $j(\cdot) = |G_{\lambda}(\alpha\cdot)|^{q}$ in the complete resolvent property~\eqref{eq:4} with the resolvent operator $J_\frac{h}{1-h\omega}^{A+\omega I}$ to obtain the estimate
	\begin{equation}
	\label{eq:4.4resolventestimate}
	\begin{split}
	\int_{\Sigma}\left|G_{\lambda}(\alpha v)\right|^{q}\dmu& \ge \int_{\Sigma}\left|G_{\lambda}\left(\alpha J_{\frac{h}{1-h\omega}}^{A+\omega I}v\right)\right|^{q}\dmu\\
	& = \lnorm{G_{\lambda}\left(\alpha(1-h\omega)J_{h}^{A}v\right)}_{q}^{q}.
	\end{split}
	\end{equation}
	Given $s < t$, take a partition $(s_{n})_{n\in \set{0,1,...,N}}$ of $[s,t]$ given by $s_{n} = s+\frac{n(t-s)}{N}$. Let
	\begin{equation}
	\label{eq:gnDef}
	g_{n}:= \frac{N}{t-s}\int_{s_{n}}^{s_{n+1}}g(\tau)\dtau.
	\end{equation}
	Then let $(v_{n})_{n\in \set{0,1,...,N}}$ be the solution to the discrete problem
	\begin{displaymath}
	\begin{cases}
    \begin{alignedat}{1}
	v_{n}+\frac{t-s}{N}Av_{n}& = v_{n-1}+\frac{t-s}{N}g_{n-1}\quad \text{ for } n = 1,...,N,\\
	v_{0}& = u(s).
    \end{alignedat}
	\end{cases}
	\end{displaymath}
	We can apply the resolvent estimate~\eqref{eq:4.4resolventestimate} to $v_{n}$, taking $h = \frac{t-s}{N}$. Further, let $S_{n} = \{x \in \Sigma: \frac{e^{-\omega t}}{(1-h\omega)^{N-n}}\left|v_{n}+h g_{n}\right| > \lambda\}$ so that we may separate terms.
	\begin{displaymath}
	\begin{split}
	&\lnorm{G_{\lambda}\left(\frac{e^{-\omega t}}{(1-h\omega)^{N-n}}v_{n}\right)}_{q} \le \lnorm{G_{\lambda}\left(\frac{e^{-\omega t}}{(1-h\omega)^{N-n+1}}\left(v_{n-1}+hg_{n-1}\right)\right)}_{q}\\
	&\hspace{0.6cm}\le \lnorm{G_{\lambda}\left(\frac{e^{-\omega t}}{(1-h\omega)^{N-n+1}}v_{n-1}\right)}_{q}+\frac{e^{-\omega t}h}{(1-h\omega)^{N-n+1}}\lnorm{g_{n-1}\mathds{1}_{S_{n-1}}}_{q}.
	\end{split}
	\end{displaymath}
	Repeating this, we have
	\begin{displaymath}
	\begin{split}
	\norm{G_{\lambda}(e^{-\omega t}v_{N})}_{q}& \le \lnorm{G_{\lambda}\left(\frac{e^{-\omega t}}{\left(1-\tfrac{\omega(t-s)}{N}\right)^{N}}v_{0}\right)}_{q}\\
	&\quad+\frac{t-s}{N}\sum_{n=0}^{N-1}\frac{e^{-\omega t}}{\left(1-\tfrac{\omega(t-s)}{N}\right)^{N-n}}\lnorm{g_{n}\mathds{1}_{S_n}}_{q}
	\end{split}
	\end{displaymath}
	which converges to~\eqref{eq:growthconditionGlambda} as $N \rightarrow \infty$ by the definition of mild solution and the projection~\eqref{eq:gnDef}.
\end{proof}

The following proposition introduces the pointwise estimate~\eqref{eq:5Glambda1} for operators with complete resolvent which we will use as the condition for applying Theorem \ref{thm:LqLinfty-deGiorgi} to the doubly nonlinear problem~\eqref{eq:1} (see Section~\ref{sec:applicationToFractionalPLaplacian}). In particular, this provides~\eqref{eq:15}.

\begin{proposition}
	\label{lem:discreteToSmoothEstimateProof}
	For $1\le q_{0}< \infty$ and $\omega\ge 0$, let $A$ be an
	$\omega$-quasi $m$-accretive operator on $L^{q_{0}}_{\mu}$ with complete resolvent. Suppose there are $q_0 \le q <r\le \infty$, $1\le \sigma<r$ and $C > 0$ such that $A$ satisfies the one-parameter Sobolev type inequality
	\begin{equation}
	\label{eq:5Glambda1}
	\norm{G_{\lambda}(u)}_{r}^{\sigma}\le C
	\; [G_{\lambda}(u),v+\omega(G_{\lambda}(u)+\lambda \mathds{1})]_{q}
	\end{equation}
	for every $(u,v)\in A$ and $\lambda\ge 0$. Let $g \in L^{1}(0,T;L^{q_0}_{\mu})\cap L^{1}(0,T;L^{q+\varepsilon}_{\mu})$ for some $\varepsilon > 0$ and $u\in C([0,T];L^{q_0}_{\mu})\cap L^{1}(0,T;L^{1\cap\infty}_{\mu})$ be the mild solution to~\eqref{eq:section3equation} where $u_{0} \in \overline{D(A)}^{\mbox{}_{L^{q_0}_{\mu}}}\cap L^{1\cap\infty}_{\mu}$. Then for every $\lambda\ge 0$, $u$ satisfies the
	``level set energy inequality''
	\begin{equation}
	\label{eq:151}
	\begin{split}
	& \norm{G_{\lambda}(e^{-\omega
			t_{2}}\,u(t_{2}))}_{q}^{q}+\frac{q}{C}\int_{t_{1}}^{t_{2}}e^{\omega(\sigma-q)s}\norm{G_{\lambda}(e^{-\omega s}\,u(s))}_{r}^{\sigma}
	\ds\\
	&\hspace{2cm}\le
	\norm{G_{\lambda}(e^{-\omega
			t_{1}}\,u(t_{1}))}_{q}^{q}+\lambda\,\omega\int_{t_{1}}^{t_{2}}\norm{G_{\lambda}(e^{-\omega s}\,u(s))}_{q-1}^{q-1}\ds\\
	&\hspace{2cm}\quad+q\int_{t_{1}}^{t_{2}}e^{-\omega s}\left|[G_{\lambda}(e^{-\omega s}u(s)),g]_{q}\right|\ds
	\end{split}
	\end{equation}
	for all $0\le t_{1}<t_{2}\le T$.
\end{proposition}

\begin{proof}
	Let $\set{s_{n}}_{n \in \set{0,1,...,N}}$ be the discretization of the interval $[t_{1},t_{2}]$ given by $s_{n} := t_{1}+\frac{n(t_{2}-t_{1})}{N}$. Then for all $n\in \set{0,...,N-1}$ set 
	\begin{displaymath}
	g_{N}(s):= \frac{N}{t_{2}-t_{1}}\int_{s_{n}}^{s_{n+1}}g(\tau)\dtau\quad \text{for } s \in [s_{n},s_{n+1})
	\end{displaymath}
	which will converge to $g$ in $L^{1}(0,T;L^{q_0}_{\mu})\cap L^{1}(0,T;L^{q}_{\mu})$ as $N \rightarrow \infty$. Let $\set{v_{n}}_{n \in \set{0,1,...,N}}$ be the associated family of solutions to the time discretized Cauchy problem satisfying
	\begin{equation}
	\label{eq:resolventEquationForv}
	v_{n+1} = J_{\frac{t_{2}-t_{1}}{N}}^{A}\left(v_{n}+\frac{t_{2}-t_{1}}{N}g_{N}(s_{n})\right)
	\end{equation}
	for all $n \in \set{0,...,N-1}$ with $v_{0} = u(t_{1})$. Note that by the complete resolvent property of $A$ with $u_{0} \in L^{q}_{\mu}$ and $g \in L^{1}(0,T;L^{q}_{\mu})$, $v_{n} \in L^{q}_{\mu}$ for all $n \in \set{0,...,N}$. We first obtain a discrete version of the integral estimate~\eqref{eq:151} by discretizing with a telescoping sum and applying a product rule. For $q \ge 1$ we use the following property of $q$-brackets,
	\begin{equation}
	\label{eq:qbracketProperty}
	[u,v]_{q} \le \frac{1}{q}\norm{u+v}_{q}^{q}-\frac{1}{q}\norm{u}_{q}^{q}
	\end{equation}
	for every $u,v \in L^{q}_{\mu}$. Here we apply~\eqref{eq:qbracketProperty} to the following telescoping sum, taking $u = G_{\lambda}(e^{-\omega s_{n}}v_{n})$ and $v = G_{\lambda}(e^{-\omega s_{n-1}}v_{n-1})-G_{\lambda}(e^{-\omega s_{n}}v_{n})$.
	\begin{align*}
	&\norm{G_{\lambda}(e^{-\omega t_{2}}v_{N})}_{q}^{q}-\norm{G_{\lambda}(e^{-\omega t_{1}}u(t_{1}))}_{q}^{q}\\
	&\hspace{2.4cm} = \sum_{n=1}^{N}\norm{G_{\lambda}(e^{-\omega s_{n}}v_{n})}_{q}^{q}-\norm{G_{\lambda}(e^{-\omega s_{n-1}}v_{n-1})}_{q}^{q}\\
	&\hspace{2.4cm} \le \sum_{n=1}^{N}q[G_{\lambda}(e^{-\omega s_{n}}v_{n}),G_{\lambda}(e^{-\omega s_{n}}v_{n})-G_{\lambda}(e^{-\omega s_{n-1}}v_{n-1})]_{q}.
	\end{align*}
	Noting that $G_{\lambda}$ is a Lipschitz continuous function, we can differentiate almost everywhere on $\mathbb{R}$. Here we define
	\begin{displaymath}
	G_{\lambda}'(s) = \begin{cases}
	1 & \text{ if } |s| > \lambda,\\
	0 & \text{ if } |s| \le \lambda,
	\end{cases}
	\end{displaymath}
	and
	\begin{displaymath}
	c_{n} = \int_{0}^{1}G_{\lambda}'(\theta e^{-\omega s_{n}}v_{n}+(1-\theta)e^{-\omega s_{n-1}}v_{n-1})d\theta
	\end{displaymath}
	so that we can rewrite this difference as an integral of the derivative with
	\begin{displaymath}
	G_{\lambda}(e^{-\omega s_{n}}v_{n})-G_{\lambda}(e^{-\omega s_{n-1}}v_{n-1}) = c_{n}\left(e^{-\omega s_{n}}v_{n}-e^{-\omega s_{n-1}}v_{n-1}\right).
	\end{displaymath}
	Then returning to the estimate in the discrete setting,
	\allowdisplaybreaks
	\begin{align*}
	&\norm{G_{\lambda}(e^{-\omega t_{2}}v_{N})}_{q}^{q}-\norm{G_{\lambda}(e^{-\omega t_{1}}u(t_{1}))}_{q}^{q}\\
	&\hspace{1.1cm} \le q\sum_{n=1}^{N}[G_{\lambda}(e^{-\omega s_{n}}v_{n}),c_{n}\left(e^{-\omega s_{n}}v_{n}-e^{-\omega s_{n-1}}v_{n-1}\right)]_{q}\\
	&\hspace{1.1cm} \le q\sum_{n=1}^{N}[G_{\lambda}(e^{-\omega s_{n}}v_{n}),c_{n}e^{-\omega s_{n-1}}(v_{n}-v_{n-1})]_{q}\\
	&\hspace{1.1cm}\quad+q\sum_{n=1}^{N}[G_{\lambda}(e^{-\omega s_{n}}v_{n}),c_{n}(e^{-\omega s_{n}}-e^{-\omega s_{n-1}})v_{n}]_{q}.
	\end{align*}
	Defining
	\begin{displaymath}
	R_{n} = [G_{\lambda}(e^{-\omega s_{n}}v_{n}),(c_{n}-1)(v_{n}-v_{n-1})]_{q},
	\end{displaymath}
	we can rewrite the previous estimate as
	\begin{displaymath}
	\begin{split}
	&\norm{G_{\lambda}(e^{-\omega t_{2}}v_{N})}_{q}^{q}-\norm{G_{\lambda}(e^{-\omega t_{1}}u(t_{1}))}_{q}^{q}\\
	&\hspace{0.7cm} \le \frac{q(t_{2}-t_{1})}{N}\sum_{n=1}^{N}e^{-\omega s_{n-1}}[G_{\lambda}(e^{-\omega s_{n}}v_{n}),-Av_{n}+g_{N}(s_{n-1})]_{q}\\
	&\hspace{0.7cm}\quad+q\sum_{n=1}^{N}e^{-\omega s_{n-1}}R_n + q\sum_{n=1}^{N}[G_{\lambda}(e^{-\omega s_{n}}v_{n}),e^{-\omega s_{n}}v_{n}c_{n}]_{q}\left(1-e^{\frac{\omega (t_{2}-t_{1})}{N}}\right).
	\end{split}
	\end{displaymath}
	We now consider the values of $v_{n}$ and $v_{n-1}$ for almost every $x \in \Sigma$ to show that $R_{n}$ is non-positive. Note that $c_n \le 1$ and in particular,
	\begin{align*}
	c_{n}(x) \in \begin{cases}
	\set{0}& \text{ if } |\theta e^{-\omega s_{n}}v_{n}+(1-\theta)e^{-\omega s_{n-1}}v_{n-1}| \le \lambda \text{ for a.e.~} \theta \in (0,1),\\
	\set{1}& \text{ if } |\theta e^{-\omega s_{n}}v_{n}+(1-\theta)e^{-\omega s_{n-1}}v_{n-1}| > \lambda \text{ for a.e.~} \theta \in (0,1),\\
	(0,1)& \text{ otherwise}.
	\end{cases}
	\end{align*}
	Since $|e^{-\omega s_{n}}v_{n}(x)| \le \lambda$ implies that $G_{\lambda}(e^{-\omega s_{n}}v_{n}(x)) = 0$ and so does not contribute to $R_n$, we consider only $x \in \Sigma$ such that $|e^{-\omega s_{n}}v_{n}(x)| > \lambda$. Then there will be some subinterval of $\theta \in (0,1)$ such that 
	\begin{displaymath}
	|\theta e^{-\omega s_{n}}v_{n}(x)+(1-\theta)e^{-\omega s_{n-1}}v_{n-1}(x)| > \lambda
	\end{displaymath}
	implying that $c_{n}(x) > 0$. If $c_{n}=1$ then $c_{n}(x)-1 = 0$ and so this will not contribute to $R_{n}$. Hence we only consider $x$ such that $c_{n}(x) \in (0,1)$. Since $|e^{-\omega s_{n}}v_{n}(x)| > \lambda$, this implies that either
	\begin{displaymath}
	|e^{-\omega s_{n-1}}v_{n-1}(x)| \le \lambda
	\end{displaymath}
	or
	\begin{displaymath}
	\sign(v_{n-1}(x)) = -\sign(v_{n}(x)).
	\end{displaymath}
	For the first case, $G_{\lambda}(e^{-\omega s_{n-1}}v_{n-1}(x)) = 0$ so
	\begin{align*}
	&(G_{\lambda}(e^{-\omega s_{n}}v_{n}(x)))^{q-1}(v_{n}(x)-v_{n-1}(x))\\
	&\quad = \left((G_{\lambda}(e^{-\omega s_{n}}v_{n}(x)))^{q-1}-(G_{\lambda}(e^{-\omega s_{n-1}}v_{n-1}(x)))^{q-1}\right)(v_{n}(x)-v_{n-1}(x))\\
	&\quad \ge 0.
	\end{align*}
	Note that for $q = 1$, $(G_{\lambda}(v_{n}(x)))^{q-1} = \sign(G_{\lambda}(v_{n}(x)))$. For the second case, $\sign(v_{n}(x)-v_{n-1}(x)) = \sign(v_{n}(x))$ so
	\begin{displaymath}
	(G_{\lambda}(e^{-\omega s_{n}}v_{n}(x)))^{q-1}(v_{n}(x)-v_{n-1}(x)) \ge 0.
	\end{displaymath}
	Putting this together we have that $R_{n} \le 0$. Returning to the discrete estimate,
	\begin{displaymath}
	\begin{split}
	\norm{G_{\lambda}(e^{-\omega  t_{2}}v_{N})}_{q}^{q}&-\norm{G_{\lambda}(e^{-\omega t_{1}}u(t_{1}))}_{q}^{q}\\
	& \le \frac{q(t_{2}-t_{1})}{N}\sum_{n=1}^{N}e^{-\omega s_{n-1}}[G_{\lambda}(e^{-\omega s_{n}}v_{n}),-Av_{n}+g_{N}(s_{n})]_{q}\\
	&\quad+\sum_{n=1}^{N}\norm{G_{\lambda}(e^{-\omega s_{n}}v_{n})}_{q}^{q}\left(1-e^{\frac{q\omega (t_{2}-t_{1})}{N}}\right).
	\end{split}
	\end{displaymath}
	Note that by~\eqref{eq:5Glambda1},
	\allowdisplaybreaks
	\begin{align*}
	e^{-\omega s_{n-1}}&[G_{\lambda}(e^{-\omega s_{n}}v_{n}),Av_{n}]_{q} = e^{-\omega s_{n-1}}e^{-(q-1)\omega s_{n}}[G_{\lambda e^{\omega s_{n}}}(v_{n}),Av_{n}]_{q}\\
	& = e^{\frac{\omega (t_{2}-t_{1})}{N}}e^{-q\omega s_{n}}[G_{\lambda e^{\omega s_{n}}}(v_{n}),Av_{n}]_{q}\\
	& \ge e^{\frac{\omega (t_{2}-t_{1})}{N}}e^{-q\omega s_{n}}\frac{1}{C}\norm{G_{\lambda e^{\omega s_{n}}}(v_{n})}_{r}^{\sigma}\\
	&\quad-e^{\frac{\omega (t_{2}-t_{1})}{N}}e^{-q\omega s_{n}}\omega\left(\norm{G_{\lambda e^{\omega s_{n}}}(v_{n})}_{q}^{q}+\lambda e^{\omega s_{n}}\norm{G_{\lambda e^{\omega s_{n}}}(v_{n})}_{q-1}^{q-1}\right)\\
	& = e^{\frac{\omega (t_{2}-t_{1})}{N}}\frac{e^{\omega s_{n}(\sigma-q)}}{C}\norm{G_{\lambda }(e^{-\omega s_{n}}v_{n})}_{r}^{\sigma}\\
	&\quad-e^{\frac{\omega (t_{2}-t_{1})}{N}}\omega\left(\norm{G_{\lambda }(e^{-\omega s_{n}}v_{n})}_{q}^{q}+\lambda \norm{G_{\lambda }(e^{-\omega s_{n}}v_{n})}_{q-1}^{q-1}\right).
	\end{align*}
	Hence we have
	\begin{displaymath}
	\begin{split}
	e^{-\omega s_{n-1}}&[G_{\lambda}(e^{-\omega s_{n}}v_{n}),Av_{n}]_{q}\\
	& \ge \frac{e^{\omega s_{n}(\sigma-q)}}{C}\norm{G_{\lambda }(e^{-\omega s_{n}}v_{n})}_{r}^{\sigma}\\
	&\quad-e^{\frac{\omega (t_{2}-t_{1})}{N}}\omega\left(\norm{G_{\lambda }(e^{-\omega s_{n}}v_{n})}_{q}^{q}+\lambda \norm{G_{\lambda }(e^{-\omega s_{n}}v_{n})}_{q-1}^{q-1}\right).
	\end{split}
	\end{displaymath}
	
	We now aim to take $N \rightarrow \infty$, first converting the discrete sums to integrals. Let $U_{N}$ be a stepwise solution to~\eqref{eq:resolventEquationForv} such that
	\begin{displaymath}
	U_{N}(s) = v_{0}\mathds{1}_{\set{0}}(s)+\sum_{n=1}^{N}v_{n}\mathds{1}_{(s_{n-1},s_{n}]}(s)
	\end{displaymath}
	for every $s \in [t_{1},t_{2}]$. We have
	\allowdisplaybreaks
	\begin{align*}
	&\norm{G_{\lambda}(e^{-\omega  t_{2}}v_{N})}_{q}^{q}-\norm{G_{\lambda}(e^{-\omega t_{1}}u(t_{1}))}_{q}^{q}\\
	&\quad \le -\frac{q(t_{2}-t_{1})}{N}\sum_{n=1}^{N}\left(\frac{e^{\omega s_{n}(\sigma-q)}}{C}\norm{G_{\lambda}(e^{-\omega s_{n}}v_{n})}_{r}^{\sigma}\right.\\
	&\qquad+e^{\frac{\omega (t_{2}-t_{1})}{N}}\omega[G_{\lambda}(e^{-\omega s_{n}}v_{n}),G_{\lambda}(e^{-\omega s_{n}}v_{n})-\lambda \sign(G_{\lambda}(e^{-\omega s_{n}}v_{n}))]_{q}\\
	&\qquad+ e^{-\omega s_{n-1}}[G_{\lambda}(e^{-\omega s_{n}}v_{n}),g_{N}(s_{n})]_{q}\Biggr)\\
	&\qquad+\sum_{n=1}^{N}\norm{G_{\lambda}(e^{-\omega s_{n}}v_{n})}_{q}^{q}\left(1-e^{\frac{q\omega (t_{2}-t_{1})}{N}}\right)\\
	&\quad \le -\frac{q}{C}\int_{t_{1}}^{t_{2}}e^{\omega (\sigma-q)\left(s-\sign(\sigma-q)\frac{t_{2}-t_{1}}{N}\right)}\lnorm{G_{\lambda}(e^{-\omega(s+\frac{t_{2}-t_{1}}{N})}U_{N})}_{r}^{\sigma}\ds\\
	&\qquad+\left|\frac{e^{\frac{q\omega (t_{2}-t_{1})}{N}}-1}{\frac{t_{2}-t_{1}}{N}}-q\omega e^{\frac{\omega(t_{2}-t_{1})}{N}}\right|\int_{t_{1}}^{t_{2}}\norm{G_{\lambda}(e^{-\omega s}U_{N})}_{q}^{q}\ds\\
	&\qquad+q\omega\lambda e^{\frac{\omega(t_{2}-t_{1})}{N}}\int_{t_{1}}^{t_{2}}\norm{G_{\lambda}(e^{-\omega s}U_{N})}_{q-1}^{q-1}\ds\\
	&\qquad+q\int_{t_{1}}^{t_{2}}\int_{\Sigma}|G_{\lambda}(e^{-\omega s}U_{N})|^{q-1}|g_{N}|\dmu \ds.
	\end{align*}
	We now prove convergence of each term in the estimate in order to obtain the continuous version~\eqref{eq:151}. Noting that $U_{N}(s) \rightarrow u(s)$ in $C([0,T];L^{q_{0}}_{\mu})$ and since $\norm{\cdot}_{q}^{q}$ and $\norm{\cdot}_{r}^{\sigma}$ are lower semicontinuous on $L^{q_{0}}_{\mu}$, we have that
	\begin{displaymath}
	\liminf\limits_{N\rightarrow\infty}\norm{G_{\lambda}(e^{-\omega  t_{2}}v_{N})}_{q}^{q} \ge \norm{G_{\lambda}(e^{-\omega  t_{2}}u(t_{2}))}_{q}^{q}
	\end{displaymath}
	and applying Fatou's lemma,
	\begin{displaymath}
	\begin{split}
	\liminf_{N\rightarrow\infty}\int_{t_{1}}^{t_{2}}&e^{\omega (\sigma-q)s}\lnorm{G_{\lambda}(e^{-\omega (s+\frac{t_{2}-t_{1}}{N})}U_{N}(s))}_{r}^{\sigma}\ds\\
	&\quad \ge \int_{t_{1}}^{t_{2}}e^{\omega (\sigma-q)s}\norm{G_{\lambda}(e^{-\omega s}u(s))}_{r}^{\sigma}\ds.
	\end{split}
	\end{displaymath}
	Next, note that by the complete resolvent property of $A$ and Lemma \ref{le:GrowthConditionCompleteResolvent} we can estimate $U_{N}$ and $u$ in $L^{q}_{\mu}$ uniformly on $[0,T]$. Hence let $M$ bound both $\norm{U_{N}}_{q}$ and $\norm{u}_{q}$.
	\begin{displaymath}
	\begin{split}
	\left|\frac{1-e^{-\frac{q\omega t}{N}}}{\frac{t}{N}}-q\omega \right|\int_{t_{1}}^{t_{2}}\norm{G_{\lambda}(e^{-\omega s}U_{N}(s))}_{q}^{q}\ds& \le \left|\frac{1-e^{-\frac{q\omega t}{N}}}{\frac{t}{N}}-q\omega \right|\int_{t_{1}}^{t_{2}}M^{q}\ds\\
	& \rightarrow 0
	\end{split}
	\end{displaymath}
	as $N \rightarrow 0$. For the next term we prove uniform convergence of $G_{\lambda}(e^{-\omega s}U_{N})$ to $G_{\lambda}(e^{-\omega s}u)$ in $C([0,T];L^{q-1}_{\mu})$	when $\lambda > 0$. For this fix $s \in [0,T]$ and let
	\begin{displaymath}
	f_{N} = G_{\lambda}(e^{-\omega s}U_{N}(s))-G_{\lambda}(e^{-\omega s}u(s))
	\end{displaymath}
	so that $f_{N} \rightarrow 0$ in $L^{q_0}$. We note that
	\begin{displaymath}
	\begin{split}
	\norm{f_{N}}_{q}& \le 2M
	\end{split}
	\end{displaymath}
	and by Chebyshev's inequality
	\begin{displaymath}
	\begin{split}
	\mu(\{x\in\Sigma:|f_{N}| > 0\})& \le  \mu(\{x\in\Sigma:e^{-\omega s}|U_{N}| \ge \lambda \text{ or } e^{-\omega s}|u| \ge \lambda\})\\
	& \le \frac{1}{\lambda^{q}}\left(\norm{e^{-\omega s}U_{N}(s)}_{q}+\norm{e^{-\omega s}u(s)}_{q}\right)\\
	& \le \frac{2M}{\lambda^{q}}.
	\end{split}
	\end{displaymath}
	Here we consider cases for $q$. For $q-1 \ge q_{0}$, apply H\"older's inequality with $\theta$ chosen to satisfy
	\begin{displaymath}
	\frac{\theta}{q_{0}}+\frac{1-\theta}{q} = \frac{1}{q-1}
	\end{displaymath}
	to obtain
	\begin{align*}
	\lim\limits_{N\rightarrow\infty}\norm{f_{N}}_{q-1}& \le \lim\limits_{N\rightarrow\infty}\left(\norm{f_{N}}_{q_0}^{\theta}\norm{f_{N}}_{q}^{1-\theta}\right) = 0.
	\end{align*}
	For $q-1 < q_{0}$, we apply Jensen's inequality noting that $|\cdot|^{\frac{q_0}{q-1}}$ is convex. Let
	\begin{displaymath}
	\Sigma_{N} = \{x\in\Sigma:|f_{N}| > 0\}
	\end{displaymath}
	then we can again estimate $f_{N}$ with
	\begin{displaymath}
	\norm{f_{N}}_{q-1} \le \mu(\Sigma_{N})^{\frac{q_0}{q-1}-1}\norm{f_{N}}_{q_{0}}.
	\end{displaymath}
	Note that for $q-1 < 1$, we have
	\begin{align*}
	&\left|\norm{G_{\lambda}(e^{-\omega s}U_{N})}_{q-1}^{q-1}-\norm{G_{\lambda}(e^{-\omega s}u)}_{q-1}^{q-1}\right|\\
	&\qquad \le \int_{\Sigma}\left|\left|G_{\lambda}(e^{-\omega s}U_{N})\right|^{q-1}-\left|G_{\lambda}(e^{-\omega s}u)\right|^{q-1}\right|\dmu\\
	&\qquad\le \int_{\Sigma}\left|f_{N}\right|^{q-1}\dmu
	\end{align*}
	so that
	\begin{displaymath}
	\lim\limits_{N\rightarrow\infty}\norm{G_{\lambda}(e^{-\omega s}U_{N}(s))}_{q-1}^{q-1} = \norm{G_{\lambda}(e^{-\omega s}u(s))}_{q-1}^{q-1}.
	\end{displaymath}
	For the last term, note that $g_{N} \rightarrow g$ in $L^{1}(0,T;L^{q}_{\mu})$ as $N \rightarrow \infty$. So by a corollary of Riesz-Fischer, there exists a subsequence $N_k$ and a function $h \in L^{1}(0,T;L^{q}_{\mu})$ such that $|g_{N_k}(x)| \le h(x)$ for all $k$ and a.e.~$x \in \Sigma$. Similarly, interpolating between $q_0$ and $q+\varepsilon$,
	\begin{displaymath}
	\begin{split}
	\norm{G_{\lambda}(e^{-\omega s}U_{N}(s))-G_{\lambda}(e^{-\omega s}u(s))}_{q}&\le \norm{G_{\lambda}(e^{-\omega s}U_{N}(s))-G_{\lambda}(e^{-\omega s}u(s))}_{q_0}^{\theta}\times\\
	&\qquad\norm{G_{\lambda}(e^{-\omega s}U_{N}(s))-G_{\lambda}(e^{-\omega s}u(s))}_{q+\varepsilon}^{1-\theta}
	\end{split}
	\end{displaymath}
	for some $\theta \in (0,1]$. Then with $U_{N}$ and $u$ bounded in $L^{q+\varepsilon}_{\mu}$, $G_{\lambda}(e^{-\omega s}U_{N}(s))\rightarrow G_{\lambda}(e^{-\omega s}u(s))$ in $L^{q_0}_{\mu}$ as $N \rightarrow \infty$ uniformly for all $t \in [0,T]$. Hence taking another subsequence we have a dominant $H_{\lambda} \in L^{\infty}(0,T;L^{q}_{\mu})$. So we can estimate the integrand pointwise
	\begin{displaymath}
	\left|G_{\lambda}(e^{-\omega s}U_{N_{k}}(s))^{q-1}g_{N}\right| \le H_{\lambda}^{q-1}h\quad \text{a.e.~on } \Sigma\times[0,T).
	\end{displaymath}
	Moreover this dominant is in $L^{1}(0,T;L^{1}_{\mu})$ with
	\begin{displaymath}
	\int_{t_1}^{t_2}\int_{\Sigma}H_{\lambda}^{q-1}h\dmu \ds \le \norm{H_{\lambda}}_{L^{\infty}(0,T;L^{q}_{\mu})}\norm{h}_{L^{1}(0,T;L^{q}_{\mu})}.
	\end{displaymath}
	Hence we apply dominated convergence to obtain the continuous estimate\\~\eqref{eq:151}.
\end{proof}

We now show that this pointwise estimate~\eqref{eq:5Glambda1} for Proposition \ref{lem:discreteToSmoothEstimateProof} implies a similar estimate when adding a Lipschitz perturbation.
\begin{lemma}\label{le:Lipschitz-perturbation}
	For $1 \le q < \infty$, let $A$ be an operator on $L^{q}_{\mu}$
	and suppose there are $1 \le r \le \infty$, $\sigma > 0$,
	$\omega \in \mathbb{R}$, $\lambda \ge 0$ and $C > 0$ such that~\eqref{eq:5Glambda1} is satisfied for all $(u,v) \in A$. Let $F:L^{q}_{\mu}\to L^{q}_{\mu}$ be Lipschitz continuous with Lipschitz constant $\omega' \ge 0$ and satisfying $F(0) = 0$. Then, the operator $A+F$ in $L^{q}_{\mu}$ satisfies
	\begin{equation}
	\label{eq:pointwiseEstimateLipschitz}
	\norm{G_{\lambda}(u)}_{r}^{\sigma}\le C
	\; [G_{\lambda}(u),v+(\omega+\omega')(G_{\lambda}(u)+\lambda \sign(u))]_{q}.
	\end{equation}
\end{lemma}

\begin{proof}
	Let $\hat{v} = v+F(u)$. Then, since $[\cdot,\cdot]_{q}$ is
	linear in the second term,
	\begin{equation}
	\label{eq:lem47eq1}
	\begin{split}
	[G_{\lambda}(u),\hat{v}+&(\omega+\omega')(G_{\lambda}(u)+\lambda \sign(u))]_{q}\\
	& = [G_{\lambda}(u),v+\omega(G_{\lambda}(u)+\lambda \sign(u))]_{q}\\
	&\quad+[G_{\lambda}(u),F(u)+\omega'(G_{\lambda}(u)+\lambda \sign(u))]_{q}.
	\end{split}
	\end{equation}
	By the Lipschitz condition, 
	\begin{displaymath}
	-\omega'|u| \le F(u) \le \omega'|u|.
	\end{displaymath}
	Hence
	\begin{displaymath}
	\begin{split}
	[G_{\lambda}(u),F(u)]_{q}& \ge -\omega'[G_{\lambda}(u),u]_{q}\\
	& = -\omega'[G_{\lambda}(u),G_{\lambda}(u)+\lambda\sign(u)]_{q}.
	\end{split}
	\end{displaymath}
	So applying this and the initial assumption to~\eqref{eq:lem47eq1}, we have~\eqref{eq:pointwiseEstimateLipschitz}.
\end{proof}

We now extend the $L^q_{\mu}-L^\infty_{\mu}$ regularity of Theorem~\ref{thm:LqLinfty-deGiorgi} to obtain $L^{\ell}_{\mu}-L^{\infty}_{\mu}$ regularity as in Theorem~\ref{thm:LellLinfty}. 
For this we consider~\eqref{eqn:lqlinfty} applied to $\tilde{u}_0 = u(\tfrac{t}{2})$ and $\tilde{g}(s) = g(s+\tfrac{t}{2})$. 
The following theorem is an extension of~\cite[Chapter 6]{CoulHau2016} and in particular with $r = \infty$ gives the desired regularity result.
\begin{theorem}
	\label{thm:CH5.2}
	For $1 \le \ell < q < r \le \infty$ and $T > 0$, let $g \in L^{1}(0,T;L^{\ell}_{\mu}\cap L^{q}_{\mu})$ and $u\in L^{\infty}(0,T;L^{\ell}_{\mu}\cap L^{r}_{\mu})$ satisfying the exponential growth property~\eqref{eq:growthconditionGlambda} for some $\omega \ge 0$. Suppose there exist increasing functions $c_1(t)$, $c_2(t)$  with $c_1(t) > 0$ and $c_2(t) \ge 0$ for all $t \in [0,T]$ and exponents $\alpha \ge 0$, $0 < \gamma^* \le \gamma < \infty$ such that
	\begin{equation}
	\label{eq:LqLrestimateAssumption}
	\begin{split}
	\norm{u(t)}_{r}& \le \max\left\{c_1(\tfrac{t}{2})\left(\tfrac{2}{t}+\omega\right)^{\alpha}\left(\norm{u(\tfrac{t}{2})}_{q}+\norm{g}_{L^{1}(\frac{t}{2},t;L^{q}_{\mu})}\right)^{\gamma},\right.\\
	&\hspace{3.5cm}\left.c_2(\tfrac{t}{2})\left(\norm{u(\tfrac{t}{2})}_{q}+\norm{g}_{L^{1}(\frac{t}{2},t;L^{q}_{\mu})}\right)^{\gamma^*}\right\}
	\end{split}
	\end{equation}
	for every $t \in (0,T]$. Define
	\begin{displaymath}
	\theta := 1-\gamma\left(\frac{\frac{1}{\ell}-\frac{1}{q}}{\frac{1}{\ell}-\frac{1}{r}}\right)
	\end{displaymath}
	and suppose that $\theta > 0$. Then one has the $L^{\ell}_{\mu}-L^{r}_{\mu}$ estimate
	\begin{equation}
	\label{eq:LlLrGeneralEstimate}
	\begin{split}
	\norm{u(t)}_{r}& \le 2^{\gamma}\left(2^{\frac{\alpha}{\gamma\theta}}+\sup_{s\in (0,t]}N(s)^{\theta}\right)^{\frac{\gamma}{\theta}}\max \left\{c_1(t)^{\frac{1}{\theta}}\left(\tfrac{1}{t}+\omega\right)^{\frac{\alpha}{\theta}},c_2(t)^{\frac{1}{\theta}}\right\}\times\\
	&\hspace{3cm}\times\left(e^{\omega t}\norm{u_0}_\ell+\int_0^t e^{\omega(t-\tau)}\norm{g(\tau)}_\ell \dtau\right)^{\frac{\theta_\ell \gamma}{\theta}}
	\end{split}
	\end{equation}
	for all $t \in (0,T]$ where
	\begin{displaymath}
    \begin{split}
    N(t)& := \sup\limits_{s \in (0,t]}\frac{M(\tfrac{s}{2})\norm{g}_{L^1(\frac{s}{2},s;L^q_{\mu})}+c_2(\tfrac{s}{2})^{\frac{1}{\gamma}}}
	{M(s)^{\frac{1}{\theta}}\left(e^{\omega s}\norm{u_0}_{\ell}+\int_{0}^{s}e^{\omega(s-\tau)}\norm{g(\tau)}_{\ell}\dtau\right)^{\frac{\theta_\ell}{\theta}}},\\
    M(t)& := \max \left\{c_1(t)^{\frac{1}{\gamma}}\left(\tfrac{1}{t}+\omega\right)^{\frac{\alpha}{\gamma}},c_2(t)^{\frac{1}{\gamma}}\right\},
    \end{split}
	\end{displaymath}
	and
	\begin{displaymath}
	\theta_{\ell} := 
	\frac{\frac{1}{q}-\frac{1}{r}}{\frac{1}{\ell}-\frac{1}{r}}.
	\end{displaymath}
\end{theorem}

\begin{proof}
	We first note that since $\gamma^* \le \gamma$, we can estimate
	\begin{displaymath}
	\begin{split}
	\left(\norm{u(\tfrac{t}{2})}_{q}+\norm{g}_{L^{1}(\frac{t}{2},t;L^{q}_{\mu})}\right)^{\gamma^*}& \le \max\left\{1,\left(\norm{u(\tfrac{t}{2})}_{q}+\norm{g}_{L^{1}(\frac{t}{2},t;L^{q}_{\mu})}\right)^{\gamma}\right\}
	\end{split}
	\end{displaymath}
    for $t \in [0,T]$.
	Hence, taking~\eqref{eq:LqLrestimateAssumption} to the power $\frac{1}{\gamma}$, we have
	\begin{displaymath}
	\begin{split}
	\norm{u(t)}_{r}^{\frac{1}{\gamma}}& \le \max\left\{c_1(\tfrac{t}{2})^{\frac{1}{\gamma}}\left(\tfrac{2}{t}+\omega\right)^{\frac{\alpha}{\gamma}}\left(\norm{u(\tfrac{t}{2})}_{q}+\norm{g}_{L^{1}(\frac{t}{2},t;L^{q}_{\mu})}\right),\right.\\
	&\hspace{2.5cm}\left.c_2(\tfrac{t}{2})^{\frac{1}{\gamma}}\left(\norm{u(\tfrac{t}{2})}_{q}+\norm{g}_{L^{1}(\frac{t}{2},t;L^{q}_{\mu})}\right),c_2(\tfrac{t}{2})^{\frac{1}{\gamma}}\right\}.
	\end{split}
	\end{displaymath}
	We then apply standard interpolation on the $L^{q}_{\mu}$ norm with exponent $\theta_\ell$ and  the growth estimate~\eqref{eq:growthconditionGlambda} on $\norm{u(\tfrac{t}{2})}_\ell$ to obtain
	\begin{displaymath}
	\begin{split}
	\norm{u(t)}_{q}
	& \le \left(e^{\omega t}\norm{u_{0}}_{\ell}+\int_{0}^{t}e^{\omega(t-\tau)}\norm{g(\tau)}_{\ell}\td\tau\right)^{\theta_{\ell}}\norm{u(t)}_{r}^{1-\theta_\ell}
	\end{split}
	\end{displaymath}
    for all $t \in [0,T]$.
	By choice of $\theta$, $\theta_\ell$ we have the relation
	\begin{displaymath}
	\gamma(1-\theta_{\ell}) = 1-\theta.
	\end{displaymath}
	Then for all $t \in (0,T]$,
	\begin{equation}
	\label{eq:LlLrIntermediateEstimate}
	\begin{split}
	\norm{u(t)}_{r}^{\frac{1}{\gamma}}& \le M(\tfrac{t}{2})\left(e^{\frac{\omega t}{2}}\norm{u_{0}}_{\ell}+\int_{0}^{\tfrac{t}{2}}e^{\omega(\tfrac{t}{2}-\tau)}\norm{g(\tau)}_{\ell}\right)^{\theta_{\ell}}\norm{u(\tfrac{t}{2})}_{r}^{\frac{1-\theta}{\gamma}}\\
	&\quad+M(\tfrac{t}{2})\norm{g}_{L^{1}(\frac{t}{2},t;L^{q}_{\mu})}+c_2(\tfrac{t}{2})^{\frac{1}{\gamma}}.
	\end{split}
	\end{equation}

    We aim to produce comparable terms on either side of this equation. Since $c_1$ and $c_2$ are increasing, $M(\tfrac{t}{2}) \le 2^{\frac{\alpha}{\gamma}}M(t)$. Furthermore, 
    \begin{displaymath}
        e^{\frac{\omega t}{2}}\norm{u_{0}}_{\ell}+\int_{0}^{\tfrac{t}{2}}e^{\omega(\tfrac{t}{2}-\tau)}\norm{g(\tau)}_{\ell} \le e^{-\frac{\omega t}{2}}\left(e^{\omega t}\norm{u_{0}}_{\ell}+\int_{0}^{t}e^{\omega(t-\tau)}\norm{g(\tau)}_{\ell}\right).
    \end{displaymath}
    Hence we define
    \begin{equation}
    \label{eq:Kudef}
	K_{u}(t) := \frac{M(t)^{-\frac{1}{\theta}}\norm{u(t)}_{r}^{\frac{1}{\gamma}}}{\left(e^{\omega t}\norm{u_{0}}_{\ell}+\int_{0}^{t}e^{\omega(t-\tau)}\norm{g(\tau)}_{\ell}\dtau\right)^{\frac{\theta_{\ell}}{\theta}}}\quad \text{for } t \in [0,T]
	\end{equation}
    in order to estimate and rearrange~\eqref{eq:LlLrIntermediateEstimate} into a relation involving $K_u$ and $N$. Note that if the denominator of \eqref{eq:Kudef} is zero, we can add some arbitrarily small $\varepsilon > 0$ to the denominator of both $K_u$ and $F$, taking $\varepsilon \rightarrow 0$ in the final estimate. Now we fix $t \in (0,T]$ so that after rearranging we may take a supremum over $(0,t]$ to obtain
	\begin{displaymath}
	\sup_{s\in (0,t]}K_{u}(s) \le 2^{\frac{\alpha}{\gamma\theta}}\sup_{s\in (0,\frac{t}{2}]}K_{u}(s)^{1-\theta}+N(t).
	\end{displaymath}

    We now aim to split the forcing term $N(t)$ so as to incorporate this into each $K_u$ term. For this we define $D(t) \ge (2^{\frac{\alpha}{\gamma\theta}})^{\frac{1}{\theta}}$ for $t \in [0,T]$ such that
	\begin{equation}
    \label{eq:Dtdef}
	D(t)-2^{\frac{\alpha}{\gamma\theta}}D(t)^{1-\theta} = N(t).
	\end{equation}
	Noting that $\theta \in (0,1)$, this is possible as the function $f(x) = x-cx^\alpha$ for $c \ge 0$ and $\alpha \in (0,1)$ is continuous, satisfies $f(c^{\frac{1}{1-\alpha}}) = 0$ and is strictly increasing for $x > (\alpha c)^{\frac{1}{1-\alpha}}$ (in particular for $x \ge c^{\frac{1}{1-\alpha}}$). Further, we can estimate \eqref{eq:Dtdef} by
    \begin{align*}
        N(t)& = \left(D(t)^{\theta}\right)^{\frac{1-\theta}{\theta}}\left(D(t)^{\theta}-2^{\frac{\alpha}{\gamma\theta}}\right)\\
        & \ge \left(D(t)^{\theta}-2^{\frac{\alpha}{\gamma\theta}}\right)^{\frac{1}{\theta}}
    \end{align*}
    so that
    \begin{equation}
        \label{eq:Dtestimate}
        D(t) \le \left(2^{\frac{\alpha}{\gamma\theta}}+N(t)^{\theta}\right)^{\frac{1}{\theta}}.
    \end{equation}
    Then,
	\begin{displaymath}
	\sup_{s\in (0,t]}K_{u}(s)-D(t) \le 2^{\frac{\alpha}{\gamma\theta}}\left(\left(\sup_{s\in \left(0,\tfrac{t}{2}\right]}K_{u}(s)\right)^{1-\theta}-D(t)^{1-\theta}\right).
	\end{displaymath}
	Noting that $K_{u}(s)$ is bounded for all $s \in [0,T]$, either
	\begin{displaymath}
	\sup_{s\in (0, t]}K_{u}(s) \le D(t)
	\end{displaymath}
	or we can extend to a supremum over $(0,t]$ and combine terms, obtaining
	\begin{displaymath}
	\begin{split}
	\sup_{s\in(0,t]}K_{u}(s)-D(t)& \le 2^{\frac{\alpha}{\gamma\theta}}\left(\sup_{s\in (0,t]}K_{u}(s)-D(t)\right)^{1-\theta}\\
	\left(\sup_{s\in(0,t]}K_{u}(s)-D(t)\right)^{\theta}& \le 2^{\frac{\alpha}{\gamma\theta}}.
	\end{split}
	\end{displaymath}
	In either case, we have the uniform bound,
	\begin{displaymath}
	K_{u}(s) \le \left(2^{\frac{\alpha}{\gamma\theta}}\right)^{\frac{1}{\theta}}+D(t).
	\end{displaymath}
    for all $s\in (0,t]$. Applying \eqref{eq:Dtestimate},
	\begin{align*}
	K_{u}(t)& \le \left(2^{\frac{\alpha}{\gamma\theta}}\right)^{\frac{1}{\theta}}+\left(2^{\frac{\alpha}{\gamma\theta}}+N(t)^{\theta}\right)^{\frac{1}{\theta}}\\
	& \le 2\left(2^{\frac{\alpha}{\gamma\theta}}+N(t)^{\theta}\right)^{\frac{1}{\theta}}.
	\end{align*}
	Rewriting this as an estimate on $\norm{u(t)}_{r}$ we obtain~\eqref{eq:LlLrGeneralEstimate}.
\end{proof}

\section{$L^{1}-L^{\infty}$ regularization}
\label{sec:applicationToFractionalPLaplacian}

By Theorem~\ref{thm:nonlinearpLapmTAccretive} and the proof of Theorem~\ref{thm:1} we know that $\overline{(-\Delta_{p})^{s}\varphi}^{\mbox{}_{L^{1}}}+F$ is m-accretive in $L^1$ with complete resolvent where $F$ is the Nemytskii operator of $f(\cdot,u)$ satisfying~\eqref{hyp:2}-\eqref{hyp:3}. 
Hence to apply Theorem~\ref{thm:LqLinfty-deGiorgi} we first prove the pointwise estimate~\eqref{eq:5Glambda1} for the operator $\overline{(-\Delta_{p})^{s}\varphi}^{\mbox{}_{L^{1}}}$ in $L^{1}$, giving Proposition~\ref{lem:discreteToSmoothEstimateProof} and thereby outlining the proof of the $L^{m+1}-L^\infty$ estimate of Theorem~\ref{thm:DeGeorgiFractionalPLaplacian}.
We then apply Theorem \ref{thm:CH5.2}, proving the $L^{\ell}-L^\infty$ estimate of Theorem \ref{thm:LellLinfty}. 

The following lemma allows us to estimate the $q$-bracket $[G_\lambda (u),(-\Delta_p)^s(u^m)]_{m+1}$. 
In particular, the restriction $m \ge 1$ in this lemma results in the same restriction in Theorem~\ref{thm:LellLinfty}. 
Recall that we use the notation $r^m = |r|^{m-1}r$ for powers and that $G_{\lambda}(r) =[\abs{r}-\lambda]^{+}\sign(r)$ for $r \in \R$.

\begin{lemma}
	\label{le:GlambdaPointwiseEstimateCases}
	Let $1 < p < \infty$, $m \ge 1$ and $\lambda \ge 0$. Given $a$, $b \in \mathbb{R}$ define $a_{\lambda} = G_{\lambda}(a)$, $b_{\lambda} = G_{\lambda}(b)$. Then,
	\begin{align*}
	(a^{m}-b^{m})^{p-1}(a_{\lambda}^{m}-b_{\lambda}^{m}) \ge |a_{\lambda}^{m}-b_{\lambda}^{m}|^{p}.
	\end{align*}
\end{lemma}

\begin{proof}
	If $a_{\lambda}-b_{\lambda} = 0$, both sides are $0$. For $a_{\lambda}-b_{\lambda} \neq 0$,
	\begin{align*}
	\sign(a_{\lambda}^{m}-b_{\lambda}^{m})& = \sign(a^{m}-b^{m})\\
	|a^{m}-b^{m}|^{p-2}(a^{m}-b^{m})(a_{\lambda}^{m}-b_{\lambda}^{m})& = |a^{m}-b^{m}|^{p-1}|a_{\lambda}^{m}-b_{\lambda}^{m}|
	\end{align*}
	so we only need to prove that $|a_{\lambda}^{m}-b_{\lambda}^{m}| \le |a^{m}-b^{m}|$. We take cases, assuming without loss of generality that $|a| \ge |b|$. First suppose that $|a| \le \lambda$ or $|b| \le \lambda$ so that $a_\lambda = b_\lambda = 0$ and the inequality is clear. Next, if $|b| \le \lambda$ and $|a| > \lambda$, then
	\begin{displaymath}
	|a_{\lambda}^m-b_\lambda^m| = (|a|-\lambda)^m \le (|a|-|b|)^m \le |a|^m-|b|^m \le |a^m-b^m|.
	\end{displaymath}
	
	Hence we consider cases for $a$, $b$ corresponding to $|a| > \lambda$ and $|b| > \lambda$. Suppose that $a > \lambda$ and $b > \lambda$. Then noting that for $m \ge 1$, $|x|^{m-1}$ is non-decreasing on the set $[0,\infty)$,
	\begin{align*}
	a^{m}-b^{m}& = \int_{b}^{a}\frac{\td }{\td x}x^{m}\dx\\
	& \ge \int_{b}^{a}m|x-\lambda|^{m-1}\dx\\
	& = (a-\lambda)^{m}-(b-\lambda)^{m}.
	\end{align*}
	Similarly for $a < -\lambda$, $b < - \lambda$, noting that $|x|^{m-1}$ is non-increasing on $(-\infty,0]$,
	\begin{displaymath}
	a^{m}-b^{m} \ge (a+\lambda)^{m}-(b+\lambda)^{m}.
	\end{displaymath}
	Finally, suppose that $a > \lambda$ and $b < -\lambda$ (similarly for $a < -\lambda$ and $b > \lambda$). Then
	\begin{align*}
	|a_{\lambda}^{m}-b_{\lambda}^{m}|& = (a-\lambda)^{m}-(b+\lambda)^{m}\\
	& \le a^{m}-b^{m}\\
	& = |a^{m}-b^{m}|.
	\end{align*}
\end{proof}

We can now derive~\eqref{eq:5Glambda1}, the pointwise estimate for Proposition~\ref{lem:discreteToSmoothEstimateProof}, giving the Sobolev-type inequality required for Theorem \ref{thm:LqLinfty-deGiorgi} in the case of~\eqref{eq:1}. Recall the notation of $q$-brackets from Section~\ref{sec:pre-ns}. In this case with $q = m+1$ for $m \ge 1$, the $q$-bracket is given by
\begin{displaymath}
[u,v]_{m+1} = \int_\Omega \abs{u}^{m-1}u\,v\dmu
\end{displaymath}
for every $u$, $v\in L^{m+1}$.
\begin{lemma}
	\label{lem:42}
	Let $\Omega$ be an open set in $\mathbb{R}^{d}$, $d \ge 1$. For $1 < p < \infty$, $0 < s < 1$, define $p_{\mbox{}_{\! s}}$ according to the Sobolev embedding~\eqref{eq:6} and let $\varphi(r) = r^{m}$ for $r \in \mathbb{R}$ where $m \ge 1$. Then $(-\Delta_{p})^{s}\varphi$ satisfies the one-parameter Sobolev type inequality
	\begin{equation}
	\label{eq:5Glambda2}
	\norm{G_{\lambda}(u)}_{mp_{\mbox{}_{\! s}}}^{mp}\le C_d
	\; [G_{\lambda}(u),(-\Delta_p)^{s}\varphi(u)]_{m+1}
	\end{equation}
	for all $\varphi(u) \in W^{s,p}_{0}$ and $\lambda \ge 0$. In particular, this is~\eqref{eq:5Glambda1} with $q = m+1$, $\omega=0$, $\sigma = mp$, $r = mp_s$ and $C = C_d$.
\end{lemma}

\begin{proof}
	Let $u^m \in W^{s,p}_{0}$. By Lemma~\ref{le:GlambdaPointwiseEstimateCases} one sees that
	\begin{displaymath}
	\begin{split}
	&[G_{\lambda}u,(-\Delta_{p})^{s}(u^{m})]_{m+1}\\
	&\qquad =
	\int_{\mathbb{R}^{d}}\int_{\mathbb{R}^{d}}\frac{\abs{((u(t,x))^{m}-(u(t,y))^{m}}^{p-2}((u(t,x))^{m}-(u(t,y))^{m})}{|x-y|^{d+sp}}\times\\
	&\hspace{5cm} ((G_{\lambda}u(t,x))^{m}-(G_{\lambda}u(t,y))^{m})\dx\dy\\
	&\qquad \ge
	\int_{\mathbb{R}^{d}}\int_{\mathbb{R}^{d}}\frac{|(G_{\lambda}u(t,x))^{m}-(G_{\lambda}u(t,y))^{m}|^{p}}{|x-y|^{d+sp}}
	\dx\dy
	\end{split}
	\end{displaymath}
	Hence we can estimate by a semi-norm,
	\begin{displaymath}
	\begin{split}
	[G_{\lambda}u,(-\Delta_{p})^{s}(u^{m})]_{m+1}
	& \ge [(G_{\lambda}u(t,x))^{m}]_{s,p}^{p}\\
	& \ge
	\frac{1}{C_{d}}\,\norm{(G_{\lambda}(u))^{m}}_{p_{\mbox{}_{\! s}}}^{
		p}\\
	& =
	\frac{1}{C_{d}}\,\norm{G_{\lambda}(u)}_{mp_{\mbox{}_{\! s}}}^{m
		p}
	\end{split}
	\end{displaymath}
	by the classical Sobolev inequality for Gagliardo semi-norms (cf.~\cite{MR2527916})
	\begin{displaymath}
	\norm{u}_{p_{\mbox{}_{s}}}\le C_{d}\,[u]_{s,p}
	\end{displaymath}
	where $p_{\mbox{}_{s}}$ is given by~\eqref{eq:6}.
\end{proof}

With these preliminaries we can now apply Proposition~\ref{lem:discreteToSmoothEstimateProof} and Theorem~\ref{thm:LqLinfty-deGiorgi} to prove the $L^{m+1}-L^{\infty}$ regularisation effect for the doubly nonlinear nonlocal problem~\eqref{eq:1}.

\begin{theorem}
	\label{thm:DeGeorgiFractionalPLaplacian}
	Let $\Omega$ be an open set in $\mathbb{R}^{d}$, $d \ge 1$. Suppose $p > 1$, $0 < s< 1$, and $m \ge 1$ such that 
	\begin{equation}
	\tag{\ref{hyp:16}}
	m(p-1)+(m+1)\frac{sp}{d} > 1.
	\end{equation}
	Let $q_{s} = p_s$ if $p \neq \frac{d}{s}$ and $q_{s} > \max(p,1+\frac{1}{m})$ if $p=\frac{d}{s}$.	Let $T > 0$ and $g \in L^{1}(0,T;L^{1})\cap L^{\psi}(0,T;L^{\rho})\cap L^{1}(0,T;L^{m+1+\varepsilon})$ for some $\varepsilon > 0$ where $\rho \ge m+1$ and $\psi > 1$ satisfy
	\begin{equation}
	\tag{\ref{eq:psirhohyp1}}
    \begin{cases}
	\frac{1}{\rho} < \left(1-\frac{1}{\psi}\right)\left(1-\frac{p}{q_{s}}\right)& \text{if } m(p-1) \ge 1,\\
    \frac{1}{\rho} \le \left(1-\frac{1}{\psi}\right)p\left(\frac{m}{m+1}-\frac{1}{q_{s}}\right)& \text{if } m(p-1) < 1.
    \end{cases}
	\end{equation}
	Let $u(t)$ be the mild solution to~\eqref{eq:1} for $u_{0} \in L^1 \cap L^{m+1}$. Then one has that
	\begin{equation}
	\label{eq:2}
	\begin{split}
	\norm{u(t)}_{\infty}&\le C\max\Biggr\{e^{\beta_{1}\omega t}\left(\tfrac{1}{t}+\omega\right)^{\alpha}\left(\norm{u_0}_{m+1}+\norm{g}_{L^1(0,t;L^{m+1})}\right)^{\gamma},\\
	&\hspace{2.7cm} e^{\beta_{2} \omega t}\norm{g}_{L^{\psi}(0,t;L^{\rho})}^{\eta}\left(\norm{u_0}_{m+1}+\norm{g}_{L^1(0,t;L^{m+1})}\right)^{\gamma_{\psi}}\Biggr\}
	\end{split}
	\end{equation}
	for all $0 < t \le T$ where we have the exponents
	\allowdisplaybreaks
	\begin{equation}
	\label{eq:46}
	\begin{split}
    \gamma& = \frac{\frac{1}{p}-\frac{1}{q_{s}}}{\frac{m}{m+1}-\frac{1}{q_{s}}},\quad
	\gamma_{\psi} = \frac{\left(1-\frac{1}{\psi}\right)\left(\frac{1}{p}-\frac{1}{q_{s}}\right)-\frac{1}{\rho p}}{\left(1-\frac{1}{\psi}\right)\left(\frac{m}{m+1}-\frac{1}{q_{s}}\right)-\frac{1}{\rho p}+\frac{1}{(m+1)p}},\\
	\alpha& = \frac{1}{mp(1-\frac{m+1}{mq_{s}})},\quad 
    \eta = \frac{1}{mp \left(1-\frac{1}{\psi}\right)\left(1-\frac{m+1}{mq_{s}}\right)+1-\frac{m+1}{\rho}},
	\end{split}
	\end{equation}
	and
	\begin{equation}
	\begin{split}
	\beta_{1}& = \begin{cases}
	\frac{\frac{1}{mp}-\frac{1}{m+1}}{\frac{1}{m+1}-\frac{1}{mq_{s}}}& \text{if } m(p-1) < 1,\\
	0& \text{if } m(p-1) \ge 1,
	\end{cases}\\
	\beta_{2}& = \begin{cases}
	\eta(m+1-mp)\left(1-\frac{1}{\psi}\right)& \text{if } m(p-1) < 1,\\
	0& \text{if } m(p-1) \ge 1.
	\end{cases}
	\end{split}
	\end{equation}
\end{theorem}

\begin{proof}
	Let $F$ be the Nemytskii operator of $f(\cdot,u)$ and $\E$ the energy functional \eqref{eq:energyFunctionalpLap}. 
    Then by Theorem~\ref{thm:1} and the proof thereof, $\overline{(\partial\E)_{\vert_{L^{1\cap\infty}}}\varphi}^{\mbox{}_{L^1}}+F$ is $\omega$-quasi $m$-accretive in $L^1$ and a mild solution to~\eqref{eq:1} exists in $L^1$ for all $u_0 \in L^{1}$. 
    By Lemma~\ref{lem:42} and Proposition~\ref{le:Lipschitz-perturbation} we have that $(-\Delta_p)^s \varphi+F$ satisfies the one-parameter Sobolev inequality~\eqref{eq:5Glambda1} with $q = m+1$, $\sigma = mp$, $r = mq_{s}$, $C = C_d$ and $\omega$ the Lipschitz constant of $f$. 
    Note that for $q_{s}$ we have chosen $\tilde{p}$ in the Sobolev embedding~\eqref{eq:6}. 
    Then we can apply Proposition~\ref{lem:discreteToSmoothEstimateProof} with $q_0 = 1$ to obtain an estimate of the form~\eqref{eq:15} with $L$ given by $\frac{m+1}{C_d}$. 
    To satisfy the conditions on $q$, $\sigma$, $r$ we first note that all are in $[1,\infty]$ with $q$ and $\sigma$ finite given that $p < \infty$. 
    Then $\sigma < r$ is equivalent to $p < q_{s}$. 
    This is clear when $p \neq \frac{d}{s}$. 
    However in the case $p = \frac{d}{s}$ we must choose $q_{s} > p$. 
    For $q < r$, we require that $m+1 < mq_{s}$. 
    When $p < \frac{d}{s}$, this implies that $1+\frac{1}{m} < (\frac{1}{p}-\frac{s}{d})^{-1}$, equivalent to~\eqref{hyp:16}, and when $p = \frac{d}{s}$, we choose $q_{s} > 1+\frac{1}{m}$. 
    In the case $p > \frac{d}{s}$ the inequality is clear. 
    We now apply Theorem~\ref{thm:LqLinfty-deGiorgi} to obtain~\eqref{eq:2}. For this we also need $m+1 \le \rho \le \infty$ and $1 < \psi \le \infty$ to satisfy~\eqref{hyp:31}, hence requiring~\eqref{eq:psirhohyp1}.
\end{proof}

We now extend this to the $L^{\ell}-L^{\infty}$ regularisation estimate of Theorem~\ref{thm:LellLinfty} for the doubly nonlinear nonlocal problem~\eqref{eq:1} by applying Theorem~\ref{thm:CH5.2}.

\begin{proof}[Proof of Theorem~\ref{thm:LellLinfty}]
	We apply Theorem~\ref{thm:DeGeorgiFractionalPLaplacian} to $\tilde{u}(s) = u(s+\frac{t}{2})$ and $\tilde{g}(s) = g(s+\tfrac{t}{2})$ to obtain~\eqref{eq:LqLrestimateAssumption} for all $t \in (0,T]$ with
	\begin{displaymath}
	c_1(t) = C e^{\beta_1 \omega t},\quad
	c_2(t) = C e^{\beta_2 \omega t}\norm{g}_{L^\psi(0,t;L^\rho)}^\eta,\quad
	\gamma = \frac{\frac{1}{p}-\frac{1}{q_{s}}}{\frac{m}{m+1}-\frac{1}{q_{s}}}.
	\end{displaymath}
	We now aim to apply Theorem~\ref{thm:CH5.2}, taking $q$ and $r$ in this theorem to be $m+1$ and $\infty$, respectively. By Corollary~\ref{cor:doublynonlinearAccretivity} we have that the operator $\overline{(\partial\E)_{\vert_{L^{1\cap\infty}}} \varphi}^{\mbox{}_{L^1}}+F$ is $\omega$-quasi $m$-accretive with complete resolvent. Hence by Lemma~\ref{le:GrowthConditionCompleteResolvent}, $u$ satisfies the exponential growth property~\eqref{eq:growthconditionGlambda}. We also require that $\gamma_{\psi} \le \gamma$, or equivalently,~\eqref{eq:psirhohyp2}. For the condition $\theta > 0$ of Theorem~\ref{thm:CH5.2}, we require that
	\begin{displaymath}
	1-\gamma\left(1-\frac{\ell}{m+1}\right) > 0.
	\end{displaymath}
	In particular, this is~\eqref{eq:45}. Hence we may apply Theorem~\ref{thm:CH5.2} to obtain~\eqref{eq:43}.
\end{proof}

%
%
%
%

\section{Strong solutions}
\label{sec:strongSols}

We now consider strong solutions, applying Theorem \ref{th:BGStrongExistence} to show that for $\varphi$ strictly increasing such that $\varphi^{-1} \in AC_{\loc}(\mathbb{R})$, mild solutions to~\eqref{eq:1} are in fact strong distributional solutions. 
Moreover we use the $L^{1}-L^{\infty}$ regularity estimate proved in Section \ref{sec:applicationToFractionalPLaplacian} to obtain the derivative estimates of Theorem~\ref{th:MildToStrongLipschitzFracPLap} in the case $m \ge 1$.

\medskip

\begin{proof}[Proof of Theorem~\ref{th:MildToStrongLipschitzFracPLap}]
    By Theorem~\ref{thm:1}, for every $u_{0}\in L^1$, there is a unique mild solution to Cauchy problem~\eqref{eq:1}. Now, let $u_{0}\in L^{1}$ such that $\varphi(u_{0})\in \tilde{D}((-\Delta_p)^s_{\vert_{L^{1\cap\infty}}})$. 
    Then, there exists a sequence $(v_n,w_n)\in (-\Delta_p)^s$ for $n \in \N$ such that $v_n \rightarrow \phi(u_0)$ in $L^1$ as $n \rightarrow \infty$ and $(w_n)_{n \in \N}$ is bounded in $L^1$. 
    Since $(v_n)_{n \in \N}$ is bounded in $L^{\infty}$ and $\Omega$ has finite measure, $\phi^{-1}(v_n) \in L^{1\cap\infty}$ uniformly for $n \in \N$.
    Let $(u_n)_{n \in \N}$ be the mild solutions with initial data $\phi^{-1}(v_n)$.
    By \cite[Lemma 7.8]{BenilanNonlinearEvolutionEqns} each mild solution $u_n$ is Lipschitz in time on $[0,T]$ with Lipschitz constant
    \begin{displaymath}
    L_n = e^{\omega t}\norm{g(0+)-w_n}_1+V(g,t+)+\omega\int_0^t e^{\omega(t-\tau)}V(g,\tau+)\dtau
    \end{displaymath}
    where
    \begin{displaymath}
        V(g,t+) = \limsup\limits_{h\rightarrow 0^{+}}\int_0^t \frac{\norm{g(\tau+h)-g(\tau)}_1}{h}\dtau.
    \end{displaymath}
    Since $w_n$ is bounded in $L^1$, $L_n$ is uniformly bounded by some $L > 0$ for all $n \in \N$.
    Then for $t \in [0,T]$ and $h > 0$, using the comparison estimate~\eqref{eq:theorem1.1estimate2},
    \begin{align*}
        \norm{u(t+h)-u(t)}_1& \le \norm{u_n(t+h)-u_n(t)}_1+(1+e^{\omega h})\norm{u(t)-u_n(t)}_1\\
        & \le Lh+2e^{\omega(t+h)}\norm{u_0-u_n(0)}_1
    \end{align*}
    for all $n \in \N$.
    Since $v_n$ converges to $\phi(u_0)$ in $L^1$, we have pointwise convergence of $u_n(0)$ to $u_0$ almost everywhere in $\Omega$. Hence, using the uniform bound for $v_n$ in $L^{\infty}$, we can take the limit supremum as $n \rightarrow \infty$ and apply Fatou's lemma to obtain
    \begin{displaymath}
        \norm{u(t+h)-u(t)}_1 \le Lh
    \end{displaymath}
    for all $t \in [0,T]$ and $h > 0$. So by the Lipschitz continuity of $F$ we have that $F(u) \in BV((0,T);L^{1})$.
    Moreover by Lemma~\ref{le:GrowthConditionCompleteResolvent}, since $\tilde{D}((-\Delta_p)^s_{\vert_{L^{1\cap\infty}}}) \subset L^\infty$, $u\in L^{\infty}([0,T];L^{\infty})$ and so $F(u) \in L^{1}(0,T;L^{\infty})$.
    Applying Theorem \ref{th:BGStrongExistence} with forcing term $\tilde{g} = -F(u)+g$ we have that $u \in W^{1,\infty}(0,T;L^{1})$ is a strong distributional solution to~\eqref{eq:1}. 
    The chain rule~\eqref{eq:ChainRuleForPhiFracPLap} follows from the proof of \cite[Theorem 4.1]{BenilanGariepy}.
    
    We have $u \in C([0,T];L^q)$ for $1 \le q < \infty$ due to the regularity of mild solutions in $L^1$. 
    Multiply the doubly nonlinear problem~\eqref{eq:1} by $\frac{\td }{\td t}\varphi(u)$ to obtain
	\begin{equation}
	\varphi'(u)\left|\frac{\td u}{\td t}\right|^{2}+\frac{\td }{\td t}\E(\varphi(u(t)))+\left(F(u(t))-g(t)\right)\frac{\td u}{\td t}(t)\varphi'(u(t)) = 0
	\end{equation}
	giving~\eqref{eq:DerivativeOfPsiPhiUFracPLap}.
\end{proof}

\medskip

We now introduce a lemma to obtain estimate~\eqref{eq:W12estimate} of Theorem~\ref{thm:derivativeEstimates} for a more general class of sub-differential operators. In particular, we consider the problem
\begin{equation}
\label{eq:61}
\begin{cases}
\begin{alignedat}{2}
u_{t}(t)+A\varphi(u(t))+f(\cdot,u(t))&=g(\cdot,t)\quad && 
\text{in $\Omega\times(0,T),$}\\
u(t)&=0\quad && \text{in $\R^{d}\setminus\Omega\times
	(0,T),$}\\
u(0)&=u_{0}\quad && \text{on $\Omega$,}
\end{alignedat}
\end{cases}
\end{equation}
where $A$ is the sub-differential in $L^1$ of a proper, lower semicontinuous convex functional $\E:L^2 \rightarrow (-\infty,\infty]$. 
Here we use the notation $\Phi(r) := \int_{0}^{r}\varphi(s)\ds$ for $r \in \R$ and $\varphi \in C(\R)$.

\begin{lemma}
	Let $A:=\left(\partial_{L^1}\E\right)_{\vert_{L^{1\cap\infty}}}$ be the sub-differential of a proper, lower semi\-continuous, convex functional $\E:L^2 \rightarrow (-\infty,\infty]$ satisfying~\eqref{eq:12} and $\E(0) = 0$. Let $\varphi\in C(\mathbb{R})$ be a strictly increasing function satisfying $\varphi(\R) = \R$ such that $\varphi^{-1} \in AC_{\loc}(\mathbb{R})$ and $\varphi(0) = 0$, and suppose that $f(\cdot,u)$ satisfies \eqref{hyp:2}-\eqref{hyp:3}. Then for every $\varphi(u_0) \in \tilde{D}(A)$ and $g \in BV(0,T;L^{1})\cap L^{1}(0,T;L^{\infty})$ the unique mild solution $u$ of~\eqref{eq:61} is strong in $L^1$ and for all $k > -1$, satisfies
	\begin{equation}
	\label{eq:W12estimateGeneral}
	\begin{split}
	\frac{1}{2}\int_{0}^{t}s^{k+2}&\int_{\Omega}\varphi'(u(s))\left|\frac{\td u}{\td s}\right|^{2}\dmu \ds+t^{k+2}\E(\varphi(u(t)))\\
	& \le (k+2)(k+1)\int_{0}^{t}s^{k}\int_{\Omega}\Phi(u(s))\dmu \ds\\
	&\quad+(k+2)\int_{0}^{t}s^{k+1}\int_{\Omega}\omega u\varphi(u)+g\varphi(u)\dmu \ds\\
	&\quad+\int_{0}^{t}\left((k+2)^2+\omega^{2}s^2\right)s^{k}\int_{\Omega} |u|^{2}\varphi'(u)\dmu \ds\\
	&\quad +\frac{1}{2}\int_{0}^{t}s^{k+2}\int_{\Omega}|g|^{2}\phi'(u)\dmu \ds.
	\end{split}
	\end{equation}
\end{lemma}

\begin{proof}
	Multiplying~\eqref{eq:61} by $s^{k+2}\frac{\td }{\td t}\varphi(u)$ for $k > -1$, we can estimate $u$ in\\ $W^{1,2}_{\loc}((0,T];L^{2})$, as in \cite{BenilanGariepy}, by
	\begin{equation}
	\label{eq:67_1}
	\begin{split}
	\int_{0}^{t}s^{k+2}&\int_{\Omega}\varphi'(u(s))\left|\frac{\td u}{\td s}\right|^{2}\dmu \ds+t^{k+2}\E(\varphi(u(t)))\\
	& = (k+2)\int_{0}^{t}s^{k+1}\E(\varphi(u))\ds\\
	&\quad+\int_{0}^{t}s^{k+2}\int_{\Omega} \left(g(s)-F(u)\right)\frac{\td v}{\td s}\dmu \ds.
	\end{split}
	\end{equation}
	Estimating $\int_{0}^{t}s^{k+1}\E(v)\ds$, we note that $A$ is the subgradient of $\E$, so
	\begin{displaymath}
	\braket{A(v) - A(0),0-v} \le \E(0) - \E(v)
	\end{displaymath}
	for $0 < t \le T$. Then
	\begin{equation}
	\label{eq:PsiLambdaEstimate}
	\begin{split}
	\E(v)& \le -\int_{\Omega}\frac{\td u}{\td t}v\dmu-\int_{\Omega}F(u)v\dmu+\int_{\Omega}gv\dmu
	\end{split}
	\end{equation}
	for $0 < t \le T$. Since $\varphi$ is increasing,
	\begin{displaymath}
	\Phi(r) \le \varphi(r)r \quad \text{ for all $r \in \mathbb{R}$.}
	\end{displaymath}
	Multiply~\eqref{eq:PsiLambdaEstimate} by $s^{k+1}$ and integrate over $(0,t)$ to obtain
	\begin{equation}
	\label{eq:IntegralEstimatePsiULambda}
	\begin{split}
	\int_{0}^{t}s^{k+1}&\E(v)\ds+t^{k+1}\int_{\Omega}\Phi(u(t))\dmu\\
	& \le (k+1)\int_{0}^{t}s^{k}\int_{\Omega}\Phi(u(s))\dmu \ds+\int_0^t s^{k+1}\int_{\Omega}u\frac{\td v}{\td s}\ds\\
	&\quad+\int_{0}^{t}s^{k+1}\int_{\Omega}\left(g(s)-F(u)\right)v\dmu \ds.
	\end{split}
	\end{equation}
	Since $\varphi(0) = 0$ and $\varphi$ is non-decreasing, $\Phi(r) \ge 0$ for all $r \in \mathbb{R}$ and $\varepsilon > 0$. Hence
	\begin{displaymath}
	\int_{\Omega}\Phi(u(t))\dmu \ge 0
	\end{displaymath}
	for all $0 < t \le T$. Then returning to~\eqref{eq:67_1} we can estimate $\int_{0}^{t}s^{k+1}\E(v)\ds$, giving
	\begin{displaymath}
	\begin{split}
	\int_{0}^{t}&s^{k+2}\int_{\Omega}\varphi'(u(s))\left|\frac{\td u}{\td s}\right|^{2}\dmu \ds+t^{k+2}\E(\varphi(u(t)))\\
	& = (k+2)(k+1)\int_{0}^{t}s^{k}\int_{\Omega}\Phi(u(s))\dmu \ds+(k+2)\int_0^t s^{k+1}\int_{\Omega}u\frac{\td v}{\td s}\ds\\
	&\quad+(k+2)\int_{0}^{t}s^{k+1}\int_{\Omega}\left(g(s)-F(u)\right)v\dmu \ds\\
	&\quad+\int_{0}^{t}s^{k+2}\int_{\Omega} \left(g(s)-F(u)\right)\frac{\td v}{\td s}\dmu \ds.
	\end{split}
	\end{displaymath}
	Further, applying the Lipschitz property of $F$,
	\begin{displaymath}
	\begin{split}
	\int_{0}^{t}s^{k+2}&\int_{\Omega}\varphi'(u(s))\left|\frac{\td u}{\td s}\right|^{2}\dmu \ds+t^{k+2}\E(\varphi(u(t)))\\
	&\le (k+2)(k+1)\int_{0}^{t}s^{k}\int_{\Omega}\Phi(u(s))\dmu \ds\\
	&\quad+\int_{0}^{t}\left(k+2+\omega s\right)s^{k+1}\int_{\Omega} |u|\varphi'(u)\left|\frac{\td u}{\td s}\right|\dmu \ds\\
	&\quad+(k+2)\int_{0}^{t}s^{k+1}\int_{\Omega}\omega u\varphi(u)+g\varphi(u)\dmu \ds\\
	&\quad + \int_{0}^{t}s^{k+2}\int_{\Omega}g\varphi'(u)\frac{\td u}{\td s}\dmu \ds.
	\end{split}
	\end{displaymath}
	Applying Young's inequality, we combine $\left|\frac{\td u}{\td s}\right|$ terms to obtain~\eqref{eq:W12estimateGeneral}.
\end{proof}

For the proof of Theorem~\ref{thm:derivativeEstimates}, we first prove~\eqref{eq:lipschitzContinuity} before applying~\eqref{eq:W12estimateGeneral} to the porous medium case with $\varphi(r) = r^{m}$ to obtain~\eqref{eq:W12estimate}.

\begin{proof}[Proof of Theorem~\ref{thm:derivativeEstimates}]
	The estimate~\eqref{eq:lipschitzContinuity} follows as an application of the regularity results for homogeneous operators presented in~\cite{BenilanRegularizingMR648452}. 
    In particular, we note that $(-\Delta_p)^s \cdot^m$ is homogeneous of order $\alpha = m(p-1)$, so we apply~\cite[Theorem 4]{BenilanRegularizingMR648452} with forcing term $\tilde{f}(t) = -F(u(t))+g(t)$. 
    Using the Lipschitz property of $F$, we have for all $t > 0$ and $u_0 \in L^1$,
    \begin{align*}
    \norm{u(t(1+\xi))-u(t)}_{1}& \le  |1-(1+\xi)^{\frac{1}{1-\alpha}}|\left(2\norm{u_0}_1+\int_0^t \omega\norm{u(\tau)}_1+\norm{g(\tau)}_1 \dtau\right)\\
    &+|1+\xi-(1+\xi)^{\frac{1}{1-\alpha}}|\left(\omega\int_{0}^t \norm{u(\tau(1+\xi)}_1+\norm{g(\tau(1+\xi))}_1\dtau\right)\\
    &+(1+\xi)^{\frac{1}{1-\alpha}}\int_0^t\norm{g(\tau(1+\xi)-g(\tau)}_1\dtau\\
    &+(1+\xi)^{\frac{1}{1-\alpha}}\omega\int_0^t\norm{u(\tau(1+\xi)-u(\tau)}_1\dtau.
    \end{align*}
    Applying Gr\"onwall's inequality,
    \begin{align*}
    \norm{u(t(1+\xi))-u(t)}_{1}& \le  \left(|1-(1+\xi)^{\frac{1}{1-\alpha}}|\left(2\norm{u_0}_1+\int_0^t \omega\norm{u(\tau)}_1+\norm{g(\tau)}_1 \dtau\right)\right.\\
    &+|1+\xi-(1+\xi)^{\frac{1}{1-\alpha}}|\left(\omega\int_{0}^t \norm{u(\tau(1+\xi)}_1+\norm{g(\tau(1+\xi))}_1\dtau\right)\\
    &+\left.(1+\xi)^{\frac{1}{1-\alpha}}\int_0^t\norm{g(\tau(1+\xi)-g(\tau)}_1\dtau\right)e^{(1+\xi)^{\frac{1}{1-\alpha}}\omega t}.
    \end{align*}
    We can divide through by $\xi$, taking the $\limsup$ as $\xi\rightarrow 0^{+}$,
    \begin{align*}
    \limsup_{\xi\rightarrow 0^+}\frac{\norm{u(t(1+\xi))-u(t)}_{1}}{\xi}& \le \frac{e^{\omega t}(1+\alpha)}{|1-\alpha|}\left(\int_0^t \omega\norm{u(\tau)}_1+\norm{g(\tau)}_1 \dtau\right)\\
    &\quad+e^{\omega t}\left(\frac{2\norm{u_0}_1}{|1-\alpha|}+V(t,g)\right)
    \end{align*}
    where $V(t,g)$ is defined by~\eqref{eq:VLipdef} and here $\alpha = m(p-1)$.
	So by the growth estimate on $u$ in $L^1$ given by Theorem~\ref{thm:1}, we can divide through again by $t$ to obtain~\eqref{eq:lipschitzContinuity}.
	
	To obtain~\eqref{eq:W12estimate}, let $A = \left(\partial_{L^1}\E\right)_{\vert_{L^{1\cap\infty}}}$ for $\E$ given by~\eqref{eq:energyFunctionalpLap}. 
    Now consider $u$ a solution to~\eqref{eq:1} with initial data satisfying $\varphi(u_{0}) \in \tilde{D}(A)$ so that we have estimate~\eqref{eq:W12estimateGeneral}. 
    Then for $\varphi(s) = s^{m}$ where $m \ge 1$, we estimate further by
	\begin{displaymath}
	\begin{split}
	\frac{1}{2}\int_{0}^{t}&s^{k+2}\int_{\Omega}\varphi'(u(s))\left|\frac{\td u}{\td s}\right|^{2}\dmu \ds+t^{k+2}\E(\varphi(u(t)))\\
	& \le \int_{0}^{t}\left((k+2)(k+1)+(k+2)^2 m+(k+2)\omega s+\omega^{2}ms^{2}\right)s^{k}\norm{u}_{m+1}^{m+1} \ds\\
	&\quad+(k+2)\int_{0}^{t}s^{k+1}\int_{\Omega}gu^m \dmu \ds+m\int_{0}^{t}s^{k+2}\int_{\Omega}|g|^{2}u^{m-1}\dmu \ds.
	\end{split}
	\end{displaymath}
	We estimate the terms involving $g$ by Young's inequality, giving
	\begin{displaymath}
	\begin{split}
	\frac{1}{2}\int_{0}^{t}s^{k+2}&\int_{\Omega}\varphi'(u(s))\left|\frac{\td u}{\td s}\right|^{2}\dmu \ds+t^{k+2}\E(\varphi(u(t)))\\
	& \le \int_{0}^{t}\left((k+2)(k+1)+(k+2)^2m\right)s^{k}\norm{u}_{m+1}^{m+1} \ds\\
	&\quad+\int_{0}^{t}\left((k+2)(\omega+1) s+m(\omega^{2}+1)s^{2}\right)s^{k}\norm{u}_{m+1}^{m+1} \ds\\
	&\quad+\int_{0}^{t}\left(k+2+m s\right)s^{k+1}\norm{g}_{m+1}^{m+1} \ds.
	\end{split}
	\end{displaymath}
	
	We estimate $\norm{u}_{m+1}$ by $\norm{u}_1$ by applying Theorem~\ref{thm:LellLinfty} to $\norm{u}_{\infty}$ and the standard growth estimate to $\norm{u}_1$,
	\begin{displaymath}
	\begin{split}
    	\norm{u(t)}_{m+1}^{m+1}& \le \norm{u(t)}_1 \norm{u(t)}_{\infty}^{m}\\
    	& \le C\max \left(e^{\omega\beta_1 t}\left(\tfrac{1}{t}+\omega\right)^{\alpha}, e^{\omega \beta_{2}t}\norm{g}_{L^{\psi}(0,t;L^{\rho})}^{\eta}\right)^{\frac{m}{\theta}}\left(1+N(t)^{\gamma m}\right)\times\\
    	&\quad\left(e^{\omega t}\norm{u_0}_1+\int_0^t e^{\omega(t-\tau)}\norm{g(\tau)}_1 \dtau\right)^{\frac{\gamma m}{(m+1)\theta}+1}
	\end{split}
	\end{displaymath}
	where variables are given by~\eqref{eq:thm12MFs} and \eqref{eq:2exponent} with $\ell = 1$. 
    Note that we require~\eqref{eq:651} to satisfy condition~\eqref{eq:45} of Theorem~\ref{thm:LellLinfty}. Then for $k=\frac{\alpha m}{\theta}$ we define $\tilde{\alpha} = \frac{\alpha m}{\theta}+2$, we have~\eqref{eq:W12estimate}.
	
	Currently we only have this estimate for $\varphi(u_{0}) \in \tilde{D}(A)$. To prove that this holds for all $u_{0} \in L^{1}$, fix $u_{0} \in L^{1}$ and consider a sequence $(u_{0,n})_{n\in\mathbb{N}}$ where $u_{0,n} \in L^{\infty}$ for all $n \in \mathbb{N}$ and $u_{0,n} \rightarrow u_{0}$ in $L^{1}$ as $n \rightarrow \infty$. By semigroup theory, $u_{0}$ generates a mild solution $u$ satisfying $u(0) = u_{0}$. Define $\beta = \varphi^{-1}$ and $w_{n} = \int_{0}^{\varphi(u_{n})}((\beta)'(r))^{1/2}\dr$ so that $w_{n}'(t) = \sqrt{\varphi'(u_{n}(t))}\frac{\td u_{n}}{\td t}(t)$ for almost every $t \in (0,T)$. As in \cite{BenilanGariepy}, we have that $\varphi(u_{n})$ is bounded in $L^{\infty}(0,T;L^{\infty})$ and $w$ is bounded in $W^{1,2}(0,T;L^{2})$. Moreover, taking a subsequence $(n_{k})_{k \ge 1}$ with $n_{k} \rightarrow \infty$ for weak convergences and relabelling, we have
	\begin{align*}
	w_{k}& \rightarrow w \quad \text{in } C([0,T];L^{2}),\\
	w_{k}'& \rightharpoonup w' \quad \text{in } L^{2}((0,T);L^{2}),\text{ and}\\
	v_{k}& \rightharpoonup v \quad \text{in } L^{\infty}((0,T)\times\Omega).
	\end{align*}
	Then by a standard localisation argument and the continuity of $\sqrt{\varphi'}$, we can apply lower semicontinuity of the $L^2$ norm and $\E$ to obtain~\eqref{eq:W12estimate} for $u_0 \in L^1$.
\end{proof}

%
%
%
%

\section{H\"older regularity}
\label{sec:holderRegularity}

This section is dedicated to the parabolic H\"older regularity of mild
solutions to the initial boundary value problem
\begin{equation}
\label{eq:6.1}
\begin{cases}
\begin{alignedat}{2}
u_{t}(t)+(-\Delta_{p})^{s}u^m(t)+f(\cdot,u(t))& = g(\cdot,t)\quad && \text{ in } \Omega\times(0,T),\\
u(t)& = 0\quad && \text{ in }
\mathbb{R}^{d}\setminus\Omega\times(0,T),\\
u(0)& = u_{0}\quad && \text{ on } \Omega,
\end{alignedat}
\end{cases}
\end{equation}
for $1<p<\infty$ with $p\neq 1+\frac{1}{m}$, $0<s<1$ and where $\Omega$ is an open, bounded domain in $\R^d$, $d \ge 2$. For global H\"older regularity we consider only the case $m = 1$. For the local H\"older result, we apply the following local elliptic H\"older regularity result from~\cite{MR3861716}.
\begin{theorem}[{\cite[Theorem~1.4]{MR3861716}}]
	\label{thm:localEllipticHolder}
	Let $\Omega \subset \R^d$ be a bounded and open set. Assume $2 \le p < \infty$, $0 < s < 1$ and $q \ge 1$ satisfies $q > \frac{d}{sp}$. We define the exponent
	\begin{equation}
	\label{eq:localTheta}
	\Theta(d,s,p,q):= \min\left(\frac{1}{p-1}\left(sp-\frac{d}{q}\right),1\right).
	\end{equation}
	Let $u \in W_{\loc}^{s,p}(\Omega)\cap L_{\loc}^{\infty}(\Omega)\cap L_{sp}^{p-1}(\R^d)$ be a local weak solution of
	\begin{displaymath}
	(-\Delta_p)^s u = h \quad \text{in } \Omega,
	\end{displaymath}
	where $h \in L^q_{\loc}(\Omega)$. Then $u \in C_{\loc}^{\delta}(\Omega)$ for every $0 < \delta < \Theta$.
\end{theorem}

\begin{proof}[Proof of Theorem~\ref{thm:localHolder}]
	By the standard growth estimate~\eqref{eq:theorem1.1estimate1}, $u(t) \in L^1$ for $t \in [0,T]$ and hence by the Lipschitz condition of $f$, $F(u(t)) \in L^1$ for all $t \in [0,T]$. 
    By the homogeneous regularizing effects of~\cite{BenilanRegularizingMR648452} as in Theorem~\ref{thm:derivativeEstimates}, we have that $u_t \in L^1$ and $u(t)$ is a strong solution. 
    Furthermore, $u(t) \in L_{\loc}^{\infty}(\Omega)$ for almost every $t \in (0,T)$ by~\cite[Theorem 3.2]{MR3861716}. 
    We now apply Theorem~\ref{thm:localEllipticHolder} with 
    \begin{displaymath}
        h  = g(t)-F(u(t))-u_t(t) \in L^1.
    \end{displaymath}
    Hence $u^m(t) \in C_{\loc}^{\delta}(\Omega)$ for almost every $t \in (0,T)$. Since $r^m$, $r \in \R$, is H\"older continuous for $m \ge 1$ with exponent $\frac{1}{m}$ we obtain H\"older continuity of $u(t)$ in this case by composition.
\end{proof}

Our proof of global H\"older regularity employs the following elliptic H\"older regularity result for the fractional $p$-Laplacian from \cite{IannizzottoHolderRegularity}.

\begin{theorem}[{\cite[Theorem~1.1]{IannizzottoHolderRegularity}}]
	\label{thm:IannizzottoHolder}
	Let $\Omega$ be a bounded domain in $\R^{d}$, $d\ge 2$, with a boundary $\partial\Omega$ of the class $C^{1,1}$, $p \in (1,\infty)$ and $0<s<1$.
	There exists $\alpha \in (0,s]$ and $C_{\Omega} > 0$ depending on $d,p,s,\Omega$ such that if $h \in L^{\infty}$, then the weak solution $u\in W^{s,p}_{0}$ of
	\begin{equation}\label{eq:5}
	\begin{cases}
    \begin{alignedat}{2}
	   (-\Delta_{p})^{s}u& =h\quad && \text{ in } \Omega,\\
	   u& = 0\quad && \text{ in } \R^{d}\setminus\Omega,
    \end{alignedat}
	\end{cases}
	\end{equation}
	belongs to $C^{\alpha}(\overline{\Omega})$ and satisfies
	\begin{equation}
	\label{eq:ellipticHolderEstimate}
	\norm{u}_{C^{\alpha}(\overline{\Omega})} \le C_{\Omega}\norm{h}_{L^{\infty}}^{\frac{1}{p-1}}
	\end{equation}
\end{theorem}

For this we use a restriction to the set of continuous functions $u : \Omega\to \R$ vanishing on the boundary $\partial\Omega$, which we denote by $C_{0}(\Omega)$. Following the notation in Definition~\ref{def:sub-diff-in-X}, we denote by $\left(\partial\E\right)_{\vert_{C_{0}}}$ the restriction of $\partial\E$ to $C_{0}(\Omega)\times C_{0}(\Omega)$. We first prove accretivity and density results for this operator on $C_0(\Omega)$. The proof of density follows the idea in~\cite[Proposition 5.4]{MR3057169}.

\begin{proposition}[{Density of $D(\left(\partial\E\right)_{\vert_{C_{0}}})$ in $C_0(\Omega)$}]
	\label{prop:C0density}
	Let $\Omega$ be a bounded domain in $\R^d$, $d \ge 2$, with a boundary $\partial\Omega$ of the class $C^{1,1}$, $F$ the Nemytskii operator of $f$ on $C_0(\Omega)$ satisfying~\eqref{hyp:2}-\eqref{hyp:3}, $p \in (1,\infty)$ and $0 < s < 1$. Then $\left(\partial\E\right)_{\vert_{C_{0}}}+F$ is m-completely accretive in $C_0(\Omega)$. Furthermore, if $s < 1-\frac{1}{p}$ then the set $D(\left(\partial\E\right)_{\vert_{C_{0}}})$ is dense in $C_0(\Omega)$.
\end{proposition}

\begin{proof}
	Since $\Omega$ has finite Lebesgue measure, one has that the subdifferential satisfies $\displaystyle{\left(\partial\E\right)_{\vert_{C_{0}}} \subseteq \partial\E}$. Then since $\partial\E+F$ is $\omega$-quasi $m$-completely accretive in $L^{2}$, for every $g\in C_{0}(\Omega)$ and $\lambda>0$, there is a unique $u_{\lambda}\in L^{2}$ satisfying
	\begin{equation}
	\label{eq:7_1}
	(1+\lambda\omega)u_{\lambda}+\lambda\big(\partial\E(u_{\lambda})+F(u_{\lambda})\big)=g
	\qquad\text{in $L^{2}$.}
	\end{equation}
	Moreover, by the complete accretivity condition, $u_\lambda \in L^\infty$. Hence $u_{\lambda}$ is a weak solution of the non-local Poisson problem~\eqref{eq:5} with
	\begin{displaymath}
	\begin{split}
	h& :=- F(u_{\lambda})+\frac{g-(1+\lambda\omega)u_{\lambda}}{\lambda}\\
	& \in L^{\infty},
	\end{split}
	\end{displaymath}
	and so Theorem~\ref{thm:IannizzottoHolder} yields that $u_\lambda \in
	C_{0}^{\alpha}(\Omega)$, satisfying 
	\begin{displaymath}
	(1+\lambda\omega)u_{\lambda}+\lambda\big((-\Delta_{p})^{s}_{\vert_{C_{0}}}u_{\lambda}+F(u_{\lambda})\big)=g
	\qquad\text{in $L^{2}$.}
	\end{displaymath}
	As $g\in C_{0}(\Omega)$ and $\lambda>0$ were arbitrary, 
	we have thereby shown that the shifted operator
	$(-\Delta_{p})^{s}_{\vert_{C_{0}}}u_{\lambda}+F(u_{\lambda})+\omega I_{C_{0}}$ satisfies the
	range condition~\eqref{eq:range-condition}.
	
	To prove the density result, fix $u \in C_c^\infty(\Omega)$. We first prove that  $u+(-\Delta_p)^s u \in L^\infty$ by splitting the domain to deal with local and nonlocal estimates. For $\varepsilon > 0$,
	\begin{displaymath}
	(-\Delta_p)^s u \le \int_{B_\varepsilon (x)}\frac{|u(x)-u(y)|^{p-1}}{|x-y|^{d+sp}}\dy+\int_{\R^d\setminus B_\varepsilon (x)}\frac{|u(x)-u(y)|^{p-1}}{|x-y|^{d+sp}}\dy.
	\end{displaymath}
	Estimating $|u(x)-u(y)|$ by the derivative $\sup_{s\in\Omega}|u'(s)||x-y|$ for the first term,
	\begin{displaymath}
	\begin{split}
	\int_{B_\varepsilon (x)}\frac{|u(x)-u(y)|^{p-1}}{|x-y|^{d+sp}}\dy& \le \int_{B_\varepsilon (x)}\frac{\sup_{s\in\Omega}|u'(s)|^{p-1}|x-y|^{p-1}}{|x-y|^{d+sp}}\dy\\
	& \le C_B \sup_{s\in\Omega}|u'(s)|^{p-1}\int_0^\varepsilon \frac{1}{r^{d+1-(1-s)p}}r^{d-1}\dr
	\end{split}
	\end{displaymath}
	where $C_B$ is a constant for integration over a $d$-dimensional ball. Then we have
	\begin{displaymath}
	\begin{split}
	\int_0^\varepsilon \frac{1}{r^{2-(1-s)p}}\dr& = C\left[r^{(1-s)p-1}\right]_0^\varepsilon
	\end{split}
	\end{displaymath}
	which is bounded for $s < 1-\frac{1}{p}$. For the nonlocal term,
	\begin{displaymath}
	\begin{split}
	\int_{\R^d\setminus B_\varepsilon (x)}\frac{|u(x)-u(y)|^{p-1}}{|x-y|^{d+sp}}\dy& \le C_B\norm{u}_\infty \int_\varepsilon^\infty \frac{1}{r^{d+sp}}r^{d-1}\dr\\
	& = C\norm{u}_\infty \left[r^{-sp}\right]_\infty^\varepsilon
	\end{split}
	\end{displaymath}
	for some $C > 0$. Since $sp > 0$, this is bounded and $(-\Delta_p)^s u \in L^\infty$. We now define $f:=A_1 u = u+(-\Delta_p)^s u \in L^\infty$ and approximate by $f_n \in C_c^\infty(\Omega) \subset C_0(\Omega)$ such that $f_n \rightarrow f$ in $L^{p_s^*}$ as $n \rightarrow \infty$ with $\norm{f_n}_\infty \le \norm{f}_\infty$ and where $\frac{1}{p_s^*}+\frac{1}{p_s} = 1$ and $p_s$ is given by~\eqref{eq:6}.
	
	We now solve $A_1 u_n = f_n \in C_0(\Omega)$ using the $m$-accretivity of $(-\Delta_p)^s_{\vert_{C_{0}}}$, finding that $u_n \in C_0(\Omega)$ and so $u_n \in D((-\Delta_p)^s_{\vert_{C_{0}}})$. Moreover, by the elliptic H\"older estimate~\eqref{eq:ellipticHolderEstimate} with $(f_n)_{n \in \N}$ bounded in $L^\infty$, $(u_n)_{n \in \N}$ is bounded in $C^\alpha(\Omega)$. Hence taking a subsequence and relabelling, we have convergence to a function in $C(\Omega)$.
	
	Next, we prove that this limit is $u$ by considering convergence in $W^{s,p}$. For $w \in D((-\Delta_p)^s_{\vert_{C_{0}}})$ we can define $\varphi_w(v):=(A_1 w,v)_{L^2}$ for all $v \in W^{s,p}$. Estimating by H\"older's inequality,
	\begin{displaymath}
	\begin{split}
	|\varphi_{u_n}(v)-\varphi_u(v)|& \le \int_\Omega |(f_n-f)v|\dmu\\
	& \le \norm{f_n-f}_{p_s^*} \norm{v}_{p_s}\\
	& \le C_d\norm{f_n-f}_{p_s^*}\norm{v}_{W^{s,p}}
	\end{split}
	\end{displaymath}
	by the standard Sobolev embedding~\eqref{eq:6}. Hence $\varphi_{u_n} \rightarrow \varphi_u$ in $(W^{s,p})'$ and so $(\varphi_{u_n})_{n\ge 1}$ is bounded in $(W^{s,p})'$. Then we also have that
	\begin{displaymath}
	\begin{split}
	(A_1 u_n,u_n)_{L^2}& = \norm{u_n}_2^2+[u_n]_{s,p}^p\\
	& \le C_u\left(1 + \norm{u_n}_{W^{s,p}}\right).
	\end{split}
	\end{displaymath}
	Now
	\begin{displaymath}
	\begin{split}
	\liminf_{\norm{v}_{W^{s,p}} \rightarrow \infty}\frac{(A_1 v,v)_{L^2}}{\norm{v}_{W^{s,p}}}& = \frac{\norm{v}_2+[v]_{s,p}^p}{\norm{u}_{W^{s,p}}}\\
	& \ge \norm{v}_{W^{s,p}}^{p-1}\\
	& \rightarrow \infty
	\end{split}
	\end{displaymath}
	for $p \in (1,\infty)$. Hence $(u_n)_{n\ge 1}$ is bounded in $W^{s,p}$ and passing to a subsequence we have that $u_n \rightharpoonup \tilde{u}$ in $W^{s,p}$. We extend $\varphi_w(v)$ to $w \in W^{s,p}$ by defining $\varphi_w(v) = \left(w+(-\Delta_p)^s w, v\right)_{L^2}$ for $w \in W^{s,p}\setminus D((-\Delta_p)^s_{\vert_{C_{0}}})$. By Minty's theorem,\\ \cite[Proposition~II.2.2]{MR1422252}, 
	\begin{displaymath}
	\braket{\varphi_v-\varphi_{u_n},v-u_n} \ge 0
	\end{displaymath}
	for all $v \in W^{s,p}$. We have the convergences $\varphi_{u_n} \rightarrow \varphi_u$ in $(W^{s,p})'$ and $u_n \rightharpoonup \tilde{u}$ in $W^{s,p}$ so that
	\begin{displaymath}
	\braket{\varphi_v-\varphi_{\tilde{u}},v-u} \ge 0
	\end{displaymath}
	for all $v \in W^{s,p}$. Applying Minty's theorem again, $\varphi_{\tilde{u}} = \varphi_u$ so that by the uniqueness provided by the accretivity of $(-\Delta_p)^s$ in $L^2$, $\tilde{u} = u$. Also, since $u_n \rightharpoonup u$ in $W^{s,p}$, we have that $u_{n} \rightarrow u$ in $C_0(\Omega)$, relabelling by appropriate subsequences. Then the density of $C_c^\infty(\Omega)$ in $C_0(\Omega)$ gives us the desired density result.
\end{proof}

We now prove H\"older continuity in the identity case, $m = 1$. In this theorem, the case $p=2$ is well-known. We note that this proof of parabolic regularity relies on the global H\"older regularity estimate of the elliptic problem which we believe is not optimal, in particular with the $L^\infty$ norm required in~\eqref{eq:ellipticHolderEstimate}, and that a stronger elliptic result would also improve this parabolic result.

\begin{proof}[Proof of Theorem~\ref{thm:globalHolder}]
	By Proposition~\ref{prop:C0density}, we have that $\left(\partial\E\right)_{\vert_{C_{0}}}+F$ is $\omega$-quasi m-completely accretivity in $C_0(\Omega)$ and $\overline{D(\left(\partial\E\right)_{\vert_{C_{0}}})}^{\mbox{}_{C_0}} = C_0(\Omega)$. The Crandall-Liggett theorem (see~\cite{MR0287357},	\cite{MR0331133}) says that $-(\partial_{C_{0}}\E+F)$ generates a strong continuous semigroup of $\omega$-quasi	contractions on $C_{0}(\Omega)$.  Further, since $\left(\partial\E\right)_{\vert_{C_{0}}}$ is homogeneous of order $p-1$, and since
	\begin{displaymath}
	\left(\partial\E\right)_{\vert_{C_{0}}}\subseteq \partial_{L^{q}}\E
	\end{displaymath}
	for all $1\le q\le \infty$, it follows
	from \cite{MR4200826} that for initial data $u_{0}\in C_{0}(\Omega)$ and $g\in C((0,T);C_{0}(\Omega))\cap BV(0,T;C_{0}(\Omega))$, the mild solution $u\in C([0,T];C_{0}(\Omega))$ of the initial boundary value problem~\eqref{eq:6.1} is a strong solution with $u\in W^{1,\infty}(\delta,T;C_{0}(\Omega))$. Moreover, $u \in C^{lip}([\delta,T];C_0(\Omega))$ for every $0<\delta<T$. In particular, $u$ is a weak solution of the non-local Poisson problem~\eqref{eq:5} with
	\begin{displaymath}
	h:=g(t)-F(u)-u_{t}(t)\in C((0,T);C_{0}(\Omega)).
	\end{displaymath}
	Hence, by the elliptic
	regularity result Theorem~\ref{thm:IannizzottoHolder}, we obtain
	that $u(t) \in C^{\alpha}(\overline{\Omega})$ for all $t \in (0,T)$ for some $\alpha \in (0,s]$.
\end{proof}

%
%
%
%

\section{Finite time of extinction}
\label{sec:comparisonPrinciple}

In the porous medium case, $\varphi(u) = u^{m}$ with $0 < m < 1$, we can obtain extinction in finite time following the method presented for the fractional Laplacian case in \cite{KLkiahmholderregularitypaper}. We first prove a comparison principle for the doubly nonlinear problem~\eqref{eq:1}, extending~\cite[Theorem 4.1]{BenilanGariepy} to inhomogeneous boundary data on $\mathbb{R}^{d}$. By constructing an explicit supersolution and subsolution on $\R^N$ we can then prove extinction in finite time.

\subsection{A parabolic comparison principle}
\label{sec:parabolicComparisonPrinciple}
We first introduce the inhomog\-eneous fractional Sobolev space for $\Omega$ an open domain in $\R^d$, $d \ge 1$, and $b \in L_{\loc}^1(\R^d\setminus\Omega)$,
\begin{displaymath}
W^{s,p}_{b}(\Omega)=\Big\{u\in
W^{s,p}(\R^{d})\,\vert\,u=b\text{ a.e.~on }\R^{d}\setminus\Omega\Big\}
\end{displaymath}
and the fractional $p$-Laplacian for $u \in W^{s,p}_{b}(\Omega)$. I this setting we have the energy functional $\E: L^2(\Omega) \rightarrow (-\infty,\infty]$ defined by
\begin{displaymath}
\E(u) = \begin{cases}
\frac{1}{2p}[u]_{s,p}^{p}& \text{ if } u \in W_{b}^{s,p}(\Omega)\cap L^2(\Omega),\\
+\infty& \text{ otherwise,}
\end{cases}
\end{displaymath}
so that the fractional p-Laplacian is given by the sub-differential operator of $\E$ in $L^2$. 
In particular, we use the variational equation,
\begin{equation}
\label{eq:variationalInhomogeneous}
    \frac{1}{2}\int_{\mathbb{R}^{2d}}\frac{(u(x)-u(y))^{p-1}(v(x)-v(y))}{|x-y|^{d+sp}}\dy\dx = \int_{\Omega}h(x)v(x)\dx
\end{equation}
so that we have the characterization,
\begin{displaymath}
\partial\E(u) = \left\{
h \in L^{2}: \text{$u$, $h$  satisfy \eqref{eq:variationalInhomogeneous} for all } v \in L^{2} \text{ with } [v]_{s,p} < \infty
\right\}.
\end{displaymath}
For every $u \in W_{b}^{s,p}(\Omega)\cap L^2$ this is then unique, so we write $(-\Delta_{p})^{s}u = \partial\E(u)$.

In this setting we consider the following inhomogeneous evolution equation,
\begin{equation}
\label{eq:fractionalPLaplacianOnRN}
\begin{cases}
\begin{alignedat}{2}
    u_{t}+(-\Delta_{p})^{s}\varphi(u)+f(\cdot,u)& = g \quad && \text{ on } \Omega\times[0,T],\\
    u(t)& =h(t)\quad && \text{ on } \mathbb{R}^{d}\setminus\Omega\times[0,T],\\
    u(0)& = u_{0}\quad && \text{ on } \Omega.
\end{alignedat}
\end{cases}
\end{equation}
In particular, we have the comparison principle.

\begin{theorem}[Comparison principle for inhomogeneous boundary data]
    \label{thm:fractionalPLaplacianComparisonPrinciple}
	Let $\Omega$ be an open domain in $\R^d$, $d \ge 1$, $T > 0$, $f$ satisfy~\eqref{hyp:2}-\eqref{hyp:3} and $\varphi:\R \rightarrow \R$ be strictly increasing and satisfy $\varphi(0) = 0$.
    Suppose $u$ and $\hat{u}\in W^{1,1}((0,T);L^{1})$ are two strong distributional solutions in $L^1$ to the inhomogeneous Dirichlet problem~\eqref{eq:fractionalPLaplacianOnRN} with initial data $u_0$, $\hat{u}_0 \in L^1$, forcing terms $g$, $\hat{g} \in L^{1}((0,T);L^{1})$ and boundary data $h$, $\hat{h} \in L_{\loc}^1(\R^d\setminus\Omega)$, respectively.
    If $h(t) \le \hat{h}(t)$ a.e.~on $\mathbb{R}^{d}$ for a.e.~$t \in (0,T)$, then
	\begin{equation}
	\label{eq:UniquenessEstimateRN}
	\begin{split}
	\int_{\Omega}(u(t)-\hat{u}(t))^{+}\dmu& \le e^{\omega t}\int_{\Omega}(u_{0}-\hat{u}_{0})^{+}\dmu\\
	&\quad+\int_{0}^{t}e^{\omega(t-s)}\int_{\Omega}(g(s)-\hat{g}(s))\mathds{1}_{\set{u>\hat{u}}}\dmu \ds
	\end{split}
	\end{equation}
	for all $0 \le t \le T$.
\end{theorem}

\begin{proof}
    Since $u$ and $\hat{u}$ are distributional solutions, we have that $\varphi(u(t)) \in W^{s,p}_{\varphi(h(t))}(\Omega)$ and $\varphi(\hat{u}(t)) \in W^{s,p}_{\varphi(\hat{h}(t))}(\Omega)$ for a.e.~$t \in (0,T)$.
	We first prove an estimate on the sign of $(-\Delta_{p})^{s}u-(-\Delta_{p})^{s}\hat{u}$ given that $u \le \hat{u}$ on $\R^d\setminus \Omega$. 
	We approximate the sign function by considering all $q \in C^{1}(\mathbb{R})$ satisfying $0 \le q \le 1$, $q(s) = 0$ for $s \le 0$ and $q'(s) > 0$ for $s > 0$ and prove that
	\begin{displaymath}
	\int_{\Omega}((-\Delta_{p})^{s}u-(-\Delta_{p})^{s}\hat{u})q(u-\hat{u})\dx \ge 0.
	\end{displaymath}
	By assumption, $q(u-\hat{u}) = 0$ on $\mathbb{R}^{d}\setminus \Omega$. So we have
	\begin{displaymath}
	\begin{split}
	\int_{\Omega}((-\Delta_{p})^{s}&u-(-\Delta_{p})^{s}\hat{u})q(u-\hat{u})\dx\\	&=\int_{\mathbb{R}^{d}}\int_{\mathbb{R}^{d}}\frac{(u(x)-u(y))^{p-1}-(\hat{u}(x)-\hat{u}(y))^{p-1}}{|x-y|^{d+ps}}\times\\
	&\quad\left(q(u(x)-\hat{u}(x))-q(u(y)-\hat{u}(y))\right)\dy\dx.
	\end{split}
	\end{displaymath}
	We split these integrals into $\set{u\ge \hat{u}}$ and $\set{u<\hat{u}}$ terms noting that $q(u(x)-\hat{u}(x)) = 0$ on $\set{u<\hat{u}}$. On $\set{u\ge\hat{u}}\times\set{u\ge\hat{u}}$, using the monotonicity of $q(s)$, we have
	\begin{displaymath}
	\begin{split}
	&\iint_{\set{u\ge\hat{u}}^{2}}\frac{(u(x)-u(y))^{p-1}-(\hat{u}(x)-\hat{u}(y))^{p-1}}{|x-y|^{d+ps}}\times\\
	&\hspace{2cm}\left(q(u(x)-\hat{u}(x))-q(u(y)-\hat{u}(y))\right)\dy\dx \ge 0.
	\end{split}
	\end{displaymath}
	On $\set{u<\hat{u}}\times\set{u<\hat{u}}$ we have
	\begin{displaymath}
	q(u(x)-\hat{u}(x))-q(u(y)-\hat{u}(y)) = 0.
	\end{displaymath}
	Then applying the symmetry of $x$ and $y$ to the remaining two terms, we have that
	\begin{align*}
	\int_{\Omega}((-\Delta_{p})^{s}&u-(-\Delta_{p})^{s}\hat{u})q(u-\hat{u})\dx\\
	& \ge 2\int_{\set{u\ge\hat{u}}}\int_{\set{u<\hat{u}}}\frac{(u(x)-u(y))^{p-1}-(\hat{u}(x)-\hat{u}(y))^{p-1}}{|x-y|^{d+ps}}\times\\
	&\quad q(u(x)-\hat{u}(x))\dy\dx\\
	& \ge 2\int_{\set{u\ge\hat{u}}}\int_{\set{u<\hat{u}}}\frac{(\hat{u}(x)-\hat{u}(y))^{p-1}-(\hat{u}(x)-\hat{u}(y))^{p-1}}{|x-y|^{d+ps}}\times\\
	&\quad q(u(x)-\hat{u}(x))\dy\dx\\
	& = 0.
	\end{align*}
	Letting $q$ converge to $[\sign_0]^{+}$ and noting that $\varphi(u) > \varphi(\hat{u})$ if and only if $u > \hat{u}$,
	\begin{align}
	\label{eq:FractionalPLaplacianSignEstimate}
	\int_{\Omega}((-\Delta_{p})^{s}\varphi(u)-(-\Delta_{p})^{s}\varphi(\hat{u}))\mathds{1}_{\set{u>\hat{u}}}\dmu& \ge 0.
	\end{align}
	By the chain rule,~\eqref{eq:fractionalPLaplacianOnRN} and~\eqref{eq:FractionalPLaplacianSignEstimate},
	\begin{align*}
	\frac{\td }{\td t}&\int_{\Omega}(u(t)-\hat{u}(t))^{+}\dmu = \int_{\Omega}(u'(t)-\hat{u}'(t))\mathds{1}_{\set{u>\hat{u}}}\dmu\\
	& = -\int_{\Omega}\left((-\Delta_{p})^{s}\varphi(u)(t)-(-\Delta_{p})^{s}\varphi(\hat{u})(t)\right)\mathds{1}_{\set{u>\hat{u}}}\dmu\\
	&\quad+\int_{\Omega}(F(u)-F(\hat{u}))\mathds{1}_{\set{u>\hat{u}}}\dmu+\int_{\Omega}(g(t)-\hat{g}(t))\mathds{1}_{\set{u>\hat{u}}}\dmu\\
	& \le \omega\int_{\Omega}(u(t)-\hat{u}(t))^{+}\dmu+ \int_{\Omega}(g(t)-\hat{g}(t))\mathds{1}_{\set{u>\hat{u}}}\dmu.
	\end{align*}
	Applying a Gr\"onwall inequality,
	\begin{displaymath}
	\begin{split}
	\int_{\Omega}(u(t)-\hat{u}(t))^{+}\dmu& \le e^{\omega t}\int_{\Omega}(u_{0}-\hat{u}_{0})^{+}\dmu\\
	&\quad+\int_{0}^{t}e^{\omega(t-s)}\int_{\Omega}(g(s)-\hat{g}(s))\mathds{1}_{\set{u>\hat{u}}}\dmu \ds.
	\end{split}
	\end{displaymath}
\end{proof}

\subsection{Proof of finite time of extinction}
We now suppose that $\Omega$ is bounded in order to construct a super-solution and a sub-solution which are truncated within a ball containing $\Omega$.

\begin{proof}[Proof of Theorem~\ref{thm:extinction}]
	Let $u$ be a strong distributional solution to~\eqref{eq:1}. We construct a super-solution and sub-solution on $\mathbb{R}^{d}\times \mathbb{R}_{+}$ to bound $u$ via separation of variables of the form $\mu(x)T(t)$ using the fundamental solution. In particular, we will choose $T$ to be a decreasing function such that $T(t^{*}) = 0$ for some $t^* > 0$.
 
    Choose $R > 0$ such that $\Omega \subset B_R(0)$. 
    In order to apply Theorem \ref{thm:fractionalPLaplacianComparisonPrinciple} we require that this super-solution $V(x,t)$ satisfies
	\begin{equation}
	\label{eq:supersolutionInequalityEqn}
	V_{t}+(-\Delta_{p})^{s}\varphi(V) \ge 0.
	\end{equation}
	Letting $\beta = \varphi^{-1}$ we set $V(x,t) = \beta(W(x,t))$ with $W(x,t) = \mu(x) T(t)$, where we choose
	\begin{displaymath}
	\mu(x) = \begin{cases}
	R^{-d-ps} & \text{ for } |x| \le R,\\
	|x|^{-d-ps} & \text{ for } |x| \in (R,3R),\\
	0 & \text{ for } |x| \ge 3R.
	\end{cases}
	\end{displaymath}
    Defining
    \begin{equation}
	\label{eq:tildeCforTstar}
	C_R: = \frac{\omega_{d-1}}{4^{d+ps}d}(3^d-2^d)\left(1-2^{-d-ps}\right)^{p-1}R^{-sp}
	\end{equation}
    with $\omega_{d}$ denoting the volume of a $d$-dimensional unit ball,
    and
    \begin{displaymath}
	t^* = \frac{1}{C_R}\int_{0}^{\norm{u_0}_\infty}\frac{1}{\left(\varphi(\tau)\right)^{p-1}}\td\tau,
	\end{displaymath}
    we set
    \begin{displaymath}
    T(t) = \begin{cases}
        R^{d+ps}\phi(\sigma(t^{*}-t))& \text{if $t < t^{*}$}\\
        0& \text{if $t \ge t^{*}$}
    \end{cases}
    \end{displaymath}
    where $\sigma(t)$ satisfies
    \begin{displaymath}
	\int_{0}^{\sigma(t)}\frac{1}{\left(\varphi(\tau)\right)^{p-1}}\td\tau = C_R t
	\end{displaymath}
    for $t \in [0,t^*]$.
    Note that $V(x,0) = \norm{u_0}_{\infty}$ and $V(x,t^{*}) = 0$ for $x \in \Omega$. Moreover, $V_t = -C_R\left(\phi(\sigma(t^{*}-t))\right)^{p-1}$.
    
    For $x \in \Omega$, we have
	\begin{align*}
	(-\Delta_{p})^{s}\mu(x)& \le \int_{\mathbb{R}^{d}\setminus B_{R}(0)}\frac{R^{-(d+ps)(p-1)}}{|x-y|^{d+ps}}\dy
	\end{align*}
	which is bounded since $ps > 0$ so that $(-\Delta_{p})^{s}\mu(x) \in L^{\infty}$ and so we can apply the singular integral form of the fractional p-Laplacian.
    Let $g = V_{t}+(-\Delta_{p})^{s}\varphi(V)$ on $\Omega\times(0,\infty)$. Note that $V \in W_{\loc}^{1,1}(0,\infty;L^1(\R^d))$.
    Then $V$ is a strong distributional solution to
	\begin{equation}
	\label{eq:fractionalPLaplacianForComparison}
	\begin{cases}
    \begin{alignedat}{2}
    	v_{t}+(-\Delta_{p})^{s}\varphi(v)& = g\quad &&\text{ on } \Omega\times(0,\infty),\\
    	v(t)& = V_{t}(t) \quad &&\text{ on }\mathbb{R}^{d}\setminus\Omega\times(0,\infty),\\
    	v(0)& = v_{0} \quad &&\text{ in } \mathbb{R}^{d},
    \end{alignedat}
	\end{cases}
	\end{equation}
    where $V(t) \ge u(t)$ on $\mathbb{R}^{d}\setminus \Omega$ and $V(0) \ge u_{0}$.
	Applying Theorem \ref{thm:fractionalPLaplacianComparisonPrinciple}, we have that
    \begin{displaymath}
    	\int_{\Omega}(u(t)-V(t))^{+}\dmu \le -\int_{0}^{t}\int_{\Omega}g(s)\mathds{1}_{\set{u>V}}\dmu \ds.
	\end{displaymath}
    To obtain that $u(t) \le V(t)$ almost everywhere on $\Omega\times(0,\infty)$, it therefore remains to prove that the right hand side is bounded by zero.

    For $x \in \Omega$, using the singular integral formulation,
	\begin{align*}
	(-\Delta_{p})^{s}\mu(x)& \ge \int_{B_{3R}(0)\setminus B_{2R}(0)}\frac{\left(R^{-d-ps}-|y|^{-d-ps}\right)^{p-1}}{|x-y|^{d+ps}}\dy\\
	& \ge \left(R^{-d-ps}(1-2^{-d-ps})\right)^{p-1}\int_{B_{3R}(0)\setminus B_{2R}(0)}\frac{1}{(4R)^{d+sp}}\dy\\
    & \ge \frac{\omega_{d-1}}{4^{d+ps}d}(3^d-2^d)R^{-sp}\left(R^{-d-ps}(1-2^{-d-ps})\right)^{p-1}.
	\end{align*} 

    Rewriting~\eqref{eq:supersolutionInequalityEqn} in terms of the separated variables $\mu(x)$ and $T(t)$,
    and applying the estimate on $(-\Delta_p)^s \mu$, it is sufficient for $T(t)$ to satisfy
	\begin{displaymath}
	\frac{\td }{\td t}\beta(R^{-d-ps} T(t))+C_R \left(R^{-d-ps}T(t)\right)^{p-1} \ge 0 \quad \text{ on } (0,t^{*}).
	\end{displaymath}
    In particular, we require that
	\begin{displaymath}
	\frac{\td \sigma}{\td t}= C_R\left(\varphi(\sigma)\right)^{p-1}\quad \text{ on } (0,t^{*}),
	\end{displaymath}
    which holds by definition of $\sigma$.
	Hence
	\begin{displaymath}
	V_{t}+(-\Delta_{p})^{s}\varphi(V) \ge 0
	\end{displaymath}
	for all $x \in \Omega$.	
 
    Similarly for $-u$ and $-V$ we have, with respect to the Lebesgue measure,
	\begin{displaymath}
	\begin{cases}
    \begin{alignedat}{2}
	(-V)_{t}+(-\Delta_{p})^{s}\varphi(-V)& \le 0\quad && \text{ on } \Omega\times(0,\infty),\\
	-V(0,\cdot)& \le u_{0}\quad && \text{ on } \mathbb{R}^{d},\\
	-V& \le u\quad && \text{ on } \mathbb{R}^{d}\setminus\Omega\times(0,\infty),
    \end{alignedat}
	\end{cases}
	\end{displaymath}
	hence we have that for almost every $x \in \Omega$ and all $t \ge 0$, $-V(t) \le u(t) \le V(t)$ and so $u(t) = 0$ for $t \ge t^{*}$.
\end{proof}

We extend this to mild solutions in Corollary~\ref{thm:extinctionrm} by approximation, applying Theorem~\ref{thm:1} and in particular, the growth estimate~\eqref{eq:theorem1.1estimate2}.

\section{Acknowledgements}

Daniel Hauer is supported by the Australian Research Council grant ARC Discovery Project No. DP200101065.

%
%


\end{document}